\documentclass[a4paper,10pt, reqno]{amsart}

\usepackage{hyperref}
\hypersetup{
    colorlinks=true,       
    linkcolor=blue,          
    citecolor=blue,        
    filecolor=blue,      
    urlcolor=blue           
}
\numberwithin{equation}{section}
\usepackage{rotating}
\usepackage[all]{xy}
\usepackage{tikz,graphicx}
\usepackage{amssymb}
\usepackage{booktabs}
\usepackage[english]{babel}
\usepackage{multirow}
\usepackage{multicol}
\usepackage{hyperref}
\usepackage{rotating}

\theoremstyle{plain}
\newtheorem{thm}{Theorem}[section]

\newtheorem{prop}[thm]{Proposition}
\newtheorem{cor}[thm]{Corollary}
\newtheorem{rem}[thm]{Remark}
\newtheorem{preg}{Question}
\newtheorem{claim}{Claim}
\theoremstyle{definition}
\newtheorem{defn}[thm]{Definition}

\newcommand{\s}{\vspace{0.3cm}}

\newcommand{\Z}{\mathbb{Z}}

\def\Aut{\operatorname{Aut}}
\def\PGL{\operatorname{PGL}}

\usepackage{enumerate}

\begin{document}
\title[Automorphism groups of smooth plane sextic curves]{The stratification by automorphism groups of smooth plane sextic curves}
\author[E. Badr] {Eslam Badr}
\address{$\bullet$\,\,Eslam Badr}
\address{Mathematics Department,
Faculty of Science, Cairo University, Giza-Egypt}
\email{eslam@sci.cu.edu.eg}
\address{Mathematics and Actuarial Science Department (MACT), American University in Cairo (AUC), New Cairo-Egypt}
\email{eslammath@aucegypt.edu}

\author[F. Bars] {Francesc Bars}
\address{$\bullet$\,\,Francesc Bars}
\address{Departament Matem\`atiques, Edif. C, Universitat Aut\`onoma de Barcelona\\
08193 Bellaterra, Catalonia}
\address{Barcelona Graduate School of Mathematics,
Catalonia} \email{francesc@mat.uab.cat}
\thanks{The authors are supported by PID2020-116542GB-I00, Ministerio de Ciencia y Universidades of Spanish
government. Also, E. Badr is partially supported by the School of Science and
Engineering (SSE) of the American Univeristy in Cairo (AUC)}

%



\maketitle
\begin{abstract}
We obtain the list of automorphism groups for smooth plane sextic curves over an algebraically closed field $K$ of characteristic $p=0$ or $p>21$. Moreover, we assign to each group a \emph{geometrically complete family over $K$} describing its corresponding stratum, that is, a generic defining polynomial equation with parameters such that any curve in the stratum is $K$-isomorphic to a non-singular plane model obtained by specializing the values of those parameters over $K$.
\end{abstract}

\section{Introduction}
Smooth plane curves of degree $d\geq 4$ with non-trivial
automorphism groups are of special interest in different branches of
mathematics. For example, in the algebraic geometry of the \emph{Cremona group} $\operatorname{Bir}(\mathbb{P}^2(\mathbb{C}))$ of the projective complex plane $\mathbb{P}^2(\mathbb{C})$. In particular, to understand the dynamics of its elements as we can consider a Cremona transformation as defining a dynamical system. More concretely, in the study of the finite subgroups of $\operatorname{Bir}(\mathbb{P}^2(\mathbb{C}))$, the set of birational classes of curves of a fixed genus is an important conjugacy invariant that was used to find infinitely many conjugacy classes of
elements of order $2n$ of $\operatorname{Bir}(\mathbb{P}(\mathbb{C}))$ for any integer $n$. See \cite{BlSt, BlaPanVus} for more details.

Concerning smooth plane sextics, there are many reasons to deal with them. For instance, they naturally appear in the theory of
$K3$-surfaces, which in turns has strong connection with
physics. In \cite{Deg1} it was conjectured that the maximal number of lines in a smooth $2$-polarized
$K3$-surface is $144$, with the maximum realized by the double plane $X\rightarrow \mathbb{P}^2$ ramified
over the smooth sextic curve $$X^6+Y^6+Z^6-10(X^3Y^3+Y^3Z^3+X^3Z^3)=0.$$
In \cite{Deg} the conjecture was extended in terms of tritangents to the ramification locus $C\subset\mathbb{P}^2$ (a smooth sextic curve) rather than lines in the surface $X\rightarrow \mathbb{P}^2$. Another reason is that smooth plane sextics are extremely rich from arithmetic geometry point of view. Roughly speaking, they are the first place to construct examples that not every twist $C'$ of a smooth $C$ over a field $k$ is given by smooth plane  model over $k$ even if $C\times\overline{k}$ admits a smooth plane model over $\overline{k}$. We refer the reader to \cite{twists,counter} for the complete details.

The structure of the automorphism groups of smooth curves of genus $g\geq2$ defined over an algebraically closed field $K$ of characteristic $p=0$ or $p>2g+1$ is an old subject of research in algebraic geometry. The most famous universal bound, the \emph{Hurwitz bound}, given by Hurwitz \cite{MR1510753} turns out to be sharp for infinitely many genera. Oikawa \cite{MR0080730} and Arakawa \cite{MR1809907} gave better upper bounds when the automorphism group fixes (not necessarily pointwise) finite subsets of points on the curve. These bounds become very useful in our study of smooth plane curves.

In the case of hyperelliptic curve, the structure of the automorphism group is quite explicit (see \cite{MR897252, MR1223022, MR2280308, MR2035219}). For non-hyperelliptic curves, we still have a lack of knowledge about the structure, except for some special cases. For example, the cases of low genus and also Hurwitz curves, see \cite{MR1796706, henn1976automorphismengruppen, MR839811, MR1072285, MR1068416}. This lack motivates us to do more investigation in this direction, especially for the case of smooth plane curves of degree $d\geq4$.

Now, given a smooth plane curve $C$ of degree $d\geq4$ over $K$, the linear series $g^2_d$ is unique up to $\operatorname{PGL}_3(K)$-conjugation (cf. \cite[Lemma 11.28]{Book}), where $g=\frac{1}{2}(d-1)(d-2)$ is the genus $g$ of $C$. In particular, any isomorphism between two smooth plane curves is linear (cf. \cite{MR529972}). Therefore, we can think about $\operatorname{Aut}(C)$, the automorphism group of $C$, as a finite subgroup of $\operatorname{PGL}_3(K)$. Next, we can associate to $C$ infinitely many smooth plane models of the form $F_C(X,Y,Z)=0$ in $\mathbb{P}^2_{K}$ of degree $d$. Any two such models are $K$-isomorphic through a change of variables $\phi\in\operatorname{PGL}_3(K)$, and their automorphism groups are $\operatorname{PGL}_3(K)$-conjugated.

Next, for a finite non-trivial group $G$, consider the stratum $\mathcal{M}^{\operatorname{Pl}}_{g}(G)$ consisting of the $K$-isomorphism classes of smooth plane curves $C$ of genus $g=\frac{1}{2}(d-1)(d-2)$ such that $\Aut(C)$ contains a subgroup
isomorphic to $G$. Similarly, we write $\widetilde{\mathcal{M}^{\operatorname{Pl}}_{g}}(G)$ when $\Aut(C)$ is itself isomorphic to $G$, in
particular, $\widetilde{\mathcal{M}^{\operatorname{Pl}}_{g}}(G)\subseteq\mathcal{M}^{\operatorname{Pl}}_{g}(G)$.

Concerning these strata, one asks the following two questions:
\begin{preg}\label{1stquestion}
Let $G$ be a finite non-trivial group. What are the values of $d$ such that the corresponding stratum $\widetilde{\mathcal{M}_g^{Pl}}(G)\neq\emptyset$, that is, $\exists$ a smooth plane curve $C$ of degree $d$ over $K$ whose $\operatorname{Aut}(C)$ isomorphic to $G$?
\end{preg}
So far, we do not have a complete answer to the above question except for very special cases. For example, by the work of S. Crass in \cite[p.28]{MR1724156}, we know that $\widetilde{\mathcal{M}_g^{Pl}}(\operatorname{A}_6)\neq\emptyset$ exactly when $d=6$, $d=12$ and $d=30$, where $\operatorname{A}_6$ is the alternating group on six letters. The recent work of Y. Yoshida in \cite{Yoshida1} provides similar results when $G$ is the groups alternating group $\operatorname{A}_5$ or the Klein group $\operatorname{PSL}(2,7)$.

\vspace{0.25cm}
In the line of the works of P. Henn in \cite{henn1976automorphismengruppen} and Komiya-Kuribayashi in \cite{MR555703} for degree 4 curves, and Badr-Bars \cite{MR3508302} for degree $5$ curves, one asks the following question:

\begin{preg}\label{2ndquestion}
Fix an integer $g=\dfrac{1}{2}(d-1)(d-2)\geq3$. How looks like the stratification by automorphism groups of the $K$-isomorphism classes of smooth plane curve $C$ of degree $d$ over $K$? Equivalently, determine the list of finite groups $G$ for which $\widetilde{\mathcal{M}_g^{Pl}}(G)\neq\emptyset$.
\end{preg}

In this paper, we aim to answer Question \ref{2ndquestion} for degree $6$ curves, which are genus $10$ curves. Besides ad-hoc computations to prove or disprove the existence of some cases, the two main ingredients are the works of Badr-Bars \cite{MR3475065} and of T. Harui \cite{Harui}. We also remark that Doi-Idei-Kaneta in \cite{DoiIdei} showed that the maximum order of the automorphism group of smooth plane sextics is $360$, moreover, they proved that the most symmetric smooth plane sexitc is $K$-isomorphic to the Wiman sextic curve
$$
\mathcal{W}_6\,:\,27X^6+9X(Y^5+Z^5)-135X^4YZ-45X^2Y^{2}Z^{2}+10Y^3Z^3=0
$$
whose automorphism group equals the alternating group $\operatorname{A}_6$. They studied the existence of $G$-invariant smooth plane sextics when $G$ is a $p$-group with order dividing $|\operatorname{A}_6|=360$ and $p$ a prime integer. The classification we give in this paper; Theorem \ref{mainresult} is a big generalization of their results. For instance, we prove the existence of infinitely many smooth plane sextic curves that are $\operatorname{Q}_8$-invariant. This violates \cite[Lemma 2.11]{DoiIdei}. The detailed connection of our results with \cite{DoiIdei} is provided just after Theorem \ref{mainresult}; see Corollary \ref{correctDIK}.


\section{Statement of the main result}
\noindent\textbf{Notations.} Throughout the paper, $L_{i,B}$ denotes the generic homogeneous polynomial of degree $i$ in the variables $\{X,Y,Z\}-\{B\}$. Also, a projective linear transformation $A=(a_{i,j})\in\PGL_3(K)$ is sometimes written as
$$[a_{1,1}X+a_{1,2}Y+a_{1,3}Z:a_{2,1}X+a_{2,2}Y+a_{2,3}Z:a_{3,1}X+a_{3,2}Y+a_{3,3}Z].$$
Moreover, we use the formal GAP library notations ``GAP$(n, m)$`` to refer the finite group of order $n$ that appears
in the $m$-th position of the atlas for small finite groups \cite{Gap}. See also \href{https://people.maths.bris.ac.uk/~matyd/GroupNames/index500.html}{GroupNames}.

\vspace{0.25cm}
Regrading Question \ref{2ndquestion} for smooth plane sextic curves, we obtain:
\begin{thm}\label{mainresult}
Let $K$ be an algebraically closed field of characteristic $p=0$ or $p>21$. The following table provides the full list of automorphism groups of smooth plane sextics over $K$, along with a geometrically complete family over $K$ for each stratum. In particular, $\widetilde{\mathcal{M}_g^{Pl}}(G)\neq\emptyset$ if and only if $G$ is one of the groups that appear in Table \ref{table:fullAuto.} below.

We do not give the algebraic restrictions on the parameters such that each family is smooth and does not admit larger automorphism group for sake of simplicity.

\vspace{0.1cm}
\begin{table}[th]
  \renewcommand{\arraystretch}{1.3}
  \caption{Automorphism Groups and Defining Equations}\label{table:fullAuto.}
  \vspace{4mm} 
  \centering
\scriptsize
\begin{tabular}{|c|c|c|c|}
  \hline
\textit{GAP ID} & $\Aut(C)$ &\textit{Generators} & $F(X,Y,Z)$ \\\hline\hline
\multirow{2}{*}{(360,118)} & \multirow{2}{*}{$\operatorname{A}_6$}& $R^{-1}T_iR$ for $i=1,2,3,4$& \multirow{2}{*}{$27X^6+9X(Y^5+Z^5)-135X^4YZ-45X^2Y^{2}Z^{2}+10Y^3Z^3$} \\
  &                                                      & See Appendix \ref{appB} &  \\\hline
\multirow{2}{*}{(216,92)} & \multirow{2}{*}{$(\Z/6\Z)^2\rtimes\operatorname{S}_3$}& $\operatorname{diag}(\zeta_{6},1,1),\,\operatorname{diag}(1,\zeta_{6},1),$& \multirow{2}{*}{$X^6+Y^6+Z^6$} \\
  &                                                      & $[X:Z:Y],\,[Y:Z:X]$ &  \\\hline
\multirow{3}{*}{(216,153)} & \multirow{3}{*}{$\operatorname{Hess}_{216}$} & $S=\operatorname{diag}(1,\zeta_3,\zeta_3^{-1}),\,U=\operatorname{diag}(1,1,\zeta_{3}),$& \multirow{3}{*}{$X^6+Y^6+Z^6-10(X^3Y^3+Y^3Z^3+Z^3X^3)$} \\
  &                                                      & $T=[Y:Z:X],$ &  \\
  &                                                      & $V=\left(\begin{array}{ccc}
1&1&1\\
1&\zeta_3&\zeta_3^{-1}\\
1&\zeta_3^{-1}&\zeta_3\\
\end{array}
\right)$ &  \\\hline
\multirow{3}{*}{(168,42)} & \multirow{3}{*}{$\operatorname{PSL}(2,7)$} & $\operatorname{diag}(1,\zeta_{7},\zeta_7^3),\,[Y:Z:X],$& \multirow{3}{*}{$X^5Y+Y^5Z+XZ^5-5X^2Y^2Z^2$} \\
  &                                                      & $\left(
  \begin{array}{ccc}
    \zeta_7-\zeta_7^6 & \zeta_7^2-\zeta_7^5 & \zeta_7^4-\zeta_7^3 \\
    \zeta_7^2-\zeta_7^5 & \zeta_7^4-\zeta_7^3 & \zeta_7-\zeta_7^6 \\
    \zeta_7^4-\zeta_7^3 & \zeta_7-\zeta_7^6 & \zeta_7^2-\zeta_7^5 \\
  \end{array}
\right)$ &  \\\hline
\multirow{3}{*}{(144,122)} & 
\multirow{3}{*}{$\Z/3\Z\times\operatorname{GL}_2(\mathbb{F}_3)$} & $\operatorname{diag}(1,\zeta_{24},\zeta_{24}^{19}),\,[X:Z:Y],$ & \multirow{3}{*}{$X^6+Y^5Z+YZ^5$} \\
  &                                                            & $[X:c(\zeta_{8}^{5}Z-Y):\zeta_{8}c(\zeta_{8}Z-Y)]$ &  \\
  &                                                            & $c:=(1+i)/2$ &  \\\hline
\multirow{2}{*}{(72,43)} & \multirow{2}{*}{$\Z/3\Z\rtimes\operatorname{S}_4$} & $\operatorname{diag}(1,\zeta_6,\zeta_6^2),\,\operatorname{diag}(1,1,-1),$ & $X^6+Y^6+Z^6+\beta_{2,2}X^2Y^2Z^2$ \\
 &  & $[Z:Y:X],\,[Y:Z:X]$ & $\beta_{2,2}\neq0$  \\\hline
(63,3) &   $\Z/21\Z\rtimes\Z/3\Z$& $\operatorname{diag}(1,\zeta_{21},\zeta_{21}^{17}),\,[Z:X:Y],$ & $X^5Y+Y^5 Z+XZ^5$  \\\hline
\multirow{3}{*}{(60,5)} & \multirow{3}{*}{$\operatorname{A}_5$} & $\operatorname{diag}(1,\zeta_5,\zeta_5^{-1}),\,[X:Z:Y],$ & $32X^6+\gamma^5X(Y^5+Z^5)+8(12-\gamma^5)X^4YZ+$ \\
 &  & $\left(
 \begin{array}{ccc}
 1 & 1 & 1 \\
 2 & (-1+\sqrt{5})/2 & (-1-\sqrt{5})/2 \\
 2 & (-1-\sqrt{5})/2 & (-1+\sqrt{5})/2 \\
 \end{array}\right)$ & $+2(48+\gamma^5)X^2Y^2Z^2+(32-\gamma^5)Y^3Z^3$  \\\hline
 \multirow{2}{*}{(54,5)} & \multirow{2}{*}{$(\Z/3\Z)^2\rtimes\Z/6\Z$} & $\operatorname{diag}(1,\zeta_3,1),\,\operatorname{diag}(1,1,\zeta_3),[X:Z:Y]$ & $X^6+Y^6+Z^6+\beta_{3,3}\left(X^3Y^3+Y^3Z^3+X^3Z^3\right)$ \\
 &  & $[Z:X:Y]$ & $\beta_{3,3}\neq0,-10$ \\\hline
(36,12) & $\Z/6\Z\times\operatorname{S}_3$ & $\operatorname{diag}(\zeta_6,1,1),\,\operatorname{diag}(1,\zeta_{3},1),[X:Z:Y]$ & $X^6+Y^6+Z^6+\beta_{0,3}Y^3Z^3,\,\beta_{0,3}\neq0$ \\\hline
   \end{tabular}
\end{table}

\begin{center}
\begin{table}[!th]
  \renewcommand{\arraystretch}{1.3}
  \vspace{4mm} 
  \centering
\scriptsize
\begin{tabular}{|c|c|c|c|}
  \hline
\multirow{3}{*}{(36,9)} & \multirow{3}{*}{$\operatorname{Hess}_{36}$} & \multirow{3}{*}{$S,\,T,\,V,\,UVU^{-1}$} & $X^6+Y^6+Z^6+\beta_{4,1}XYZ(X^3+Y^3+Z^3)$ \\
 &  &  & $+3\beta_{4,1}X^2Y^2Z^2-2(\beta_{4,1}+5)(X^3Y^3+Y^3Z^3+X^3Z^3)$  \\
  &  &  & $\beta_{4,1}\neq0$  \\\hline
(30,4) & $\Z/30\Z$ & $\operatorname{diag}(1,\zeta_{30}^5,\zeta_{30}^6)$& $X^6+Y^6+XZ^5$ \\\hline
(25,1) & $\Z/25\Z$ & $\operatorname{diag}(1,\zeta_{25},\zeta_{25}^{20})$ & $X^6+Y^5 Z+XZ^5$ \\\hline
 \multirow{2}{*}{(24,12)} & \multirow{2}{*}{$\operatorname{S}_4$} & $\operatorname{diag}(1,-1,1),\,\operatorname{diag}(1,1,-1),$ & $X^6+Y^6+Z^6+\beta_{2,2}X^2Y^2Z^2$ \\
 &  & $[X:Z:Y],[Y:Z:X]$ & $+\beta_{2,4}(X^2Y^4+Y^2Z^4+X^4Z^2+X^2Z^4+X^4Y^2+Y^4Z^2)$ \\\hline
(21,1) & $\Z/7\Z\rtimes\Z/3\Z$ & $\operatorname{diag}(1,\zeta_{7},\zeta_7^3),\,[Y:Z:X]$ & $X^5Y+Y^5Z+XZ^5+\beta_{4,2}X^2Y^2Z^2,\,\beta_{4,2}\neq-5,0$\\\hline
 \multirow{2}{*}{(18,3)} & \multirow{2}{*}{$\Z/3\Z\times\operatorname{S}_3$} & \multirow{2}{*}{$\operatorname{diag}(1,\zeta_3,1),\,\operatorname{diag}(1,1,\zeta_3),[X:Z:Y]$} & $X^6+Y^6+Z^6+\beta_{0,3}Y^3Z^3+\beta_{3,3}X^3\left(Y^3+Z^3\right)$ \\
 &  &  & $\beta_{0,3}\neq\beta_{3,3}$ \\\hline
 \multirow{2}{*}{(18,4)}  &  \multirow{2}{*}{$\Z/3\Z\rtimes\operatorname{S}_3$} & \multirow{2}{*}{$S,\,T,V^2$} & $X^6+Y^6+Z^6+\beta_{4,1}XYZ(X^3+Y^3+Z^3)$  \\
 &  &  & $+\beta_{3,3}(X^3Y^3+Y^3Z^3+X^3Z^3)+\beta_{2,2}X^2Y^2Z^2$  \\\hline

(16,8) & $\operatorname{SD}_{16}$ & $\operatorname{diag}(1,\zeta_{8},\zeta_{8}^3),\,[X:Z:Y]$ & $X^6+Y^5Z+YZ^5+\beta_{4,2}X^2Y^2Z^2,\,\beta_{4,2}\neq0$ \\\hline
 (15,1) & $\Z/15\Z$ & $\operatorname{diag}(1,\zeta_{15}^5,\zeta_{15}^{6})$ & $X^6+Y^6+XZ^5+\beta_{3,3}X^3Y^3,\,\beta_{3,3}\neq0,\pm2$ \\\hline
 (12,5)  &  $\Z/2\Z\times\Z/6\Z$ & $\operatorname{diag}(\zeta_6,1,1),\,\operatorname{diag}(1,1,-1)$ & $X^6+Y^6+Z^6+\beta_{2,4}Y^2Z^4+\beta_{4,2}Y^4Z^2,\,\beta_{2,4}\neq\beta_{4,2}$  \\\hline
 (12,4)  &  $\operatorname{D}_{12}$ & $\operatorname{diag}(1,\zeta_6,\zeta_6^2),\,[Z:Y:X]$ & $X^6+Y^6+Z^6+\beta_{3,0}X^3Z^3+\beta_{2,2}X^2Y^2Z^2+\beta_{1,4}XY^4Z$  \\
 &  & & $(\beta_{3,0},\beta_{1,4})\neq(0,0)$  \\\hline
 (12,3)  &  $\operatorname{A}_{4}$ & $\operatorname{diag}(1,1,-1),\,\operatorname{diag}(1,-1,1)$ & $X^6+Y^6+Z^6+\beta_{4,2}(X^4Y^2+Y^4Z^2+X^2Z^4)$   \\
 &  & $[Y:Z:X]$ & $+\beta_{2,4}(X^2Y^4+Y^2Z^4+X^4Z^2)+\beta_{2,2}X^2Y^2Z^2$  \\\hline
(12,3)  &  $\operatorname{A}_{4}$ & $\operatorname{diag}(1,1,-1),\,\operatorname{diag}(1,-1,1)$ & $X^6+\zeta_3^{-1}Y^6+\zeta_3Z^6+\beta_{4,2}(\zeta_3^{-1}X^4Y^2+\zeta_3Y^4Z^2+X^2Z^4)$   \\
 &  & $[Y:Z:X]$ & $+\beta_{2,4}(\zeta_3^{-1}X^2Y^4+\zeta_3Y^2Z^4+X^4Z^2)$  \\\hline
 (10,1)  &  $\operatorname{D}_{10}$ & $\operatorname{diag}(1,\zeta_5,\zeta_5^{-1}),\,[X:Z:Y]$ & $X^6+XZ^5+XY^5+\beta_{4,1}X^4YZ+$\\
 &  & & $+\beta_{2,2}X^2Y^{2}Z^{2}+\beta_{0,3}Y^3Z^3$ \\\hline
 (10,2)  &  $\Z/10\Z$&$\operatorname{diag}(1,\zeta_{10}^5,\zeta_{10}^{6})$ & $X^6+Y^6+XZ^5+\beta_{4,2}X^4Y^2+\beta_{2,4}X^2Y^4$  \\
 &  & & $\beta_{4,2},\beta_{2,4}\neq0$  \\\hline
 (9,2)  &  $\varrho_{h}\left((\Z/3\Z)^2\right)$ & $\operatorname{diag}(1,\zeta_{3},1),\operatorname{diag}(1,1,\zeta_{3})$ & $X^6+Y^6+Z^6+Z^3\left(\beta_{3,0}X^3+\beta_{0,3}Y^3\right)+\,\beta_{3,3}X^3Y^3$  \\\hline
 \multirow{2}{*}{(9,2)}  &  \multirow{2}{*}{$\varrho_{nh}\left((\Z/3\Z)^2\right)$} & \multirow{2}{*}{$S,\,T$} & $X^5Y+Y^5Z+XZ^5+\beta_{2,4}(X^2Y^4+Y^2Z^4+X^4Z^2)$  \\
 &  & & $+\beta_{1,3}(XY^3Z^2+X^2YZ^3+X^3Y^2Z)$  \\\hline
 (8,4)  &  $\operatorname{Q}_8$ & $\operatorname{diag}(1,\zeta_{4},\zeta_4^{-1}),[X:\zeta_8Z:-\zeta_8^{-1}Y]$ & $X^6+Y^5Z+YZ^5+\beta_{2,0}X^2(Z^4-Y^4)+\beta_{2,2}X^2Y^2Z^2$ \\\hline
 (8,3)  &  $\operatorname{D}_8$ & $\operatorname{diag}(1,\zeta_{4},\zeta_4^{-1}),[X:Z:Y]$ & $X^6+Y^5Z+YZ^5+\beta_{0,3}Y^3Z^3+\beta_{4,1}X^4YZ$  \\
 &  & & $+X^2\left(\beta_{2,0}Z^4+\beta_{2,2}Y^2Z^2+\beta_{2,0}Y^4\right)$  \\\hline
 \multirow{3}{*}{(6,1)}  &  \multirow{3}{*}{$\operatorname{S}_3$} & \multirow{3}{*}{$\operatorname{diag}(1,\zeta_{3},\zeta_3^{-1}),[X:Z:Y]$} & $X^6+Y^6+Z^6+\beta_{4,1}X^4YZ+\beta_{3,3}X^3(Y^3+Z^3)$  \\
 &  & & $+\beta_{2,2}X^2Y^2Z^2+\beta_{1,2}XYZ(Y^3+Z^3)+\beta_{0,3}Y^3Z^3$  \\
 &  & & $\beta_{4,1}\neq\beta_{1,2}$ or $\beta_{3,3}\neq\beta_{0,3}$  \\\hline
(6,2)  &  $\varrho_{(0,1)}\left(\Z/6\Z\right)$ & $\operatorname{diag}(\zeta_6,1,1)$ & $X^6+L_{6,X}$  \\\hline
  (6,2)  &  $\varrho_{(1,3)}\left(\Z/6\Z\right)$ & $\operatorname{diag}(1,\zeta_6,-1)$ & $X^6+Y^6+Z^6+\beta_{2,0}X^4Z^2+\beta_{0,3}Y^3Z^3+ $  \\
    &   &  & $+X^2\left(\beta_{4,0}Z^4+\beta_{4,3}Y^3Z\right)$  \\\hline
 \multirow{2}{*}{(5,1)}  &  $\varrho_{(1,2)}\left(\Z/5\Z\right)$ & $\operatorname{diag}(1,\zeta_5,\zeta_5^2)$ & $X^6+XZ^5+XY^5+\beta_{3,1}X^3YZ^2$  \\
   &  &  & $+\beta_{2,3}X^2Y^3Z+\beta_{0,2}Y^2Z^4$  \\\hline
 (5,1)  &  $\varrho_{(0,1)}\left(\Z/5\Z\right)$ & $\operatorname{diag}(1,1,\zeta_{5})$ & $Z^5Y+L_{6,Z}$  \\\hline
%
 (4,1)  &  $\Z/4\Z$ & $\operatorname{diag}(1,\zeta_{4},\zeta_4^{-1})$ & $X^6+Y^5Z+YZ^5+\beta_{0,3}Y^3Z^3+\beta_{4,1}X^4YZ$  \\
 &  & & $+X^2\left(\beta_{2,0}Z^4+\beta_{2,2}Y^2Z^2+\beta_{2,4}Y^4\right)$  \\\hline
 (4,2)  &  $\left(\Z/2\Z\right)^2$ & $\operatorname{diag}(1,1,-1),\,\operatorname{diag}(1,-1,1)$ & $Z^6+Z^4L_{2,Z}+Z^2L_{4,Z}+L_{6,Z}$  \\
 &  & & $L_{i,Z}\in K[X^2,Y^2]$  \\\hline
(3,1)  &  $\varrho_{(0,1)}\left(\Z/3\Z\right)$ & $\operatorname{diag}(1,1,\zeta_3)$ & $Z^6+Z^3L_{3,Z}+L_{6,Z}$  \\\hline
 (3,1)  &  $\varrho_{(1,2)}\left(\Z/3\Z\right)$ & $\operatorname{diag}(1,\zeta_3,\zeta_3^{-1})$ & $X^5Y+Y^5Z+XZ^5+\beta_{2,4}X^2Y^4+\beta_{0,2}Y^2Z^4+\beta_{4,0}X^4Z^2$  \\
    &   &  & $+XYZ(\beta_{3,2}X^2Y+\beta_{1,3}Y^2Z+\beta_{2,1}XZ^2)$  \\\hline
 (3,1)  &  $\varrho_{(1,2)}\left(\Z/3\Z\right)$ & $\operatorname{diag}(1,\zeta_3,\zeta_3^{-1})$ & $X^6+Y^6+Z^6+XYZ\left(\beta_{4,1}X^3+\beta_{1,4}Y^3+\beta_{1,2}Z^3\right)$  \\
    &   &  & $+\beta_{2,2}X^2Y^2Z^2+\beta_{3,3}X^3Y^3+\beta_{3,0}X^3Z^3+\beta_{0,3}Y^3Z^3$  \\\hline

 (2,1)  &  $\Z/2\Z$ & $\operatorname{diag}(1,1,-1)$ & $Z^6+Z^4L_{2,Z}+Z^2L_{4,Z}+L_{6,Z}$  \\\hline
   \end{tabular}
\end{table}
\end{center}
Secondly, the following diagram shows how looks like the stratification of smooth plane sextic curves by automorphism groups.
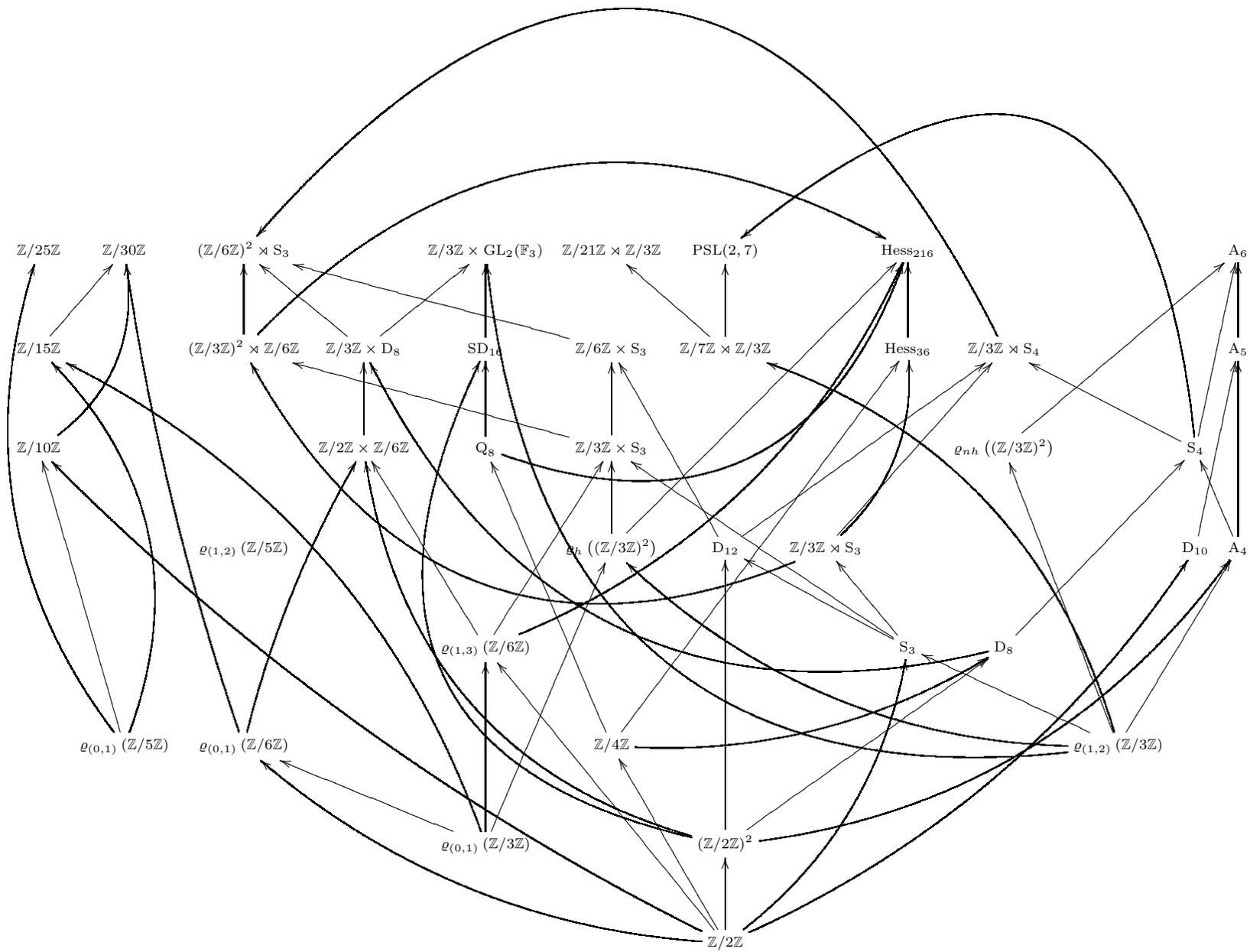
\scriptsize
\begin{figure}
\rotatebox{90}{
$$
\xymatrixrowsep{1.2cm}
\xymatrixcolsep{0.12cm}
\xymatrix{
\Z/25\Z & \Z/30\Z &(\Z/6\Z)^2\rtimes\operatorname{S}_3 &  & \Z/3\Z\times\operatorname{GL}_2(\mathbb{F}_3)  &  \Z/21\Z\rtimes\Z/3\Z &  \operatorname{PSL}(2,7) & &\operatorname{Hess}_{216} &  & & & \operatorname{A}_6 \\
\Z/15\Z\ar[ru] & & (\Z/3\Z)^2\rtimes\Z/6\Z\ar@/^5.5pc/[rrrrrru]\ar[u] & \Z/3\Z\times\operatorname{D}_8\ar[lu]\ar[ru]  & \operatorname{SD}_{16}\ar[u]& \Z/6\Z\times\operatorname{S}_3\ar[lllu]  & \Z/7\Z\rtimes\Z/3\Z\ar[lu]\ar[u] & & \operatorname{Hess}_{36}\ar[u]&\Z/3\Z\rtimes\operatorname{S}_4\ar@/^-11.9pc/[lllllllu]  & & &\operatorname{A}_5\ar[u]   \\
\Z/10\Z\ar@/^-1.5pc/[ruu] & &   &  \Z/2\Z\times\Z/6\Z\ar[u]   & \operatorname{Q}_8\ar[u]\ar@/^-4.9pc/[rrrruu]       &\Z/3\Z\times\operatorname{S}_3\ar[u]\ar[lllu] & & & &\varrho_{nh}\left((\Z/3\Z)^2\right)
\ar[rrruu] &  & \operatorname{S}_4\ar[llu]\ar@/^-9.9pc/[llllluu]\ar[ruu] & \\
 & & \varrho_{(1,2)}\left(\Z/5\Z\right)  &     &        &\varrho_{h}\left((\Z/3\Z)^2\right)
  \ar[u]\ar[rrruuu] & \operatorname{D}_{12}\ar[luu]\ar[rrruu] & \Z/3\Z\rtimes\operatorname{S}_3\ar@/^-1.2pc/[ruu]\ar[rruu]\ar@/^5.9pc/[llllluu] & & &  &\operatorname{D}_{10}\ar[ruu] & \operatorname{A}_4\ar[lu]\ar[uu]\\
  & &   &     & \varrho_{(1,3)}\left(\Z/6\Z\right)\ar[luu]\ar[ruu]\ar@/^-2.5pc/[rrrruuuu] 
  & & &
&\operatorname{S}_3\ar[lu]\ar[llu]\ar[llluu] &\operatorname{D}_8\ar[rruu]\ar@/^5.9pc/[lllllluuu] &  & \\
 &  \varrho_{(0,1)}\left(\Z/5\Z\right)\ar[luuu]\ar@/^2.8pc/[luuuuu]\ar@/^-2.8pc/[luuuu]  &  \varrho_{(0,1)}\left(\Z/6\Z\right) 
 \ar@/^0.5pc/[ruuu]\ar@/^0.5pc/[luuuuu]    & &  & \Z/4\Z\ar@/^-1.4pc/[rrrru]\ar[luuu]\ar[rrruuuu]
   &                                 &            &
  &
  & \varrho_{(1,2)}\left(\Z/3\Z\right) 
\ar[llu]\ar@/^2.4pc/[llllluu]\ar[rruu]\ar[luuu]\ar@/^-3.2pc/[lllluuuu]\ar@/^8.7pc/[lllllluuuuu]   \\
 &    &  & &\varrho_{(0,1)}\left(\Z/3\Z\right) 
\ar[uu]\ar[llu]\ar[ruuu]\ar@/^-2.9pc/[lllluuuuu] &   &  (\Z/2\Z)^2\ar[rrruu]\ar@/^-2.8pc/[rrrrrruuu]\ar[uuu]\ar@/^3.8pc/[llluuuu]\ar@/^7.4pc/[lluuuuu]                               &            &  &  & \\
 &    &   & &  &  &                    \Z/2\Z\ar[u]\ar[luu]\ar[lluuu]\ar@/^1.9pc/[lllluu]\ar@/^1.5pc/[lllllluuuuu]\ar@/^-1.5pc/[rruuu]\ar@/^-1.9pc/[rrrrruuuu]             &            &  &  & \\
}
$$
}
\label{stratdiagram}
\end{figure}
\normalsize
\end{thm}
\newpage
\begin{cor}\label{correctDIK} We correlate the following observations with the results in \cite{DoiIdei} regarding subgroups of automorphisms for smooth plane sextics of orders $5,\,8,\,9$ and $27$ respectively.
\begin{enumerate}[1.]
\item We have $\mathcal{M}^{\operatorname{Pl}}_{10}\left(\varrho_{(0,1)}\left(\Z/5\Z\right)\right)\cap\mathcal{M}^{\operatorname{Pl}}_{10}(\operatorname{A}_6)=\emptyset$, which agrees with \cite[Proposition 2.13]{DoiIdei}. Also, $\mathcal{M}^{\operatorname{Pl}}_{10}\left(\varrho_{(1,2)}\left(\Z/5\Z\right)\right)$ has two disjoint ES-irreducible components namely, $\widetilde{\mathcal{M}^{\operatorname{Pl}}}_{10}\left(\varrho_{(1,2)}\left(\Z/5\Z\right)\right)$ and $\mathcal{M}^{\operatorname{Pl}}_{10}\left(\operatorname{D}_{10}\right)=\widetilde{\mathcal{M}^{\operatorname{Pl}}}_{10}\left(\operatorname{D}_{10}\right)\,\sqcup\,\widetilde{\mathcal{M}^{\operatorname{Pl}}}_{10}\left(\operatorname{A}_{6}\right)$, which agrees with \cite[Lemma 2.14]{DoiIdei}.
\item We have $\mathcal{M}^{\operatorname{Pl}}_{10}(\operatorname{Q}_8)\neq\emptyset$, which opposes \cite[Lemma 2.11]{DoiIdei}. More precisely, the stratum has dimension two and it decomposes as
    $$
    \mathcal{M}^{\operatorname{Pl}}_{10}(\operatorname{Q}_8)=\widetilde{\mathcal{M}^{\operatorname{Pl}}}_{10}\left(\operatorname{Q}_8\right)\,\sqcup\,\widetilde{\mathcal{M}^{\operatorname{Pl}}}_{10}\left(\operatorname{SD}_{16}\right)\,\sqcup\,\widetilde{\mathcal{M}^{\operatorname{Pl}}}_{10}\left(\Z/3\Z\times\operatorname{GL}_{2}(\mathbb{F}_3)\right)\,\sqcup\,\widetilde{\mathcal{M}^{\operatorname{Pl}}}_{10}\left(\operatorname{Hess}_{216}\right).
    $$
  However, the main result of \cite{DoiIdei} confirming that the Wiman sextic is the most symmetric smooth plane sextic remains correct.

\item We have $\mathcal{M}^{\operatorname{Pl}}_{10}(G)=\emptyset$ when $G=\Z/2\Z\times\Z/4\Z$ or $(\Z/2\Z)^3$, which agrees with \cite[Lemmas 2.8 and 2.9]{DoiIdei}. Also, $\mathcal{M}^{\operatorname{Pl}}_{10}(G)\neq\emptyset$ for $G=\Z/8\Z$ and $\operatorname{D}_8,$ and each of them has a
    single ES-irreducible component, moreover,
$$\mathcal{M}^{\operatorname{Pl}}_{10}(\Z/8\Z)= \mathcal{M}^{\operatorname{Pl}}_{10}(\operatorname{SD}_{16})= \widetilde{\mathcal{M}^{\operatorname{Pl}}_{10}}(\operatorname{SD}_{16})\,\sqcup\,\widetilde{\mathcal{M}^{\operatorname{Pl}}_{10}}(\Z/3\Z\times\operatorname{GL}_2(\mathbb{F}_3)),
$$
and
$$  \mathcal{M}^{\operatorname{Pl}}_{10}(\operatorname{D}_8)= \widetilde{\mathcal{M}^{\operatorname{Pl}}_{10}}(\operatorname{D}_{8})\,\sqcup\,\widetilde{\mathcal{M}^{\operatorname{Pl}}_{10}}(\operatorname{S}_4)\,\sqcup\, \widetilde{\mathcal{M}^{\operatorname{Pl}}_{10}}(\operatorname{A}_6)\,\sqcup\,\widetilde{\mathcal{M}^{\operatorname{Pl}}_{10}}(\Z/3\Z\rtimes\operatorname{S}_4)\,\sqcup\,\widetilde{\mathcal{M}^{\operatorname{Pl}}_{10}}\left((\Z/6\Z)^2\rtimes\operatorname{S}_3\right).$$

\item The strata $\mathcal{M}^{\operatorname{Pl}}_{10}(G)=\emptyset$ for $G=\Z/9\Z$ and $(\Z/3\Z)^3$, which agrees with \cite[Lemmas 1.4 and 1.5]{DoiIdei}.
\item For the Heisenberg group $\operatorname{He}_3$ of order $27$, generated by $S,\,U,\,T$, it is the case that $\mathcal{M}^{\operatorname{Pl}}_{10}(\operatorname{He}_3)\neq\emptyset$ and also that
$\mathcal{M}^{\operatorname{Pl}}_{10}(\operatorname{He}_3)\,\cap\,\mathcal{M}^{\operatorname{Pl}}_{10}(\operatorname{A}_6)=\emptyset$, which agrees with \cite[Lemmas 1.8 and 1.9]{DoiIdei}. Moreover, we improve such a result by showing that
\begin{eqnarray*}
\mathcal{M}^{\operatorname{Pl}}_{10}(\operatorname{He}_3)&=&\mathcal{M}^{\operatorname{Pl}}_{10}\left((\Z/3\Z)^2\rtimes\Z/6\Z\right)\\
&=&\widetilde{\mathcal{M}^{\operatorname{Pl}}_{10}}\left((\Z/3\Z)^2\rtimes\Z/6\Z\right)\,\sqcup\,\widetilde{\mathcal{M}^{\operatorname{Pl}}_{10}}\left((\Z/6\Z)^2\rtimes\operatorname{S}_3\right)\,\sqcup\,\widetilde{\mathcal{M}^{\operatorname{Pl}}_{10}}\left(\operatorname{Hess}_{216}\right).
\end{eqnarray*}
\end{enumerate}
\end{cor}

We can see in the stratification another example of a phenomenon that did not happen for quartics, but it happened for quintics for the first time and now for sextics. In \cite{finalstratum}, Badr-Lorenzo defined a \emph{final stratum} by automorphism groups to be a non-zero dimensional stratum that does not properly contain any other stratum. This may sound
odd since we could expect that by adding conditions in the parameters we would get bigger
automorphisms groups. However, they showed in \cite{finalstratum} that this is a normal situation for higher odd degrees $d$ such that $d=1\,\operatorname{mod}\,4$.

For even degrees, we now have an example. More precisely,
\begin{cor}
The stratum $\widetilde{\mathcal{M}^{\operatorname{Pl}}}_{10}\left(\varrho_{(1,2)}\left(\Z/5\Z\right)\right)$ is the unique \emph{final stratum} that exists for degree $6$.
\end{cor}

\begin{rem}
  It appeared a unique final stratum for quintics and sextics respectively. So one may wonder about the existence and uniqueness of such final strata in each degree $d\geq7$.
  \end{rem}

\section{Automorphisms of maximal order and defining equations}
Let $C:F(X,Y,Z)=0$ be a smooth plane sextic curve with non-trivial automorphism group $\operatorname{Aut}(C)\subseteq\operatorname{PGL}_3(K)$ over an algebraically closed field $K$ of characteristic $0$, and let $\sigma\in\operatorname{Aut}(C)$ be of maximal order $m$. Up to projective equivalence, we can assume that $\sigma$ acts as $(X:Y:Z)\mapsto(X:\zeta_m^a
Y:\zeta_m^b Z)$, where $\zeta_m$ denotes a fixed primitive $m$-th root of unity in
$K$, and $a,b$ are integers with $0\leq a<b\leq m-1$.
In this case, we say that $C$ is of \emph{Type $m,(a,b)$}.

Our work in \cite{MR3475065} allows us to generate all possible Types $m,(a,b)$ once the degree $d$ is fixed. In particular, we know that the positive integer $m$ should divide $21,\,24,\,25$ or $30$, see \cite[Corollary 33]{MR3475065}. On the other hand, we can associate to each Type $m,(a,b)$ a \emph{geometrically complete family over $K$}. That is, a defining polynomial equation $F_{m,(a,b)}(X,Y,Z)=0$ with parameters $\beta_{i,j}\in K$ as its coefficients, where any smooth plane degree $d$ curve of Type $m,(a,b)$ is defined over $K$ by a specialization of those parameters.

As a consequence, we obtain:

\begin{prop}\label{maximalordercyclic}
Let $C$ be a smooth plane curve of degree $6$ over $K$. Then, $C$ reduces to one of the following types. In particular, the stratum $\mathcal{M}^{\operatorname{Pl}}_{10}(\Z/m\Z)\neq\emptyset$ if and only if $m\in\{2,3,4,5,6,7,8,10,12,15,21,24,25,30\}$.
\scriptsize
\begin{table}[!th]
  \renewcommand{\arraystretch}{1.3}
  \caption{Maximal Cyclic Subgroups and Defining Equations\,\,\,}\label{table:CYclic Auto621.}
  \vspace{4mm} 
  \centering
\scriptsize
\begin{tabular}{|c|c|c|}
  \hline
 & Type: $m, (a,\,b)$ & $F(X,Y,Z)$ \\\hline\hline
1 & $30,(5,6)$& $X^{6} + Y^{6}+X Z^{5} $ \\\hline
2 & $25,(1,20)$& $X^{6} + Y^{5} Z+X Z^{5} $ \\\hline
3 & $24,(1,19)$& $X^{6} + Y^{5} Z +Y Z^{5}$ \\\hline
4 &   $21,(1,17)$& $X^{5} Y + Y^{5} Z+X Z^{5} $ \\\hline
5 & $15,(5,6)$& $X^{6} + Y^{6}+X Z^{5}+{\beta_{3,3}} X^{3} Y^{3} $ \\\hline
6  &  $12,(1,7)$& $X^{6} + Y^{5} Z +Y Z^{5} +{\beta_{3,3}}Y^{3} Z^{3} $ \\\hline
7 & $10,(5,6)$& $X^{6} + Y^{6}+X Z^{5}+ {\beta_{4,2}}X^{4} Y^{2} + {\beta_{2,4}}X^{2} Y^{4}  $ \\\hline
8 &  $8,(1,3)$& $X^{6} + Y^{5} Z +Y Z^{5} + {\beta_{4,2}}X^{2} Y^{2} Z^{2} $ \\\hline
9 & $7,(1, 3)$ & $X^{5}Y+Y^{5}Z+X Z^{5}+{{\beta_{4,2}}} X^{2}Y^{2}Z^{2}$\\\hline
10 & $6,(1,2)$& $X^6+Y^6+Z^6+\beta_{3,0}X^3Z^3+\beta_{2,2}X^2Y^2Z^2+\beta_{1,4}XY^4Z$ \\\hline
11 & $6,(1,3)$& $X^{6} + Y^{6} + Z^{6}+{\beta_{2,0}}X^{4} Z^{2}  + {\beta_{0,3}} Y^{3} Z^{3} +  X^{2}{\left({\beta_{4,0}}Z^{4} + {\beta_{4,3}}Y^{3} Z \right)}$ \\\hline
12 & $6,(0,1)$& $Z^{6} + {L_{6,Z}}$ \\\hline
13 & $5,(1,2)$& $X^6+XZ^5+XY^5+\beta_{3,1}X^3YZ^2+\beta_{2,3}X^2Y^3Z+\beta_{0,2}Y^2Z^4$ \\\hline
14 & $5,(1,4)$& $X^{6} + X Z^{5}+X Y^{5}+{\beta_{4,1}}X^{4} Y Z  + {\beta_{2,2}}X^{2} Y^{2} Z^{2}  + {\beta_{0,3}}Y^{3} Z^{3} $ \\\hline
15 &   $5,(0,1)$& $Z^{5}Y  + {L_{6,Z}}$ \\\hline
16 & $4,(1,3)$& $X^6+Y^5Z+YZ^5+\beta_{0,3}Y^3Z^3+\beta_{4,1}X^4YZ+X^2\left(\beta_{2,0}Z^4+\beta_{2,2}Y^2Z^2+\beta_{2,4}Y^4\right)$ \\\hline
17 & $3,(1, 2)$ & $X^5Y+Y^5Z+XZ^5+XYZ\left(\beta_{3,2}X^2Y+\beta_{1,3}Y^2Z+\beta_{2,1}XZ^2\right)+\beta_{2,4}X^2Y^4+\beta_{0,2}Y^2Z^4+\beta_{4,0}X^4Z^2$\\\hline
18 & $3,(1, 2)$ & $X^6+Y^6+Z^6+XYZ\left(\beta_{4,1}X^3+\beta_{1,4}Y^3+\beta_{1,2}Z^3\right)+\beta_{2,2}X^2Y^2Z^2+\beta_{3,3}X^3Y^3+\beta_{3,0}X^3Z^3+\beta_{0,3}Y^3Z^3$\\\hline
19 & $3,(0,1)$& $Z^{6} + Z^{3}{L_{3,Z}}  + {L_{6,Z}}$ \\\hline
20 & $2,(0,1)$& $Z^{6} + Z^{4}{L_{2,Z}}  + Z^{2}{L_{4,Z}}+{L_{6,Z}} $
\\\hline
   \end{tabular}
\end{table}
\normalsize
\end{prop}

\section{The automorphism groups for very large Types $m, (a,b)$}
In this section we determine the full automorphism group $\Aut(C)$ when there is $\sigma\in\Aut(C)$ of order $m\in\{d(d-1),\,(d-1)^2,\,d(d-2),\,d^2-3d+3,\,q(d-1):q\geq2\}$.

In particular, we need the following result, see \cite{MR3475065,Harui}:

\begin{thm}\label{BB} Let $C$ be a smooth plane degree $d$ curve of Type $m,\,(a,b)$.
\begin{enumerate}[1.]
\item If $m=d(d-1)$, then $\operatorname{Aut}(C)$ is cyclic of order $d(d-1)$. In this case $C$
is $K$-isomorphic to $C:X^d+Y^d+XZ^{d-1}=0$, where $\Aut(C)=\langle\sigma\rangle$ with $\sigma=\operatorname{diag}(1,\zeta_{d(d-1)}^{d-1},\zeta_{d(d-1)}^d)$.

\item If $m=(d-1)^2$, then $\operatorname{Aut}(C)$ is cyclic of order $(d-1)^2$. In this case $C$
is $K$-isomorphic to $C:X^d+Y^{d-1}Z+XZ^{d-1}=0$, where $\Aut(C)=\langle\sigma\rangle$ with $\sigma=\operatorname{diag}(1,\zeta_{(d-1)^2},\zeta_{(d-1)^2}^{-(d-1)})$.
\item If $m=d(d-2)$, then $C$ is $K$-isomorphic to
$C:X^d+Y^{d-1}Z+YZ^{d-1}=0$. For $d\neq 4,6$, $\Aut(C)$ is a central extension of the dihedral group $\operatorname{D}_{2(d-2)}$ by $\Z/d\Z$. More precisely,
$$\operatorname{Aut}(C)=\langle\sigma,\,\tau\,|\,\sigma^{d(d-2)}=\tau^2=1,\, \tau\sigma\tau=\sigma^{-(d-1)},...\rangle,$$
with $\sigma=\operatorname{diag}(1,\zeta_{d(d-2)},\zeta_{d(d-2)}^{-(d-1)})$ and $\tau=[X:Z:Y]$.
Therefore, $|\operatorname{Aut}(C)|=2d(d-2)$. When $d=6$, $\operatorname{Aut}(C)$ is a central extension of the symmetry group $\operatorname{S}_4$ by $\Z/6$, so it has order $144$. When $d=4$, $C$ is $K$-isomorphic to the Fermat quartic curve $\mathcal{F}_4:X^4+Y^4+Z^4=0$ and $\operatorname{Aut}(C)\cong(\Z/4\Z)^2\rtimes\operatorname{S}_3$
\item If $m=d^2-3d+3$, then $C$ is $K$-isomorphic to the Klein
curve defined by $\mathcal{K}_d:X^{d-1}Y+Y^{d-1}Z+Z^{d-1}X=0$. Moreover, for $d\geq 5$, we have
$$\operatorname{Aut}(C)=\langle\sigma,\tau|\sigma^{d^2-3d+3}=\tau^3=1\,\tau^{-1}\sigma\tau=\sigma^{-(d-1)}\rangle,$$
with $\sigma=\operatorname{diag}(1,\zeta_{d^2-3d+3},\zeta_{d^2-3d+3}^{-(d-2)})$ and $\tau=[Y:Z:X]$.
such that $|\operatorname{Aut}(C)|=3(d^2-3d+3)$.
\item If $m=q(d-1)$ for some $q\geq 2$, then
$\operatorname{Aut}(C)$ is cyclic of order $qq'(d-1)$ for some $q'$. 
\end{enumerate}
\end{thm}

Substituting $d=6$ in Theorem \ref{BB} yields:
\begin{cor}\label{verylarge} The stratum $\widetilde{\mathcal{M}^{\operatorname{Pl}}_{10}}(G)\neq\emptyset$ if  $G=\Z/30\Z,\,\Z/25\Z,\,\Z/3\Z\times\operatorname{GL}_2(\mathbb{F}_3),\,$ $\Z/21\Z\rtimes\Z/3\Z,\,\Z/15\Z,\,\Z/10\Z.$
Moreover, the generators and the defining equation corresponding to Types $m, (a,b)$ with $m=30,25,24,21,15$ and $10$ are given below.
\begin{center}
\begin{table}[!th]
  \renewcommand{\arraystretch}{1.3}
  \caption{Automorphism Groups For Very Large Types $m, (a,b)$}\label{table:VeryLargeAuto.}
  \vspace{4mm} 
  \centering
\scriptsize
\begin{tabular}{|c|c|c|c|}
  \hline
Type $m, (a,b)$& $\Aut(C)$ &\textit{Generators} & $F(X,Y,Z)$ \\\hline\hline
$30, (5,6)$ & $\Z/30\Z$& $\operatorname{diag}(1,\zeta_{30}^5,\zeta_{30}^6)$& $X^6+Y^6+XZ^5$ \\\hline
$25, (1,20)$ & $\Z/25\Z$& $\operatorname{diag}(1,\zeta_{25},\zeta_{25}^{20})$ & $X^6+Y^5 Z+XZ^5$ \\\hline
$24, (1,19)$ & 
$\Z/3\Z\times\operatorname{GL}_2(\mathbb{F}_3)$ & $\operatorname{diag}(1,\zeta_{24},\zeta_{24}^{19}),\,[X:Z:Y],$ & $X^6+Y^5Z+YZ^5$ \\
  &                                                            & $[X:c(\zeta_{8}^{5}Z-Y):\zeta_{8}c(\zeta_{8}Z-Y)]$ &  \\\hline
$21, (1,17)$ &   $\Z/21\Z\rtimes\Z/3\Z$& $\operatorname{diag}(1,\zeta_{21},\zeta_{21}^{17}),\,[Z:X:Y]$ & $X^5Y+Y^5 Z+Z^5X$  \\\hline
$15, (5,6)$ & $\Z/15\Z$ & $\operatorname{diag}(1,\zeta_{15}^5,\zeta_{15}^{6})$ & $X^6+Y^6+XZ^5+\beta_{3,3}X^3Y^3$ \\
 &  & & $\beta_{3,3}\neq0,\pm2$  \\\hline
$10, (5,6)$  &  $\Z/10\Z$&$\operatorname{diag}(1,\zeta_{10}^5,\zeta_{10}^{6})$ & $X^6+Y^6+XZ^5+\beta_{4,2}X^4Y^2+\beta_{2,4}X^2Y^4$  \\
 &  & & $\beta_{4,2},\beta_{2,4}\neq0$  \\\hline
   \end{tabular}
\end{table}
\end{center}
Here $c:=(1+i)/2$.
\end{cor}

\begin{proof}
Cases (1), (2) and (4) follows by applying Theorem \ref{BB}-(1), (2) and (4) to Table \ref{table:CYclic Auto621.}-(1),(2) and (4). Moreover, for the Klein sextic curve, we have that
\begin{eqnarray*}
\operatorname{Aut}(\mathcal{K}_6)&=&\langle a,b,c:a^3=b^7=c^3=1,\,ab=ba,\,ac=ca,\,cbc^{-1}=b^4\rangle\\
&\cong&\operatorname{GAP}(63,3),
\end{eqnarray*}
with $a=\sigma^7=\operatorname{diag}(1,\zeta_{3},\zeta_{3}^{2}),\,b=\sigma^3=\operatorname{diag}(1,\zeta_{7},\zeta_{7}^{3}),$ and $c=[Z:X:Y]$.

On the other hand, if $C$ is of Type $24, (1,19)$, then, by Theorem \ref{BB}-(3), $\operatorname{Aut}(C)$ contains the subgroup
\begin{eqnarray*}
H&:=&\langle\sigma,\tau:\sigma^{24}=\tau^2=1,\,\tau\sigma\tau=\sigma^{-5}\rangle\cong\operatorname{GAP}(48,26)\\
&=&\Z/3\Z\times\operatorname{SD}_{16}.
\end{eqnarray*}
Among all central extensions $G$ of $\operatorname{S}_4$ by $\Z/6\Z$,
$\operatorname{GAP}(144,122)$ is the only one that have such a subgroup. For more details, see \href{https://people.maths.bris.ac.uk/~matyd/GroupNames/129/e14/S4byC6.html}{Extensions of $S_4$ by $\Z/6\Z$}.

Regarding $C$ of Type $15, (5,6)$, it follows by Theorem \ref{BB}-(5) with $q=3$ that $\operatorname{Aut}(C)$ is always cyclic. Furthermore, $\Aut(C)\simeq\Z/30\Z$ if and only if $\beta_{3,3}=0$. The extra condition that $\beta_{3,3}\neq\pm2$ are to ensure the smoothness of $C$. Similarly, we deduce by Theorem \ref{BB}-(5) with $q=2$ that $\Aut(C)$ is always cyclic when $C$ is of Type $10, (5,6)$. If it is bigger, then it should be $\Z/30\Z$. This holds if and only if $\beta_{4,2}=\beta_{2,4}=0$, which is excluded by assumption.
\end{proof}

\section{Preliminaries about automorphism groups}
The determination of the finite subgroups $G$ of $\operatorname{PGL}_3(K)$ is quite well understood in the subject. For instance, we recall this one made by H. Mitchell \cite{Mit}, which is based entirely on geometrical methods. H. Mitchell \cite[\S 1-10]{Mit} proved that $G$ fixes a point, a line or a triangle unless it is primitive and conjugate to some group in a specific list. However, as a consequence of Maschke's theorem in group representation theory, the first two cases are equivalent, in the sense that if $G$ fixes a point (respectively a line), then it also
fixes a line not passing through the point (respectively a point not lying the line).

\vspace{0.2cm}
\noindent\textbf{Notations.} For a non-zero monomial $cX^{i_1}Y^{i_2}Z^{i_3}$ with $c\in K^*$, its exponent is defined to be $\operatorname{max}\{i_1,i_2,i_3\}$. For a homogenous polynomial $F(X,Y,Z)$, the core of it is
defined to be the sum of all terms of $F$ with the greatest
exponent. Now, let $C_0$ be a smooth plane curve over $K$, a pair $(C,G)$ with
$G\leq \operatorname{Aut}(C)$ is said to be a descendant of $C_0$ if $C$ is defined
by a homogenous polynomial whose core is a defining polynomial of
$C_0$ and $G$ acts on $C_0$ under a suitable change of the
coordinates system, i.e. $G$ is $\operatorname{PGL}_3(K)$-conjugate to a subgroup of
$\operatorname{Aut}(C_0)$.

An element of $\operatorname{PGL}_3(K)$ is called \emph{intransitive} if it has the matrix shape
$$\left(
                                              \begin{array}{ccc}
                                                 1 & 0 & 0\\
                                                0 & \ast & \ast\\
                                                0 & \ast & \ast \\
                                              \end{array}
                                            \right).$$
The subgroup of $\operatorname{PGL}_3(K)$ of all intransitive elements is denoted by $\operatorname{PBD}(2,1)$. Obviously, there is a natural map $\Lambda:\operatorname{PBD}(2,1)\rightarrow \operatorname{PGL}_2(K)$ given by $$\left(
                                              \begin{array}{ccc}
                                                 1 & 0 & 0\\
                                                0 & \ast & \ast\\
                                                0 & \ast & \ast \\
                                              \end{array}
                                            \right)\in\operatorname{PBD}(2,1)\mapsto\left(
                                                                                                     \begin{array}{cc}
                                                                                                       \ast & \ast \\
                                                                                                       \ast& \ast \\
                                                                                                     \end{array}
                                                                                                   \right)\in \operatorname{PGL}_2(K).$$

To determine the full automorphism groups of the other types of smooth plane sextic curves, Theorem \ref{teoHarui} below is very useful. It can be viewed as a projection of Mitchell's classification to smooth plane curves. For more details, we refer to the work of T. Harui \cite[Theroem 2.1]{Harui}.

\begin{thm}\label{teoHarui} Let $C$ be a smooth plane curve of degree $d\geq4$ defined over an algebraically closed field $K$ of characteristic $0$. Then, one of the following situations holds:
\begin{enumerate}[1.]
  \item $\operatorname{Aut}(C)$ fixes a point on $C$ and then it is cyclic.
  \item $\operatorname{Aut}(C)$ fixes a point not lying on $C$ where we can think about $\Aut(C)$ in the following commutative diagram, with exact rows and vertical injective
morphisms:
  $$
\xymatrix
{
1\ar[r]  & K^*\ar[r]                    & \operatorname{PBD}(2,1)\ar[r]^{\Lambda}& \operatorname{PGL}_2(K)\ar[r]& 1         \\
         &                              &                                        &                              &\\
1\ar[r]  & N\ar[r]\ar@{^{(}->}[uu] & \operatorname{Aut}(C)\ar[r]\ar@{^{(}->}[uu] & G'\ar[r]\ar@{^{(}->}[uu]& 1
}
$$
Here, $N$ is a cyclic group of order dividing the degree $d$ and $G'$ is a
subgroup of $\operatorname{PGL}_2(K)$, which is conjugate to a cyclic group $\Z/m\Z$ of
order $m$ with $m\leq d-1$, a Dihedral group $\operatorname{D}_{2m}$ of order $2m$
with $|N|=1$ or $m|(d-2)$, one of the alternating groups $\operatorname{A}_4$, $\operatorname{A}_5$, or the symmetry group $\operatorname{S}_4$.
\begin{rem}\label{extensionrem}
We note that $N$ is viewed as the part of $\Aut(C)$ acting on the variable $B\in\{X,Y,Z\}$ and fixing the other two variables, while $G'$ is the part acting on $\{X,Y,Z\}-\{B\}$ and fixing $B$. For example, if $B=X$, then every automorphism in $N$ has the shape $\operatorname{diag}(\zeta_n,1,1)$ for some $n$th root of unity $\zeta_n$.
\end{rem}
\item $\operatorname{Aut}(C)$ is conjugate to a subgroup $G$ of $\operatorname{Aut}(\mathcal{F}_d)$, where $\mathcal{F}_d$ is the Fermat curve $X^d+Y^d+Z^d=0$. In particular, $|G|$ divides $|\operatorname{Aut}(\mathcal{F}_d)|=6d^2$, and $(C,G)$ is a descendant of $\mathcal{F}_d$.
\item $\operatorname{Aut}(C)$ is conjugate to a subgroup $G$ of $\operatorname{Aut}(\mathcal{K}_d)$, where $\mathcal{K}_d$ is the Klein curve curve
$XY^{d-1}+YZ^{d-1}+ZX^{d-1}$. In this case, $|\operatorname{Aut}(C)|$ divides $|\operatorname{Aut}(\mathcal{K}_d)|=3(d^2-3d+3)$, and
$(C,G)$ is a descendant of $\mathcal{K}_d$.
\item $\operatorname{Aut}(C)$ is conjugate to one of the finite primitive subgroup of $\operatorname{PGL}_3(K)$ namely, the Klein group
$\operatorname{PSL}(2,7)$, the icosahedral group $\operatorname{A}_5$, the alternating group
$\operatorname{A}_6$, or to one of the Hessian groups $\operatorname{Hess}_{*}$ with $*\in\{36,\,72,\,216\}$.
\end{enumerate}
\end{thm}

An homology of period $n$ is a projective linear transformation of the plane $\mathbb{P}^2(K)$, which is $\operatorname{PGL}_3(K)$-conjugate to
$\operatorname{diag}(1,1,\zeta_{n})$, where $\zeta_{n}$ is a primitive $n$th root of unity. Such a transformation fixes pointwise a line $\mathcal{L}$ (its axis) and a point $P$ off this line (its center). In its canonical form, $\mathcal{L}:Z=0$ and center $P=(0:0:1)$.

The following fact due to H. Mitchell \cite{Mit} turn out to be very useful in hand when one wants to determine the automorphism group of smooth plane curves over an algebraically closed field $K$ of characteristic 0.
\begin{thm}\label{Mitchell1}
Let $G$ be a finite group of $\operatorname{PGL}_3(K)$. The following holds.
\begin{enumerate}[1.]
  \item If $G$ contains an homology of period $n\geq4$, then it fixes a point, a line or a triangle. On the other hand, a transformation inside $G$, which leaves invariant the center of an homology, must leave invariant its axis and vice versa.
  \item The Hessian group $\operatorname{Hess}_{216}$ is the only finite subgroup of $\operatorname{PGL}_3(K)$ that contains homologies of period $n=3$, and does not leave invariant a point, a line or a triangle.
  \item Inside $G$, a transformation that leaves invariant the center of an homology must leave invariant its axis and vice versa.
\end{enumerate}
\end{thm}

On the other hand, there is a strong connection between having homologies inside $\Aut(C)$ and having the so-called \emph{Galois points}. To our knowledge, the notion of Galois points was first introduced by H. Yoshihara in 1996. It has been extensively studied by other mathematicians around, see for example \cite{Fuk06, Fuk08, Fuk09, Fuk13, Hom06, MY00, Yoshihara} for more details.

\begin{defn}
 A point $P\in\mathbb{P}^2(K)$ is said to be a Galois point for $C$ if the natural projection $\pi_P$ from $C$ to a line $\mathcal{L}$ with center $P$ is a Galois covering. If, moreover, $P\in C$ (respectively $P\notin C$), then $P$ is called an inner (respectively outer) Galois point.
\end{defn}

In particular, we have the following fact for smooth plane curves due to T. Harui \cite[Lemma 3.7]{Harui}:

\begin{prop}\label{d,d-1 homologies}
Let $C$ be a smooth plane curve of degree $d\geq5$ over an algebraically closed field $K$ of characteristic $0$. If $\sigma\in\Aut(C)$ is an homology with center $P$, then $|\langle\sigma\rangle|$ divides $d$ when $P\notin C$ and $|\langle\sigma\rangle|$ divides $d-1$ when $P\in C$. The equality $|\langle\sigma\rangle|=d$ (respectively $|\langle\sigma\rangle|=d-1$) holds if and only if $P$ is an outer (respectively inner) Galois point for $C$.
\end{prop}
\section{The automorphism group for Type $12,(1,7)$: The stratum $\mathcal{M}^{\operatorname{Pl}}_{g}(\Z/12\Z)$}
In this section, we will describe the full automorphism group of smooth plane sextic curves $C$ of Type $12, (1,7)$. That is, when $C$ is defined by an equation of the form:
$$C:X^6+Y^5Z+YZ^5+\beta_{3,3}Y^3Z^3=0,$$
for some $\beta_{3,3}\in K$. In this situation, we have $\sigma:=\operatorname{diag}(1,\zeta_{12},\zeta_{12}^7)\in\Aut(C)$ of order $12$. Also, we can assume that $\sigma$ is of maximal order in $\Aut(C)$, in particular, $\beta_{3,3}\neq0$. Otherwise, $C$ will be of Type $24, (1,19)$, which was covered before.

We obtain the following result concerning $\Aut(C)$:
\begin{prop}\label{1217}
Let $C$ be a smooth plane sextic curves $C$ of Type $12, (1,7)$ as above. Then, the full automorphism group $\operatorname{Aut}(C)$ is classified as follows.
\begin{enumerate}[1.]
  \item If $\beta_{3,3}\neq\pm\dfrac{10}{3}$, then $\operatorname{Aut}(C)=\langle\sigma,\,\tau\rangle\cong \Z/3\Z\times\operatorname{D}_{8}$, where $\sigma:=\operatorname{diag}(1,\zeta_{12},\zeta_{12}^7)$ and $\tau:=[X:Z:Y]$.
\item If $\beta_{3,3}=\pm\dfrac{10}{3}$, then $C$ is $K$-isomorphic to the Fermat curve $\mathcal{F}_6:X^6+Y^6+Z^6=0$. In this case, $\operatorname{Aut}(\mathcal{F}_6)=\langle\eta_1,\,\eta_2,\,\eta_3,\,\eta_4\rangle\cong(\Z/6\Z)^2\rtimes\operatorname{S}_3$, where $\eta_1:=[X:Z:Y],\,\eta_2:=[Y:Z:X],\,\eta_3:=\operatorname{diag}(\zeta_6,1,1)$ and $\eta_4:=\operatorname{diag}(1,\zeta_6,1)$.
\end{enumerate}
\end{prop}

\begin{proof}
Non of the finite primitive subgroups of $\operatorname{PGL}_3(K)$ mentioned in Theorem \ref{teoHarui}-(5) has elements of order $12$. Actually, non of them has elements of order $>6$ except the Klein group $\operatorname{PSL}(2,7)$ which has some elements of order $7$. Also, $12\nmid\Aut(\mathcal{K}_6)$, so $C$ is not a descendant of the Klein sextic curve $\mathcal{K}_6$. Moreover, $\operatorname{Aut}(C)$ is not cyclic as $$\langle\sigma,\tau:\sigma^{12}=\tau^2=1,\,\tau\sigma\tau=\sigma^{7}\rangle\cong \Z/3\Z\times\operatorname{D}_{8}=\operatorname{GAP}(24,10)$$ is a non-cyclic subgroup of automorphisms. Therefore, $\Aut(C)$ fixes a point $P$ not lying on $C$ or $C$ is a descendant of the Fermat sextic curve $\mathcal{F}_6$.

Suppose first that $\operatorname{Aut}(C)$ fixes a line $\mathcal{L}$ in $\mathbb{P}_2(K)$ and a point $P\notin C$ off this line. Since $\sigma$ and $\tau$ are automorphisms, then $\mathcal{L}$ should be the reference line $X=0$ and $P$ the reference point $(1:0:0)$. In particular, by Theorem \ref{teoHarui}-(2), all automorphisms of $C$ are intransitive of the shape
     $
      \left(
        \begin{array}{ccc}
          1 & 0 & 0 \\
          0 & * & * \\
          0 & * & * \\
        \end{array}
      \right).      $
Moreover, we can think about $\Aut(C)$ in a short exact sequence $1\rightarrow N\rightarrow \Aut(C)\rightarrow\Lambda(\Aut(C))\rightarrow 1,$ where $N=\langle\sigma^2\rangle$ by Remark \ref{extensionrem}, and $\Lambda(\Aut(C))$ contains the Klein $4$-group $(\Z/2\Z)^2$ generated by $\Lambda(\sigma)=\operatorname{diag}(\zeta_{12},-\zeta_{12})$ and $\Lambda(\tau)=[Z:Y]$. Thus $\Lambda(\Aut(C))$ should be $(\Z/2\Z)^2,\,\operatorname{A}_4,\,\operatorname{A}_5$ or $\operatorname{S}_4$. If $\Lambda(\Aut(C))=\operatorname{A}_4,\,\operatorname{A}_5$ or $\operatorname{S}_4$, then the group structure of $\operatorname{A}_4$ confirms that there must be an automorphism $\eta\notin\langle\sigma,\tau\rangle$ of $C$ such that $\Lambda(\eta)$ has order $3$, $\Lambda(\eta)\Lambda(\sigma)\Lambda(\eta)^{-1}=\Lambda(\sigma)\Lambda(\tau)$ and $\Lambda(\eta)\Lambda(\tau)\Lambda(\eta)^{-1}=\Lambda(\sigma)$. This implies that $\eta=\left(
        \begin{array}{ccc}
          1 & 0 & 0 \\
          0 & \nu & \pm\nu \\
          0 & -i\nu & \pm i\nu \\
        \end{array}
      \right)$ or $\eta=\left(
        \begin{array}{ccc}
          1 & 0 & 0 \\
          0 & \nu & \pm\nu \\
          0 & i\nu & \mp i\nu \\
        \end{array}
      \right)$
for some $\nu\in K^*$, however, it does not preserve the defining equation for $C$. Consequently, no more automorphisms for $C$ appear in this case.

Secondly, assume that $C$ is a descendant of $\mathcal{F}_6$. We note that $\Aut(\mathcal{F}_6)\cong\operatorname{GAP}(216,92)=(\Z/6\Z)^2\rtimes\operatorname{S}_3$, $4^{\operatorname{th}}$ semidirect product of $(\Z/6\Z)^2$ and $\operatorname{S}_3$ acting faithfully (see \href{https://people.maths.bris.ac.uk/~matyd/GroupNames/193/e15/S3byC6%5E2.html#s4}{semidirect products of $(\Z/6\Z)^2$ and $\operatorname{S}_3$}).
More precisely, it is generated by $\eta_1,\,\eta_2,\,\eta_3$, and $\eta_4$ of orders $2, 3, 6, 6$ such that $$(\eta_1\eta_2)^2=(\eta_1\eta_3)(\eta_3\eta_1)^{-1}=(\eta_3\eta_4)(\eta_4\eta_3)^{-1}=\eta_1\eta_4\eta_1(\eta_3\eta_4)^{-5}=\eta_2\eta_3\eta_2^{-1}(\eta_3\eta_4)^{-5}=1.$$
By assumption, $\phi^{-1}\,\Aut(C)\,\phi$ is a subgroup of $\Aut(\mathcal{F}_6)$ for some $\phi\in\PGL_3(K)$. There is no loss of generality to assume that $\phi^{-1}\,\sigma\,\phi=[X:\zeta_6\,Z:Y]$ because all automorphisms of $\mathcal{F}_6$ of order $12$ are $\Aut(\mathcal{F}_6)$-conjugated. This reduces $\phi$ to be of the shape:
$$
\phi=\left(
      \begin{array}{ccc}
        1 & 0 & 0 \\
        0 & e & e\,\zeta_{12} \\
        0 & r & -r\,\zeta_{12} \\
      \end{array}
    \right),
$$
for some $e,\,r\in K^*$ such that $\beta_{3,3}=\dfrac{1-er\left(e^4+r^4\right)}{(er)^3}$. Since $\phi^{-1}\tau\phi$ is an involution for both $^{\phi}C$ and $\mathcal{F}_6$, then $e^4-r^4=0$ (so $r=ce$ for some $c$ such that $c^4=1$) and the transformed equation $^{\phi}C$ is
$$
X^6+Y^6+Z^6+(3-16ce^6)\zeta_{12}^4\left(Y^2Z^4+\zeta_{12}^4Y^4Z^2\right)=0,
$$

Accordingly, $^{\phi}C$ admits a larger automorphism group than $\phi^{-1}\langle\sigma,\tau\rangle\phi$ if and only if $3-16ce^6=0$. Equivalently, $\beta_{0,3}=\dfrac{10}{3c^2}=\pm\dfrac{10}{3}$), which was to be shown.
\end{proof}

%
\section{The automorphism group for Type $8,(1,3)$: The stratum $\mathcal{M}^{\operatorname{Pl}}_{g}(\Z/8\Z)$}
In this section, we will describe the full automorphism group of smooth plane sextic curves $C$ of Type $8, (1,3)$. That is, when $C$ is defined by an equation of the form:
$$C:X^6+Y^5Z+YZ^5+\beta_{4,2}X^2Y^2Z^2=0,$$
for some $\beta_{4,2}\in K$. In this situation, we have $\sigma:=\operatorname{diag}(1,\zeta_{8},\zeta_{8}^3)\in\Aut(C)$ of order $8$. We further can assume that $\sigma$ has maximal order or $C$ will be of Type $24, (1,19)$. In particular, $\beta_{4,2}\neq0$.

In this direction, we have:
\begin{prop}\label{813}
Let $C$ be a non-singular plane sextic curves $C$ of Type $8, (1,3)$ as above. Then, $\operatorname{Aut}(C)=\langle\sigma,\,\tau\rangle\cong\operatorname{SD}_{16}$, a quasidihedral group, where $\sigma:=\operatorname{diag}(1,\zeta_{8},\zeta_{8}^3)$ and $\tau:=[X:Z:Y]$.
\end{prop}

\begin{proof}
The group $\langle\sigma,\tau:\sigma^{8}=\tau^2=1,\,\tau\sigma\tau=\sigma^{3}\rangle\cong \operatorname{SD}_{16}=\operatorname{GAP}(16,8)$ is always a subgroup of automorphisms of order $16$. This implies that $\operatorname{Aut}(C)$ is not one of the finite primitive subgroups of $\operatorname{PGL}_3(K)$ mentioned in Theorem \ref{teoHarui}-(5) as it has elements of order $8>7$, not cyclic as $\operatorname{SD}_{16}$ non-cyclic, and $C$ is not a descendant of the Klein curve $\mathcal{K}_6$ or the Fermat curve $\mathcal{F}_6$ as $16\nmid|\Aut(\mathcal{K}_6)|,\,|\Aut(\mathcal{F}_6)|$. Thus, we have no choice except that $\operatorname{Aut}(C)$ fixes the line $\mathcal{L}: X=0$ in $\mathbb{P}_2(K)$ and the point $(1:0:0)$ off this line not lying on $C$. In particular, there is a short exact sequence $1\rightarrow N\rightarrow \Aut(C)\rightarrow\Lambda(\Aut(C))\rightarrow 1,$ where $N=\langle\sigma^4\rangle$ by Remark \ref{extensionrem}, and $\Lambda(\Aut(C))$ contains the dihedral group $\operatorname{D}_8$ generated by $\Lambda(\sigma)=\operatorname{diag}(\zeta_8,\zeta_8^3)$ and $\Lambda(\tau)=[Z:Y]$. Therefore, if $\Lambda(\Aut(C))$ is strictly bigger than $\operatorname{D}_8$, then it must be $\operatorname{S}_4$ by Theorem \ref{3.3}-(2).
Since $\Lambda(H):=\langle\operatorname{diag}(\zeta_4,-\zeta_4),[Z:Y]\rangle=(\Z/2\Z)^2$ is normal in $\operatorname{S}_4$, we only need to determine the conditions under which $C$ has an automorphism $\eta$ such that $\Lambda(\eta)$ is of order $3$ and belongs to the normalizer of $\Lambda(H)$. By \cite[Lemma 2.2.3-(b)]{Hugg1}, we obtain
$$
\Lambda(\eta)=\left(
                \begin{array}{cc}
                  \zeta_4^\nu & -\zeta_4^{\nu+\nu'} \\
                  1 & \zeta_4^{\nu'} \\
                \end{array}
              \right),
$$
  for some $\nu,\nu'\in\{0,1,2,3\}$. However, the polynomial $Y^5Z+YZ^5+\beta_{4,2}X^2Y^2Z^2$ is not invariant under the action of $\Lambda(\eta)$. For example, the transformed equation for $C$ will contain $Z^6$, a contradiction. Consequently, $C$ admits no more automorphisms, which was to be shown.
\end{proof}


\section{The automorphism group for Type $7,(1,3)$: The stratum $\mathcal{M}^{\operatorname{Pl}}_{g}(\Z/7\Z)$}
In this section, we will describe the full automorphism group of smooth plane sextic curves $C$ of Type $7, (1,3)$. That is, when $C$ is defined by an equation of the form:
$$C:X^5Y+Y^5Z+XZ^5+\beta_{4,2}X^2Y^2Z^2=0,$$
for some $\beta_{4,2}\in K$. We further can assume that $\sigma:=\operatorname{diag}(1,\zeta_{7},\zeta_{7}^3)$ is an automorphism of maximal order $8$. In particular, $\beta_{4,2}\neq0$ or $C$ will be of Type $21, (1,17)$.

In this case, $\Aut(C)$ is determined by the next result.
\begin{prop}\label{713}
Let $C$ be a smooth plane sextic curves $C$ of Type $7, (1,3)$ as above. Then, the full automorphism group $\operatorname{Aut}(C)$ is classified as follows.
\begin{enumerate}[1.]
  \item If $\beta_{4,2}\neq-5$, then $\operatorname{Aut}(C)=\langle\sigma,\,\tau\rangle\cong\Z/7\Z\rtimes\Z/3\Z$, the semidirect product of $\Z/7\Z$ and $\Z/3\Z$ acting faithfully, where Let $\tau:=[Y:Z:X]$.

  \item If $\beta_{4,2}=-5$, then $\operatorname{Aut}(C)=\langle\sigma,\tau,\eta\rangle\cong\operatorname{PSL}(2,7)$, where $\eta$ is the involution $$
      \left(
  \begin{array}{ccc}
    \zeta_7-\zeta_7^6 & \zeta_7^2-\zeta_7^5 & \zeta_7^4-\zeta_7^3 \\
    \zeta_7^2-\zeta_7^5 & \zeta_7^4-\zeta_7^3 & \zeta_7-\zeta_7^6 \\
    \zeta_7^4-\zeta_7^3 & \zeta_7-\zeta_7^6 & \zeta_7^2-\zeta_7^5 \\
  \end{array}
\right)
  .$$
\end{enumerate}
\end{prop}

\begin{proof}
Since $\langle\sigma,\tau:\sigma^{7}=\tau^3=1,\,\tau\sigma\tau^{-1}=\sigma^{4}\rangle\cong\Z/7\Z\rtimes\Z/3\Z=\operatorname{GAP}(21,1)$ is a subgroup of automorphisms of order $21$, then $\operatorname{Aut}(C)$ is not cyclic, it does not leave invariant a point in $\mathbb{P}_2(K)$ as $\sigma$ and $\tau$ do not share any fixed points, not one of the finite primitive subgroups of $\operatorname{PGL}_3(K)$ except possibly the Klein group $\operatorname{PSL}(2,7)$, and $C$ is not a descendant of the Fermat curve $\mathcal{F}_6$. Moreover, if $C$ is a descendant of the Klein curve $\mathcal{K}_6$, then it does not admit a bigger automorphism group unless it is $K$-isomorphic to the Klein curve itself as $\Z/7\Z\rtimes\Z/3\Z$ is a maximal subgroup of $\Aut(\mathcal{K}_6)$. This is not allowed by the assumption that automorphisms of $C$ have orders $\leq7$ because $\mathcal{K}_6$ has automorphisms of order $21$.

Now, assume that $\Aut(C)$ is $\PGL_3(K)$-conjugate to $\operatorname{PSL}(2,7)$ for some specializations of the parameter $\beta_{4,2}$. Because the centralizer of $\langle\tau\rangle$ inside $\operatorname{PSL}(2,7)$ 
is $\Z/3\Z$ and its normalizer is $\operatorname{S}_3$, there must be an involution $\eta$ of $C$ such that $\eta\tau\eta=\tau^{-1}$. Write $
\eta=\left(
       \begin{array}{ccc}
         a & b & c \\
         d & e & f \\
         g & h & r \\
       \end{array}
     \right),
$
then the condition $\eta\tau\eta=\tau^{-1}$ reads as
$
\left(
       \begin{array}{ccc}
         c & a & b \\
         f & d & e \\
         r & g & h \\
       \end{array}
     \right)=\lambda\left(
                      \begin{array}{ccc}
                        g & h & r \\
                        a & b & c \\
                        d & e & f \\
                      \end{array}
                    \right)
$
for some $\lambda\in K^*$. This holds only if $\lambda^3=1$,
$$
\eta=\left(
  \begin{array}{ccc}
    a & b & c \\
    \zeta_3^\nu b & \zeta_3^\nu c & \zeta_3^\nu a \\
    \zeta_3^{2\nu}c & \zeta_3^{2\nu}a & \zeta_3^{2\nu}b \\
  \end{array}
\right),
$$
for some $\nu\in\{0,1,2\}$. Because $\eta^2=1$, we further obtain $c=-ab\zeta_3^{2\nu}/(\zeta_3^{\nu}a+b)$ such that $ab(\zeta_3^{\nu}a+b)\neq0$ (otherwise, $\eta$ reduces to $X\leftrightarrow Z$ which is never an automorphism for $C$). Thus,
$$
\eta_{\nu}=\left(
  \begin{array}{ccc}
    a & b & -ab\zeta_3^{2\nu}/(\zeta_3^{\nu}a+b) \\
    \zeta_3^\nu b & -ab/(\zeta_3^{\nu}a+b) & a\zeta_3^\nu \\
    -ab\zeta_3^{\nu}/(\zeta_3^{\nu}a+b) & \zeta_3^{2\nu}a & \zeta_3^{2\nu}b \\
  \end{array}
\right),
$$
For invertibility, we need to impose $a^2+\zeta_3^{\nu}ab+\zeta_3^{2\nu}b^2\neq0$. If $\eta=\eta_0$, then the transformed equation $^{\eta_0}C$ would contain $Z^6$ unless $\beta_{4,2}=\dfrac{a^5b^4 + b^5(a+ b)^4-a^4(a+b)^5}{a^3b^3(a+b)^3}$. Moreover, we must eliminate the coefficients of $X^5Z,\,XY^5,$ and $Y^2Z^4$ in $^{\eta_0}C$. Accordingly,
\begin{eqnarray*}
  a^6+a^5b-5a^4b^2-5a^3b^3+5a^2b^4+5ab^5+b^6&=& 0, \\
  (a^3-3ab^2- b^3)(a^3+a^2b-2ab^2-b^3)&=& 0, \\
  (a^6-a^5b-12a^4b^2-9a^3b^3+8a^2b^4+7ab^5+b^6)(a^3+a^2b-2ab^2-b^3)&=& 0.
\end{eqnarray*}
Using MATHEMATICA, we can see that the above system holds only if $\beta_{4,2}=-5$. The work of Klein \cite{klein1879b} (see also \cite{Elkies}) assures that $a=\zeta_7-\zeta_7^6$ and $b=\zeta_7^2-\zeta_7^5$ fill the job and we are done.

We can handle the situation for $^{\eta_1}C$ and $^{\eta_2}C$ in a similar fashion.
\end{proof}


\section{The automorphism group that contain homologies of period $\geq3$}
In this section we will study $\Aut(C)$ when $C$ admits an homology of period $\geq3$. By Proposition \ref{maximalordercyclic}, $C$ is of Type $6, (0,1),\,5,(0,1), 6,(1,3)$ or $3,(0,1)$.


\subsection{Type $5,(0,1)$}\label{case501}
In this case, $C$ is defined by a smooth plane equation of the form:
$$C:Z^5Y+L_{6,Z}=0,$$
where $\sigma:=\operatorname{diag}(1,1,\zeta_5)$ is an automorphism of maximal order $5$. In particular, $\Aut(C)$ contains an homology of period $d-1=5$ with center $P=(0:0:1)$, then $P\in C$ is an inner Galois point by Proposition \ref{d,d-1 homologies}. Moreover, it is the unique inner Galois point for $C$ by \cite[Theorem 4]{Yoshihara}. Hence, $P\in C$ must be invariant under the action of $\Aut(C)$. By \cite[Lemma 11.44 and Theorem 11:49]{Book}, we deduce that $\Aut(C)$ is always cyclic of order $5$.

Therefore, we have:
\begin{prop}\label{501}
Let $C$ be a smooth plane sextic curves $C$ of Type $5, (0,1)$ as above. Then, $\operatorname{Aut}(C)=\langle\sigma\rangle\simeq\Z/5\Z$, where $\sigma:=\operatorname{diag}(1,1,\zeta_{5})$.
\end{prop}



\subsection{Type $6,(0,1)$}
In this case, $C$ is defined by a smooth plane equation of the form:
$$C:X^6+L_{6,X}=0,$$
where $\sigma:=\operatorname{diag}(\zeta_6,1,1)$ is an automorphism of maximal order $6$.

Because $\sigma$ is an homology of period $d=6$ with center $P=(1:0:0)$, then $P\notin C$ is an outer Galois point by Proposition \ref{d,d-1 homologies}. Moreover, automorphisms of $C$ have orders $\leq6$, so $C$ can not be $K$-isomorphic to the Fermat curve $\mathcal{F}_6$. Hence, by the results of Yoshihara \cite[Theorem 4', Proposition 5']{Yoshihara}, $P$ is a unique outer Galois point for $C$. Thus it is invariant under the action of $\Aut(C)$. By Theorem \ref{Mitchell1}-(1), $\Aut(C)$ also fixes the axis $\mathcal{L}:X=0$ of $\sigma$, not necessarily point-wise. In particular,
all automorphisms of $C$ are intransitive of the shape
$$
      \left(
        \begin{array}{ccc}
          1 & 0 & 0 \\
          0 & * & * \\
          0 & * & * \\
        \end{array}
      \right).
      $$
By Theorem \ref{teoHarui}, we see that either $C$ is a descendant of the Fermat curve $\mathcal{F}_6$ or $\Aut(C)$ lives in a short exact sequence
$1\rightarrow N=\langle\sigma\rangle\rightarrow \Aut(C)\rightarrow\Lambda(\Aut(C))\rightarrow 1,$ where $\Lambda(\Aut(C))$ is a cyclic group $\Z/m\Z$ of
order $m\leq 5$, a Dihedral group $\operatorname{D}_{2m}$ of order $2m$
with $m=1,2,$ or $4$, one of the alternating groups $\operatorname{A}_4$, $\operatorname{A}_5$, or the symmetry group $\operatorname{S}_4$.

In what follows, we deal with these two possibilities.
\begin{enumerate}
  \item If $C$ is a descendant of $\mathcal{F}_6$, then $\phi^{-1}\Aut(C)\phi\leq\Aut(\mathcal{F}_6)$ for some $\phi\in\PGL_3(K)$. We further can assume that $\phi^{-1}\sigma^2\phi=\sigma^2$ as homologies of order $3$ inside $\Aut(\mathcal{F}_6)$ form two conjugacy classes represented by $\sigma^2$ and $\sigma^{-2}$, which are not $\PGL_3(K)$-conjugated. In particular, $\phi$ has the same shape as automorphisms of $C$, that is, there is no loss of generality to assume that $\Aut(C)\leq\Aut(\mathcal{F}_6).$ Moreover, the shapes of automorphisms of $C$ leads to $$\Aut(C)\subseteq\{[X:\zeta_6^rY:\zeta_6^{r'}Z],\,[X:\zeta_6^{r'}Z:\zeta_6^rY]:0\leq r,r'\leq5\}.$$
      Therefore, by \cite[Lemma 6.5]{Harui}, $\Aut(C)$ lives in a short exact sequence
        $$1\rightarrow H\times\langle\operatorname{diag}(\zeta_6,1,1)\rangle\rightarrow\operatorname{Aut}(C)\rightarrow G \rightarrow1,$$
where $H\leq\langle\operatorname{diag}(1,\zeta_6,1)\rangle$ of order $\leq3$ and $G\leq\langle[X:Z:Y]\rangle$.

\begin{claim}\label{claim6,601}
If $G$ is trivial, then $\Aut(C)$ is $\Z/6\Z$ or $\Z/2\Z\times\Z/6\Z$.
\end{claim}
\begin{proof} (of Claim \ref{claim6,601})
If $H=1$, then $\Aut(C)=\langle\operatorname{diag}(\zeta_6,1,1)\rangle\simeq\Z/6\Z$, and $C$ is defined by an equation of the form
   $$X^6+Y^6+Z^6+\,\,\textit{lower order terms in}\,\, Y,Z=0.$$
If $H=\langle\operatorname{diag}(1,-1,1)\rangle$, then $\Aut(C)=\langle\operatorname{diag}(\zeta_6,1,1),\,\operatorname{diag}(1,-1,1)\rangle\simeq\Z/2\Z\times\Z/6\Z$, and $C$ is defined by an equation of the form
   $$X^6+Y^6+Z^6+\beta_{4,2}Y^4Z^2+\beta_{2,4}Y^2Z^4=0,$$
   for some $\beta_{4,2}\neq\beta_{2,4}$.

If $H=\langle\operatorname{diag}(1,\zeta_3,1)$, then $\Aut(C)=\langle\operatorname{diag}(\zeta_6,1,1),\,\operatorname{diag}(1,\zeta_3,1)\rangle$, and $C$ is defined by an equation of the form
   $$X^6+Y^6+Z^6+\beta_{3,3}Y^3Z^3=0.$$
   However, we reject this case as $[X:Z:Y]\in\Aut(C)$, which violates the assumption $G=1$.

This shows Claim \ref{claim6,601}.
\end{proof}

\begin{claim}\label{claim7,601}
If $G=\langle[X:Z:Y]\rangle$, then $\Aut(C)$ is $\Z/2\Z\times\Z/6\Z$ or $\Z/6\Z\times\operatorname{S}_3$.
\end{claim}
\begin{proof} (of Claim \ref{claim7,601})
If $H=1$, then $\Aut(C)=
   \langle\operatorname{diag}(\zeta_6,1,1),[X:Z:Y]\rangle\simeq\Z/2\Z\times\Z/6\Z$, and $C$ is defined by an equation of the form
   $$X^6+Y^6+Z^6+\beta_{5,1}(Y^{5}Z+YZ^{5})+\beta_{4,2}(Y^{4}Z^2+Y^2Z^{4})+\beta_{3,3}Y^3Z^3=0,$$
such that $\beta_{5,1}\neq0$ or $\beta_{4,2}\neq0$. 

If $H=\langle\operatorname{diag}(1,-1,1)\rangle$, then $\Aut(C)=\langle\operatorname{diag}(\zeta_6,1,1),\,\operatorname{diag}(1,-1,1),\,[X:Z:Y]\rangle$, and we reject this case as $[\zeta_6X:-Z:Y]$ would be an automorphism of order $12>6$, a contradiction.

If $H=\langle\operatorname{diag}(1,\zeta_3,1)$, then $\Aut(C)=\langle\operatorname{diag}(\zeta_6,1,1),\,\operatorname{diag}(1,\zeta_3,1),\,[X:Z:Y]\rangle$, and $C$ is defined by an equation of the form
   $$X^6+Y^6+Z^6+\beta_{3,3}Y^3Z^3=0,$$
   for some $\beta_{3,3}\neq0$. It is easy to check that
   \begin{eqnarray*}
     \Aut(C) &=& \langle\sigma,\tau',\eta:\sigma^6=\tau'^2=\eta^3=1,\,\sigma\eta=\eta\sigma,\,\sigma\tau'=\tau'\sigma,\,\tau'\eta\tau'=\eta^{-1}\rangle \\
      &\simeq&\operatorname{GAP}(36,12)=\Z/6\Z\times\operatorname{S}_3.
   \end{eqnarray*}
   with $\tau':=[X:Z:Y]$ and $\eta:=\operatorname{diag}(1,\zeta_3,1).$

This shows Claim \ref{claim7,601}.
\end{proof}

  \item Otherwise, $\Aut(C)$ lives in a short exact sequence
$$1\rightarrow N=\langle\sigma\rangle\rightarrow \Aut(C)\rightarrow\Lambda(\Aut(C))\rightarrow 1,$$ where $\Lambda(\Aut(C))$ is a cyclic group $\Z/m\Z$ of
order $m\leq 5$, a Dihedral group $\operatorname{D}_{2m}$ of order $2m$
with $m=1,2,$ or $4$, one of the alternating groups $\operatorname{A}_4$, $\operatorname{A}_5$, or the symmetry group $\operatorname{S}_4$.
\begin{claim}\label{claim1601}
$\Lambda(\Aut(C))$ is either trivial or $\Z/2\Z$.
\end{claim}

\begin{proof}(of Claim \ref{claim1601})
If $\Lambda(\Aut(C))\simeq\Z/4\Z,\,\operatorname{D}_8$ or $\operatorname{S}_4$, then, by \cite[Lemma 2.2.1, I]{Hugg1}, we can assume that $\langle\Lambda(\eta):=\operatorname{diag}(1,\zeta_4)\rangle\leq\Lambda(\Aut(C))$ for some $\eta\in\Aut(C)$. Since $L_{6,X}$ is invariant under the action of $\Lambda(\eta)$, $C$ reduces up to $K$-isomorphism to $X^6+\beta_{2,4}Y^2Z^4=0,$ which is singular at the point $(0:0:1)$, a contradiction.

If $\Lambda(\Aut(C))\simeq\Z/3\Z$, then, by a similar argument as before, $C$ would admit an automorphism $\eta$ whose image $\Lambda(\eta)=\operatorname{diag}(1,\zeta_3)$ leaves invariant $L_{6,X}$. This reduces $C$, up to $K$-isomorphism, to $C:X^6+Y^6+Z^6+\beta_{3,4}Y^3Z^3=0.$ Clearly $[Z:Y]$ is another element in $\Lambda(\Aut(C))$ of order $2\nmid 3$, a contradiction.

If $\Lambda(\Aut(C))\simeq\Z/5\Z$ or $\operatorname{A}_5$, then $C$ would have an automorphism of order $15>5$ as $N=\langle\sigma\rangle$ is normal in $\Aut(C)$ and $\Aut(C)$ contains element of order $5$ with $\operatorname{gcd}(|N|,5)=1$. This contradicts the assumption that automorphisms of $C$ are of orders $\leq6$.

Finally, if $\Lambda(\Aut(C))\simeq\operatorname{D}_4$ or $\operatorname{A}_4$, then, without loss of generality using \cite[Lemma 2.2.1, I]{Hugg1}, we can assume that $L_{6,X}$ is invariant under the action of $\Lambda(\tau)=\operatorname{diag}(1,-1)$ and $\Lambda(\eta)=[Z:Y]$ for some $\tau,\eta\in\Aut(C)$. This reduces $C$ up to $K$-isomorphism to $C:X^6+Y^6+Z^6+\beta_{2,4}\left(Y^2Z^4+Y^4Z^2\right)=0,$
       where $\tau=\operatorname{diag}(1,\lambda,-\lambda)$  and $\eta=[X:\mu Z:\mu Y]$ for some $\lambda,\mu$ such that $\lambda^6=\mu^6=1$. This in turns implies that
  \begin{eqnarray*}
    \Aut(C) &\supseteq&\left\{\operatorname{diag}(1,\zeta_6^{r},\pm\zeta_6^{r}),\,[X:\pm\zeta_{6}^{r}Z:\zeta_{6}^rY]\,:\,r=0,1,\cdots,5\right\}\\
    &=&\langle\sigma,\tau,\eta:\sigma^6=\tau^2=\eta^2=1,\,\sigma\tau=\tau\sigma,\sigma\eta=\eta\sigma,(\eta\tau)^2=\sigma^3\rangle \\
    &\simeq& \operatorname{GAP}(24,10)=\Z/3\Z\times\operatorname{D}_8,
  \end{eqnarray*}
  with $\tau=\operatorname{diag}(1,1,-1)$ and $\eta=[X:Z:Y]$. Accordingly, $C$ would have automorphisms of order $12>6$, which is not allowed by assumption.
%

This proves Claim \ref{claim1601}.
\end{proof}

\begin{claim}\label{claim2601}
If $\Lambda(\Aut(C))\simeq\Z/2\Z$, then $\Aut(C)\simeq\Z/2\Z\times\Z/6\Z$.
\end{claim}

\begin{proof}(of Claim \ref{claim2601})
By the aid of \cite[Lemma 2.2.1, I]{Hugg1} we can assume that $L_{6,X}$ is leaved invariant under the action of $\Lambda(\eta)=\operatorname{diag}(1,-1)$ for some $\eta\in\Aut(C)$. This reduces $C$, up to $K$-isomorphism, to $C:X^6+Y^6+Z^6+\beta_{2,4}Y^2Z^4+\beta_{4,2}Y^4Z^2=0,$
      where $\eta=\operatorname{diag}(1,\lambda,-\lambda)$ for some $\lambda\in K$ such that $\lambda^6=1$. This in turns implies that all automorphisms of $C$ are diagonal, more precisely,
  \begin{eqnarray*}
    \Aut(C) &=&\langle\sigma,\tau:\sigma^6=\tau^2=1,\,\sigma\tau=\tau\sigma\rangle \\
    &\simeq& \operatorname{GAP}(12,5)=\Z/2\Z\times\Z/6\Z,
  \end{eqnarray*}
  with $\tau:=\operatorname{diag}(1,1,-1)$. This shows Claim \ref{claim2601}.
  \end{proof}
\end{enumerate}

As a consequence of the previous discussion, we conclude:
\begin{prop}\label{aut601}
Let $C$ be a smooth plane sextic curves $C$ of Type $6, (0,1)$ as above. Then, $\Aut(C)$ is always cyclic of order $6$ generated by $\sigma=\operatorname{diag}(\zeta_6,1,1)$ unless one of the following situations holds.
\begin{enumerate}[1.]
  \item $C$ is $K$-isomorphic to $X^6+Y^6+Z^6+\beta_{2,4}Y^2Z^4+\beta_{4,2}Y^4Z^2=0$ for some $\beta_{2,4}\neq\beta_{4,2}$. In this case, $\Aut(C)=\langle\sigma,\tau\rangle\simeq\Z/2\Z\times\Z/6\Z$ with $\tau:=\operatorname{diag}(1,1,-1)$.
  \item $C$ is $K$-isomorphic to $X^6+Y^6+Z^6+\beta_{5,1}(Y^{5}Z+YZ^{5})+\beta_{4,2}(Y^{4}Z^2+Y^2Z^{4})+\beta_{3,3}Y^3Z^3=0,$ such that $\beta_{5,1}\neq0$ or $\beta_{4,2}\neq0$. In this case, $\Aut(C)=\langle\sigma,\tau'\rangle\simeq\Z/2\Z\times\Z/6\Z$ with $\tau':=[X:Z:Y]$.
  \item $C$ is $K$-isomorphic to $X^6+Y^6+Z^6+\beta_{3,3}Y^3Z^3=0,$ for some $\beta_{3,3}\neq0$. In this case, $\Aut(C)=\langle\sigma,\tau',\eta\rangle\simeq\Z/6\Z\times\operatorname{S}_3$ with $\eta:=\operatorname{diag}(1,\zeta_3,1)$.
\end{enumerate}
\end{prop}

\begin{rem}
  We remark that the above two families in Proposition \ref{aut601}-(1), (2) are $K$-equivalent via a change of variables in $\PGL_3(K)$ of the form $$\phi=[\lambda X:\mu(Y+Z):Y-Z]$$ for some $\lambda,\mu\in K$ such that $\lambda^6=1$ and $\mu^4+\beta_{4,2}\mu^2+\beta_{2,4}=0$. So any curve in any of the families is $K$-isomorphic to a member in the other family and vice versa.
\end{rem}



\subsection{Type $6,(1,3)$}\label{case613}
In this case, $C$ is defined by a smooth plane equation of the form:
$$C:X^6+Y^6+Z^6+\beta_{2,0}X^4Z^2+\beta_{0,3}Y^3Z^3+ X^2\left(\beta_{4,0}Z^4+\beta_{4,3}Y^3Z\right),$$
where $\sigma:=\operatorname{diag}(1,\zeta_6,-1)$ is an automorphism of maximal order $6$. In particular, $\sigma^2=\operatorname{diag}(1,\zeta_3,1)$ is an homology of period $3$ in $\Aut(C)$. Then, by Theorem \ref{Mitchell1}, $\Aut(C)$ fixes a line and a point off this line, fixes a triangle, or $\Aut(C)$ is $\PGL_3(K)$-conjugate to $\operatorname{Hess}_{216}$.

\begin{claim}\label{claim1,613}
  If $\Aut(C)$ fixes a triangle and neither a line nor a point is leaved invariant, then $\Aut(C)\simeq\Z/6\Z,\,\Z/2\Z\times\Z/6\Z,\,\Z/3\Z\times\operatorname{S}_3,$\,$\Z/6\Z\times\operatorname{S}_3$ or $(\Z/3\Z)^2\rtimes\Z/6\Z$.
\end{claim}

\begin{proof}(of Claim \ref{claim1,613})
  If $\Aut(C)$ fixes a triangle and neither a line nor a point is leaved invariant, then $C$ is a descendant of the Fermat curve $\mathcal{F}_6$ or the Klein curve $\mathcal{K}_6$. Clearly, $C$ is not a descendant of $\mathcal{K}_6$ because $6\nmid|\Aut(\mathcal{K}_6)|\,(=63)$. On the other hand, if $C$ is a descendant of $\mathcal{F}_6$, then, by Proposition \ref{1217}, we can assume that $\phi^{-1}\Aut(C)\phi\leq\Aut(\mathcal{F}_6)(\simeq(\Z/6\Z)^2\rtimes\operatorname{S}_3)$ for some $\phi\in\PGL_3(K)$ such that $\phi^{-1}\sigma^2\phi=\sigma^2$ as homologies of period $3$ in $\Aut(\mathcal{F}_6)$ form two conjugacy classes represented by $\sigma^2$ and $\sigma^{-2}$, but $\sigma^2$ and $\sigma^{-2}$ are not $\PGL_3(K)$-conjugated. Since $\phi^{-1}\sigma\phi$ is an element of order $6$ inside $\Aut(\mathcal{F}_6)$, then $\phi=\operatorname{diag}(a,1,b),\,[a Z:Y:b X]$ or $[a X+b Z:Y:c X-\dfrac{bc}{a}Z]$.

  \begin{itemize}
    \item Assume that $\phi=\operatorname{diag}(a,1,b)$ or $[a Z:Y:b X]$. In particular, $a^6=b^6=1$ and $C$ is $K$-equivalent to
      $$
      C':X^6+Y^6+Z^6+\beta_{2,0}'X^4Z^2+\beta_{0,3}'Y^3Z^3+ X^2\left(\beta_{4,0}'Z^4+\beta_{4,3}'Y^3Z\right)=0.
      $$
Now, recall that the automorphisms of $\mathcal{F}_6$ are
$\tau_1:=[X:\zeta_{6}^rY:\zeta_{6}^{r'}Z],\,\tau_{2}:=[\zeta_{6}^{r'}Z:\zeta_{6}^rY:X],\,\tau_{3}:=[X:\zeta_{6}^{r'}Z:\zeta_{6}^rY],\,\tau_{4}:=[\zeta_{6}^rY:X:\zeta_{6}^{r'}Z],\,\tau_{5}:=[\zeta_{6}^rY:\zeta_{6}^{r'}Z:X],\,\tau_{6}:=[\zeta_{6}^{r'}Z:X:\zeta_{6}^rY]$ for $0\leq r\leq5$.

Obviously, $\tau_{4,5,6}$ are never automorphisms of $C'$. Secondly, $\tau_3\in\Aut(C')$ only when $\beta_{2,0}'=\beta_{4,0}'=\beta_{4,3}'=0$. In this case, $C$ is $K$-isomorphic to $X^6+Y^6+Z^6+\beta_{0,3}'Y^3Z^3=0$, hence $\Aut(C)\simeq\Z/6\Z\times\operatorname{S}_3$ by Proposition \ref{aut601}.
Thirdly, $\tau_2\in\Aut(C')$ only when $\beta_{0,3}'=\beta_{4,3}'=0$ and $\beta_{4,0}'=\beta_{2,0}'\zeta_6^{2(r'-r)}$. In this case, $C$ is $K$-isomorphic, via a transformation of the shape $\operatorname{diag}(1,\lambda,\mu)$ with $\lambda^4\mu'^2=\lambda^2\mu^4\zeta_6^{2(r'-r)}$ and $\lambda^6=1$, to $X^6+Y^6+Z^6+\beta_{2,0}X^2Z^2(X^2+Z^2)=0$. But in this case, $[Z:\zeta_6 Y:-X]$ is an automorphism of order $12>6$, a contradiction.
    \item Assume that $\phi=[a X+b Z:Y:c X-\dfrac{bc}{a}Z]$. In particular, $\beta_{2,0}=\dfrac{1-a^6-c^6-a^2c^4}{a^4c^2},\,b=\zeta_6^{n}a$ and $-2ac=\zeta_6^{n'}$, 
    and $C$ is $K$-equivalent via $\phi$ then re-scaling $Z\mapsto\zeta_6^{-n}Z$ to
\begin{eqnarray*}
     C':X^6+Y^6+Z^6&+&\alpha_1'XZ\left(X^4+Z^4\right)+\alpha_2'X^2Z^2\left(X^2+Z^2\right)+\alpha_3'X^3Z^3\\
     &+&\alpha_4'\left(X^3+Z^3\right)Y^3+\alpha_5'XY^3Z\left(X-Z\right)=0,
     \end{eqnarray*}
where
\begin{eqnarray*}
\alpha_1'&:=&\dfrac{(-2a^4\beta_{4,0}\zeta_6^{4n'}+32a^{12}+16a^6-1)}{8a^6},\\
\alpha_2'&:=&\dfrac{(64a^{12}-4a^6+1)}{4a^6},\\
\alpha_3'&:=&\dfrac{(2a^4\beta_{4,0}\zeta_6^{4n'}+96a^{12}-16a^6-1)}{4a^6},\\
\alpha_4'&:=&-\dfrac{(4a^4\beta _{4,3}+\beta _{6,3}\zeta^{2n'})\zeta_6^{n'}}{8a^3},\\
\alpha_5'&:=&-\dfrac{(4a^4\beta_{4,3}-3\beta _{6,3}\zeta^{2n'})\zeta_6^{n'}}{8a^3},
\end{eqnarray*}
Here $[Z:\zeta_{6}^{1-n-n'}Y:X]$ should be an automorphism for $C'$ of order $6$, consequently, $\alpha_4'=0$ when $2\mid n+n'$ and $\alpha_5'=0$ when $2\nmid n+n'$.

- If $\left(\alpha_1',\alpha_2',\alpha_5'\right)=(0,0,0)$ and $\alpha_4'=0$, then $\Aut(C)$ is again $\Z/6\Z\times\operatorname{S}_3$ by Proposition \ref{aut601}-(3).

- If $\left(\alpha_1',\alpha_2',\alpha_5'\right)=(0,0,0)$ and $\alpha_4'\neq0$, then $\operatorname{diag}(1,\zeta_3,1)$, $\operatorname{diag}(1,1,\zeta_3),$ $[Z:Y:X]$ are obvious automorphisms for $C'$. Equivalently, $\operatorname{GAP}(18,3)=\Z/3\Z\times\operatorname{S}_3$ is always a subgroup of $\Aut(C')$. Moreover, if $\alpha_3'=0$, then $\tau_3,\tau_4,\tau_5,\tau_6$ are not automorphisms for $C'$, so $\Aut(C')$ is exactly $\Z/3\Z\times\operatorname{S}_3$.
But if $\alpha_3'\neq0$, then it is easy to see that
\begin{eqnarray*}
\Aut(C')&\subseteq&\langle\operatorname{diag}(1,\zeta_3,1),\operatorname{diag}(1,1,\zeta_3),[Z:Y:X],[Z:X:Y]\rangle\\
&\simeq&\operatorname{GAP}(54,5)=(\Z/3\Z)^2\rtimes\Z/6\Z.
\end{eqnarray*}
We remark that $\Z/3\Z\times\operatorname{S}_3$ is maximal in $(\Z/3\Z)^2\rtimes\Z/6\Z$, hence $\Aut(C')$ is either $\Z/3\Z\times\operatorname{S}_3$ or $(\Z/3\Z)^2\rtimes\Z/6\Z$. Furthermore, $\Aut(C')=(\Z/3\Z)^2\rtimes\Z/6\Z$ if and only if $[Z:X:Y]\in\Aut(C')$, which holds if and only if $\alpha_3'=\alpha_4'$.

- If $\left(\alpha_1',\alpha_2',\alpha_5'\right)\neq(0,0,0)$, then $\tau_3,\tau_4,\tau_5$ and $\tau_6$ are never automorphisms for $C'$. If we first assume that $\alpha_4'\neq0$ (so $\alpha_5'=0$), then
$$
\Aut(C')\subseteq\langle\operatorname{diag}(1,\zeta_3,1),\operatorname{diag}(1,1,\zeta_3),[Z:Y:X]\rangle\simeq\Z/3\Z\times\operatorname{S}_3.
$$
Since $\Z/6\Z$ is maximal in $\Z/3\Z\times\operatorname{S}_3$ and $\operatorname{diag}(1,1,\zeta_3)$, $[Z:Y:X]$ are in $\Aut(C')$ only if $\alpha_1'=\alpha_2'=0$, then $\Aut(C')=\langle[Z:\zeta_3Y:X]\rangle=\Z/6\Z$ in this situation. Secondly, if $\alpha_5'\neq0$ (so $\alpha_4'=0$), then
\begin{eqnarray*}
\Aut(C')&\subseteq&\left\{I,\operatorname{diag}(1,\zeta_6^{\pm2},1),[Z:sY:X]\,:\,s=-1,\zeta_6,-\zeta_6^2\right\}\\
&=&\langle[Z:\zeta_6Y:X]\rangle.
\end{eqnarray*}
Therefore, $\Aut(C')$ is again $\Z/6\Z$ generated by $[Z:\zeta_6Y:X]$. Thirdly, if $\alpha_4'=\alpha_5'=0$ and $(\alpha_1',\alpha_3')\neq(0,0)$, then $\operatorname{diag}(1,\zeta_6,1)$ becomes an extra automorphism for $C'$, more precisely,
$$\Aut(C')=\langle\operatorname{diag}(1,\zeta_6,1),[Z:Y:X]\rangle=\Z/2\Z\times\Z/6\Z.$$
Finally, the situation $\alpha_4'=\alpha_5'=\alpha_1'=\alpha_3'=0$ is absurd as $[Z:\zeta_6Y:-X]$ will be an automorphism of order $12>6$.
  \end{itemize}

This shows Claim \ref{claim1,613}, equivalently, Proposition \ref{3.3}-(2), (3) below.
\end{proof}

\begin{claim}\label{claim2,613}
  If $\Aut(C)$ fixes a line $\mathcal{L}$ and a point $P$ not lying on $\mathcal{L}$, then $\Aut(Z)\simeq\Z/6\Z$.
\end{claim}

\begin{proof}(of Claim \ref{claim2,613})
  Since $\sigma$ is a non-homology inside $\Aut(C)$ in its canonical form, then $\mathcal{L}$ must be one of the reference lines $B=0$ with $B=X,Y$ or $Z$ and $P$ one of the reference points $(1:0:0),\,(0:1:0)$, or $(0:0:1)$ respectively. Consequently, all automorphisms of $C$ are all intransitive of exactly one the following shapes:
  $$
 \operatorname{Shape\,I:} \left(
  \begin{array}{ccc}
    1 & 0 & 0 \\
    0 & \ast & \ast \\
    0 & \ast & \ast \\
  \end{array}
\right),\, \operatorname{Shape\,II:}
\left(
  \begin{array}{ccc}
    \ast & 0 & \ast \\
    0 & 1 & 0 \\
    \ast & 0 & \ast \\
  \end{array}
\right),\, \operatorname{Shape\,III:} \left(
  \begin{array}{ccc}
    \ast & \ast & 0 \\
    \ast & \ast & 0 \\
    0 & 0 & 1 \\
  \end{array}
\right).
$$
Also, $\Aut(C)$ lives in a short exact sequence
$$1\rightarrow N\rightarrow \Aut(C)\rightarrow\Lambda(\Aut(C))\rightarrow 1,$$ where $N$ is a cyclic group of order dividing $d=6$, and $\Lambda(\Aut(C))$ is a cyclic group $\Z/m\Z$ of order $m\leq d-1$, a Dihedral group $\operatorname{D}_{2m}$ of order $2m$ with $|N|=1$ or
$m=1,2$ or $4$, one of the alternating groups $\operatorname{A}_4$, $\operatorname{A}_5$, or the symmetry group $\operatorname{S}_4$.

Up to projective equivalence via a permutation of two of the variables, we can assume that all automorphisms of $C$ are of the shape I without any loss of generality. In this case, $\langle\sigma^3\rangle$ is a subgroup of $N$ of order $2$ and $\Lambda(\Aut(C))$ contains $\Lambda(\sigma)=[\zeta_3^{-1}Y,Z]$ of order $3$. So, $\Lambda(\Aut(C))\neq\operatorname{D}_{2m}$ as $3\nmid\,2m$ if $m|4$. On the other hand, if $\Lambda(\Aut(C))=\Z/m\Z$ then $m=3$ and $\Aut(C)=\langle\sigma\rangle=\Z/6\Z$ as claimed. Moreover, if $\Lambda(\Aut(C))=\operatorname{A}_5$, then $\Aut(C)$ should contain an automorphism of order $10>6$, which is absurd by the assumption that $\sigma$ has maximal order in $\Aut(C)$. It remains to prove that $\Lambda(\Aut(C))$ can not be $\operatorname{S}_4$ or $\operatorname{A}_4$.
\begin{enumerate}
\item If $\Lambda(\Aut(C))=\operatorname{S}_4$, then $\Lambda(\Aut(C))$ should contain an element $\Lambda(\tau)$ of order $2$ such that $\Lambda(\tau)\Lambda(\sigma)\Lambda(\tau)=\Lambda(\sigma)^{-1}$ because $\Z/3\Z$ in $\operatorname{S}_4$ is equal to its centralizer and has an $\operatorname{S}_3$ as its normalizer. Hence $\Lambda(\tau)=[cZ:bY]$ for some $b,c\in K^*$, which holds only when $\beta_{2,0}=\beta_{4,0}=\beta_{4,3}=0$ and $C$ is $K$-isomorphic to $C:X^6+Y^6+Z^6+\beta_{0,3}Y^3Z^3=0$ (in particular, $b^6=c^6=(bc)^3=1$). Furthermore, the centralizer of $\langle\Lambda(\tau)\rangle$ in $\operatorname{S}_4$ is $(\Z/2\Z)^2$, thus there must be another element $\Lambda(\tau')\notin\langle\Lambda(\sigma),\Lambda(\tau)\rangle$ of order $2$ that commutes with $\Lambda(\tau)$. This implies that $\Lambda(\tau')=[Y:-Z]$, which leaves invariant $Y^6+Z^6+\beta_{0,3}Y^3Z^3$ if and only if $\beta_{0,3}=0$, equivalently, when $C$ is the Fermat curve $\mathcal{F}_6$, a contradiction.

\item If $\Lambda(\Aut(C))=\operatorname{A}_4$, then $\Aut(C)$ has a subgroup $G$ of order $24$ that contains $H=\langle\sigma^3\rangle$ as a normal subgroup such that the quotient $G/H$ equals $\operatorname{A}_4$. Going through \href{https://people.maths.bris.ac.uk/~matyd/GroupNames/1/C3sC8.html}{Group of order $24$} we can see that $G$ is either $\operatorname{SL}_2(\mathbb{F}_3)=\operatorname{GAP}(24,3)$ or $\Z/2\Z\times\operatorname{A}_4=\operatorname{GAP}(24,13)$.
\begin{itemize}
\item For $G$ to be $\Z/2\Z\times\operatorname{A}_4$, $C$ should have an involution $\tau$ that commutes with $\sigma^3$ as any $\Z/2\Z$ inside $\Z/2\Z\times\operatorname{A}_4$ is contained in a $(\Z/2\Z)^2$. This yields $\tau$ of the form $\operatorname{diag}(1,\pm1,\mp1),\,$ $[X:\pm Y:sY\mp Z]$ or $[X:sY+tZ:\dfrac{1-s^2}{t}Y-tZ]$ for some $s,t\in K$.

For $\tau=\operatorname{diag}(1,\pm1,\mp1)$ to be an automorphism, $\beta_{0,3}$ and $\beta_{4,3}$ must be $0$. In this situation, $\Aut(C)$ contains $\Z/2\Z\times\Z/6\Z$ not $\Z/2\Z\times\operatorname{A}_4$, by Proposition \ref{aut601}, which is absurd.

For $\tau=[X:\pm Y:sY\mp Z]$ to be an automorphism, $s$ must be $0$ or $YZ^5$ will appear in the transformed equation under the action of $\tau$. This brings us back to the previous case when $\operatorname{diag}(1,\pm1,\mp1)\in\Aut(C)$.

We also discard $\tau=[X:sY+tZ:\dfrac{1-s^2}{t}Y-tZ]\in\Aut(C)$ for the following reasons. The transformed equation has the term $\dfrac{\beta_{2,0}}{t^2}\left((s^2-1)Y+t^2Z\right)^2X^4$, in particular, $\beta_{2,0}=0$ or $s^2=t^2=1$. The latter situation is rejected because eliminating $YZ^5$ leads to eliminating $Z^6$ from the transformed equation (we note that both monomials appear with coefficient $\pm(\beta_{0,3}-2)$). Thus $\beta_{2,0}$ must be $0$, moreover, the coefficients of $YZ^5$ and $Z^6$ implies that $s\neq0$ and $t=\dfrac{1-s^2}{s}$. The action now does not give the monomial $Y^3Z^3$, so $\beta_{0,3}=0$ and the transformed equation has $Y^2Z^4$ with non-zero coefficient, a contradiction.
\item For $G$ to be $\operatorname{SL}_2(\mathbb{F}_3)$, $C$ should have an automorphism $\tau$ of order $4$ such that $\tau^2=\sigma^3$. This gives us $\tau$ of the form $\operatorname{diag}(\pm\zeta_4,1,1),\,$ $[X:\pm\zeta_4Y:sY\mp\zeta_4Z]$ or $[X:sY+tZ:-\dfrac{s^2+1}{t}Y-tZ]$ for some $s,t\in K$.

   One easily checks that neither $\operatorname{diag}(\pm\zeta_4,1,1)$ nor $[X:\pm\zeta_4Y:sY\mp\zeta_4Z]$ preserves the core $X^6+Y^6+Z^6$ of the defining equation for $C$. On the other hand, $\tau=[X:sY+tZ:-\dfrac{s^2+1}{t}Y-tZ]\in\Aut(C)$ only if $s\neq0$ and $t=-\left(\dfrac{1+s^2}{s}\right)$ (we remark that $YZ^5$ will appear otherwise). To eliminate the monomial $X^4Y^2$, $\beta_{2,0}=0$, and to eliminate $X^2YZ^3$ and $X^2Z^4$, $\beta_{4,0}=\beta_{4,3}=0$. In particular, again by Proposition \ref{aut601}, $\Aut(C)$ contains $\Z/6\Z\times\operatorname{S}_3$ not $\operatorname{SL}_2(\mathbb{F}_3)$, which is not valid.

\end{itemize}

We finally deduce that $\Lambda(\Aut(C))\neq\operatorname{A}_4$.
\end{enumerate}
This proves Claim \ref{claim2,613}.
\end{proof}

Regarding the Hessian group $\operatorname{Hess}_{216}$, its representations inside $\operatorname{PGL}_3(\overline{k})$ are unique
 up to conjugation, see H. Mitchell \cite[p. $217$]{Mit} for more details. For example, we choose
$\operatorname{Hess}_{216}=\langle S,T,U,V\rangle$, where
$$S=\operatorname{diag}(1,\zeta_3,\zeta_3^{-1}),U=\operatorname{diag}(1,1,\zeta_3),\ V=\left(
\begin{array}{ccc}
1&1&1\\
1&\zeta_3&\zeta_3^{-1}\\
1&\zeta_3^{-1}&\zeta_3\\
\end{array}
\right),\ \ T=[Y:Z:X].$$
Also, we consider the primitive Hessian subgroup of order 36,
$\operatorname{Hess}_{36}=\langle S,T,V\rangle$, and the
primitive Hessian subgroup of order 72, $\operatorname{Hess}_{72}=\langle
S,T,V,UVU^{-1}\rangle$.

Now, suppose that $\Aut(C)$ is conjugate to $\operatorname{Hess}_{216}$, then we may assume that $\phi^{-1}\Aut(C)\phi=\operatorname{Hess}_{216}$ such that $\phi^{-1}\sigma\phi=[\zeta_3Z:Y:\zeta_3X]$ or $[\zeta_3^{-1}Z:Y:\zeta_3^{-1}X]$ as any group of order $6$ inside $\operatorname{Hess}_{216}$ is conjugated to $\langle[\zeta_3Z:Y:\zeta_3X]\rangle$. This imposes
    $$
    \phi=\left(
      \begin{array}{ccc}
         a & 0 & -a \\
         0 & 1 & 0 \\
         b & 0 & b \\
       \end{array}
     \right)
    $$
with $ab\neq0$. The transformed equation under the action of $\phi$ should be invariant under any permutation of the variables $X,Y,Z$ as $\langle[X:Z:Y],[Y:Z:X]\rangle$ becomes a subgroup of automorphisms. For instance, the coefficients of $Z^6,\,X^6$ and $XZ^5$ in $C'$ must be $1,1$ and $0$ respectively. So, $\beta_{2,0}=\frac{1-4a^6+2b^6}{2a^4b^2}$ and $\beta_{4,0}=\frac{1+2a^6-4b^6}{2a^2b^4}$. Comparing the coefficients of $X^3Z^3$ and $Y^3Z^3$ yields $\beta_{4,3}=-\frac{32a^6-32b^6+b^3\beta_{0,3}}{a^2 b}$. Next, we force the coefficients of $XY^3Z^2$ and $X^2Z^4$ to be $0$, equivalently, $\beta_{0,3}=-\frac{8\left(a^6-b^6\right)}{b^3}$ and $16a^6+16b^6-1=0$. This reduces $C'$ to
$$
X^6+Y^6+Z^6-2(32a^6-1)(X^3Y^3+Y^3Z^3+Z^3X^3)=0.
$$
It remains to guarantee that $V\in\Aut(C')$. This holds only when $16a^6=3$ and $16b^6=-2$, so $\beta_{0,3}=\frac{-5}{2b^3},\,\beta_{4,3}=\frac{-15}{2a^2 b},\,\beta_{4,0}=\frac{15}{16a^2b^4},\,\beta_{2,0}=0$. This justifies Proposition \ref{3.3}- (2) below.


Summing up, we obtain:
\begin{prop}\label{3.3}
Let $C$ be a smooth plane sextic curves $C$ of Type $6, (1,3)$ as above. Then, $\Aut(C)$ is classified as follows:
\begin{enumerate}[1.]
\item If $\beta_{0,3}=\frac{-5}{2b^3},\,\beta_{4,3}=\frac{-15}{2a^2 b},\,\beta_{4,0}=\frac{15}{16a^2b^4},\,\beta_{2,0}=0$ with $16a^2=3$ and $16b^2=-2$, then $C$ is $K$-isomorphic to $$X^6+Y^6+Z^6-10(X^3Y^3+Y^3Z^3+Z^3X^3)=0.$$ In this situation, $\Aut(C)$ is conjugate to $\operatorname{Hess}_{216}$.
\item If $\beta_{2,0}=\beta_{4,0}=\beta_{4,3}=0$, then $C$ is defined by $X^6+Y^6+Z^6+\beta_{0,3}Y^3Z^3=0,$ a descendant of the Fermat curve. In particular, $\Aut(C)=\langle\sigma,\eta,\tau\rangle\simeq\Z/6\Z\times\operatorname{S}_3$ as in Proposition \ref{aut601}.
\item If $\beta_{2,0}=\dfrac{-64a^{12}+60a^6-1}{16a^8}$, then $C$ is $K$-isomorphic to
\begin{eqnarray*}
     C':X^6+Y^6+Z^6&+&\alpha_1'XZ\left(X^4+Z^4\right)+\alpha_2'X^2Z^2\left(X^2+Z^2\right)+\alpha_3'X^3Z^3\\
     &+&\alpha_4'\left(X^3+Z^3\right)Y^3+\alpha_5'XY^3Z\left(X-Z\right)=0.
     \end{eqnarray*}
In this situation, $\Aut(C')$ is classified as follows.
\begin{enumerate}[(i)]
  \item If $\left(\alpha_1',\alpha_2',\alpha_5'\right)=(0,0,0)$ and $\alpha_4'=0$, then $\Aut(C')$ is $\Z/6\Z\times\operatorname{S}_3$ as in Proposition \ref{aut601}-(3).
  \item If $\left(\alpha_1',\alpha_2',\alpha_5'\right)=(0,0,0),\,$ $\alpha_4'\neq0$, and $\alpha_3'=0$ or $\alpha_3'\neq0$ such that $\alpha_3'\neq\alpha_4'$, then $\Aut(C')$ is $\Z/3\Z\times\operatorname{S}_3$, generated by $\operatorname{diag(1,\zeta_3,1)},\,$ $\operatorname{diag}(1,1,\zeta_3)$ and $[Z:Y:X]$.
  \item If $\left(\alpha_1',\alpha_2',\alpha_5'\right)=(0,0,0)$ and $\alpha_4'=\alpha_3'\neq0$, then $\Aut(C')$ is $(\Z/3\Z)^2\rtimes\Z/6\Z$, generated by $\operatorname{diag(1,\zeta_3,1)},\,$ $\operatorname{diag}(1,1,\zeta_3),\,$ $[Z:Y:X]$ and $[Z:X:Y]$.
   \item If $\left(\alpha_1',\alpha_2',\alpha_5'\right)\neq(0,0,0)$ such that $\alpha_4'=\alpha_5'=0$, then $\left(\alpha_1',\alpha_3'\right)\neq(0,0)$ and $\Aut(C')$ is $\Z/2\Z\times\Z/6\Z$, generated by $\operatorname{diag}(1,\zeta_6,1)$ and $[Z:Y:X]$. In particular, $C$ is $K$-isomorphic to a member in the family described in Proposition \ref{aut601}-(1).
   \item If $\left(\alpha_1',\alpha_2',\alpha_5'\right)\neq(0,0,0)$ such that $\alpha_4'\neq0$ or $\alpha_5'\neq0$, then $\Aut(C')$ is $\Z/6\Z$, generated by $[Z:\zeta_3Y:X]$ when $\alpha_4'\neq0$ and by $[Z:\zeta_6Y:X]$ when $\alpha_5'\neq0$.
\end{enumerate}
  \item Otherwise, $\Aut(C)=\Z/6\Z$ generated by $\sigma=\operatorname{diag}(1,\zeta_6,-1)$.
\end{enumerate}
\end{prop}

\begin{rem}
It is known in the literature that $$X^6+Y^6+Z^6-10(X^3Y^3+Y^3Z^3+Z^3X^3)=0$$ has $\operatorname{Hess}_{216}$ as its automorphism group (cf. \cite[\s 2]{Harui}). However, our result shows that it is the unique smooth plane sextic curve with this property.
\end{rem}

\subsection{Type $3,(0,1)$}\label{case301}
In this case, $C$ is defined by a smooth plane equation of the form:
$$C:Z^6+Z^3L_{3,Z}+L_{6,Z}=0,$$
where $\sigma:=\operatorname{diag}(1,1,\zeta_3)$ is an automorphism of maximal order $3$. In particular, $L_{3,Z}\neq0$ or $\operatorname{diag}(1,1,\zeta_6)$ will be an automorphism of order $6>3$. Moreover, by smoothness, $L_{6,Z}$ has degree $\geq5$ in both $X$ and $Y$.

The automorphism group of $C$ is determined by the next result, which is a direct consequence of \cite[Theorem 2.3]{BadrBarspreprint}.
\begin{prop}\label{301}
Let $C$ be a smooth plane sextic curve of Type $3, (0,1)$ as above. Then, $\Aut(C)$ is always cyclic generated by $\sigma$ except when $C$ is $K$-isomorphic to $C'$ of the form
$$
C':X^6+Y^6+Z^6+Z^3\left(\beta_{3,0}X^3+\beta_{0,3}Y^3\right)+\,\beta_{3,3}X^3Y^3=0,
$$
such that $\beta_{3,0},\beta_{0,3},\beta_{3,3}$ are pair-wise distinct modulo $\{\pm1\}$. In this case, $\Aut(C')$ equals $(\Z/3\Z)^2$ generated by $\sigma$ and $\operatorname{diag}(1,\zeta_3,1)$.
\end{prop}
%

\section{Remaining types}
In this section, we investigate the automorphism group of smooth plane curves $C$ of one of the following types: Type $6,(1,2)$, Type $5, (1,2)$, Type $5, (1,4)$, Type $4, (1,3)$, Type $3, (1,2)$, or Type $2, (0,1)$.

\subsection{Type $6,(1,2)$}
In this situation, $C$ is defined by a smooth plane equation of the form:
$$C:X^6+Y^6+Z^6+\beta_{3,0}X^3Z^3+\beta_{2,2}X^2Y^2Z^2+\beta_{1,4}XY^4Z=0,$$
where $\sigma:=\operatorname{diag}(1,\zeta_6,\zeta_6^2)$ is an automorphism of maximal order $6$.

It is obvious that $\tau:=[Z:Y:X]$ is always an automorphism for $C$, in particular, $\Aut(C)$ is never cyclic. Also, $C$ is not a descendant of the Klein curve $\mathcal{K}_6$ as $6\nmid|\Aut(\mathcal{K}_6)| (=63)$. Moreover, $\Aut(C)$ is never $\PGL_3(K)$-conjugate to the Klein group $\operatorname{PSL}(2,7)$ because the latter contains elements of order $7>6$. Also, $\Aut(C)$ is not conjugate to $A_5,\,A_6,\,\operatorname{Hess}_{36},\,$ or $\operatorname{Hess}_{72}$ as they do not have elements of order $6$ inside. Thus we conclude that $\Aut(C)$ fixes a line and a point off this line, conjugate to $\operatorname{Hess}_{216}$, or $C$ is a descendant of the Fermat curve $\mathcal{F}_6$.

\begin{claim}\label{claim1,612}
  $\Aut(C)$ is never conjugate to $\operatorname{Hess}_{216}$.
\end{claim}

\begin{proof}
All copies of $\Z/6\Z$ inside $\operatorname{Hess}_{216}$ are conjugated, see \href{https://people.maths.bris.ac.uk/~matyd/GroupNames/193/ASL(2,3).html}{Hessian group of order $216$}. Fix the copy generated by to $\rho:=[X:Z:\zeta_3Y]$. In particular, $\rho^2$ is an homology of period $3$, hence $\langle\sigma\rangle$ can not be $\PGL_3(K)$-conjugate to $\langle\rho\rangle$ because $\sigma^2=\operatorname{diag}(1,\zeta_3,\zeta_3^{-1})$ is a non-homology.

This shows Claim \ref{claim1,612}.
\end{proof}

\begin{claim}\label{claim2,612}
If $C$ is a descendant of $\mathcal{F}_6$, then $\Aut(C)$ is $\operatorname{D}_{12},\,\Z/3\Z\rtimes\operatorname{S}_4,$ or $\Z/6\Z\times\operatorname{S}_3$.
\end{claim}

\begin{proof}
We can assume that $\phi^{-1}\Aut(C)\phi\leq\Aut(\mathcal{F}_6)$ for some $\phi\in\PGL_3(K)$ such that $\phi^{-1}\sigma^2\phi=\sigma^{2}$ since non-homologies of period $3$ that are obtained from an automorphism of order $6$ inside $\Aut(\mathcal{F}_6)$ form a single conjugacy class represented by $\sigma^2$. In particular, $\phi=\operatorname{diag}(1,a,b),\,$ $[bZ:X:aY]$ or $[aY:bZ:X]$ with $a^6=b^6=1$. In all situations, $C$ is $K$-isomorphic to
$$C':X^6+Y^6+Z^6+\beta_{3,0}'X^3Z^3+\beta_{2,2}'X^2Y^2Z^2+\beta_{1,4}'XY^4Z=0,$$
where $\left(\beta_{3,0}',\beta_{2,2}',\beta_{1,4}'\right)$ equals $\left(\beta_{3,0}b^3,\beta_{2,2}(ab)^2,\beta_{1,4}a^4b\right)$.
Next, we follow the notations in Section \ref{case613} to determine when $C'$ may admit a bigger automorphism subgroup of $\Aut(\mathcal{F}_6)$ than $\langle\sigma,\tau\rangle\simeq\operatorname{GAP}(12,4)=\operatorname{D}_{12}$.

- First, $\tau_{5,6}$ is an automorphism for $C'$ only if $\beta_{3,0}'=\beta_{1,4}'=0$. In such a case $\Aut(C)$ is the semidirect product of $\Z/3\Z$ and $\operatorname{S}_4$ provided that $\beta_{2,2}'\neq0$. Indeed, it is generated by $\sigma,\,\tau,\,\tau':=\operatorname{diag}(1,1,-1)$ and $\eta:=[Y:Z:X]$, where   $\sigma^6=\tau^2=\tau'^2=\eta^3=1,\,\tau\sigma\tau=\sigma^{-1},\,\sigma\tau'=\tau'\sigma,\,\eta\sigma\eta^{-1}=\sigma\tau,\,\tau\tau'\tau=\sigma^3\tau',\,\tau\eta\tau=\eta^{-1},\,\eta\tau'\eta^{-1}=\sigma^3$. Therefore, $\Aut(C)\simeq\operatorname{GAP}(72,43)=\Z/3\Z\rtimes\operatorname{S}_4$.

- Secondly, if $\beta_{3,0}'\neq0$ or $\beta_{1,4}'\neq0$, then $\tau_{3,4,5,6}$ are never an automorphism for $C'$. More precisely, if $\beta_{3,0}'\neq0$ and $\beta_{1,4}'=\beta_{2,2}'=0$, then $C$ becomes $K$-isomorphic, via a permutation $X\leftrightarrow Y$, to a curve in the family of Proposition \ref{aut601}-(3), in particular, $\Aut(C)\simeq\Z/6\Z\times\operatorname{S}_3$. Otherwise, $\beta_{1,4}'\neq0$ or $\beta_{2,2}'\neq 0$ and $\Aut(C)=\langle\sigma,\tau\rangle\simeq\operatorname{D}_{12}$.

This shows Claim \ref{claim2,612}.
\end{proof}

\begin{claim}\label{claim3,612}
  If $\Aut(C)$ fixes a line $\mathcal{L}$ and a point $P$ not lying on $\mathcal{L}$, then $\Aut(C)$ is $\operatorname{D}_{12}$ or $\Z/3\Z\rtimes\operatorname{S}_4.$
\end{claim}

\begin{proof}
Because $\sigma,\tau\in\Aut(C)$, then $\mathcal{L}:Y=0$ and $P=(0:1:0)$. So, all automorphisms of $C$ are intransitive of the shape:
  $$
\left(
  \begin{array}{ccc}
    \ast & 0 & \ast \\
    0 & 1 & 0 \\
    \ast & 0 & \ast \\
  \end{array}
\right).
$$
Also, we can think about $\Aut(C)$ in a short exact sequence
$$1\rightarrow N\rightarrow \Aut(C)\rightarrow\Lambda(\Aut(C))\rightarrow 1,$$ where $N$ is a cyclic group of order dividing $6$, and $\Lambda(\Aut(C))$ is a cyclic group $\Z/m\Z$ of order $m\leq d-1$, a Dihedral group $\operatorname{D}_{2m}$ of order $2m$ with $|N|=1$ or
$m=1,2$ or $4$, one of the alternating groups $\operatorname{A}_4$, $\operatorname{A}_5$, or the symmetry group $\operatorname{S}_4$.

We have $\langle\sigma^3\rangle\leq N$ of order $2$, and $\Lambda(\Aut(C))$ contains an $\operatorname{S}_3$ generated by $\Lambda(\sigma)=[\zeta_3X,\zeta_3^{-1}Z]$ of order $3$ and $\Lambda(\tau)=[Z:X]$ of order $2$. Then, $\Lambda(\Aut(C))\neq\Z/m\Z$ as $\operatorname{S}_3$ is not cyclic, $\Lambda(\Aut(C))\neq\operatorname{D}_{2m}$ as $3\nmid 2m$ if $m|4$, $\Lambda(\Aut(C))\neq A_4$ as $\operatorname{A}_4$ doesn't have an $\operatorname{S}_3$ as a subgroup, $\Lambda(\Aut(C))\neq\operatorname{A}_5$ as $\Aut(C)$ would contain an automorphism of order $10>6$ otherwise. On the other hand, if $\Lambda(\Aut(C))=\operatorname{S}_4$, then $\Lambda(\Aut(C))$ should contain an element $\Lambda(\tau')\notin\langle\Lambda(\sigma),\Lambda(\tau)\rangle$ of order $2$ such that $\Lambda(\tau)\Lambda(\tau')=\Lambda(\tau')\Lambda(\tau)$ since $\langle\Lambda(\tau)\rangle=\Z/2\Z$ in $\operatorname{S}_4$ has centralizer $(\Z/2\Z)^2$. This implies that $\Lambda(\tau')=[X:-Z]$, equivalently, $\tau'=\operatorname{diag}(\lambda,1,-\lambda)$. Such an automorphism exists if and only if $\beta_{3,0}=\beta_{1,4}=0$. In such a case, we reduce to the previous scenario when $\Aut(C)$ was a semidirect product of $\Z/3\Z$ and $\operatorname{S}_4$.

This shows Claim \ref{claim3,612}.
\end{proof}

By the previous discussion, we conclude:
\begin{prop}\label{3.4}
Let $C$ be a smooth plane sextic curves $C$ of Type $6, (1,2)$ as above. Then, $\Aut(C)$ is classified as follows:
\begin{enumerate}[1.]
\item If $\beta_{3,0}=\beta_{1,4}=0$, then $C$ is given by $X^6+Y^6+Z^6+\beta_{2,2}X^2Y^2Z^2=0$ for some $\beta_{2,2}\neq0$, a descendant of the Fermat curve. In this situation, $\Aut(C)=\langle\sigma,\tau,\tau',\eta\rangle\simeq\Z/3\Z\rtimes\operatorname{S}_4$ with $\tau=[Z:Y:X],\,\tau'=\operatorname{diag}(1,1,-1)$ and $\eta=[Y:Z:X]$.
\item If $\beta_{3,0}\neq0, \beta_{1,4}=\beta_{2,2}=0$, then $\Aut(C)\simeq\Z/6\Z\times\operatorname{S}_3$ as in Proposition \ref{aut601}-(3).
  \item Otherwise, $\Aut(C)=\langle\sigma,\tau\rangle\simeq\operatorname{D}_{12}$.

\end{enumerate}
\end{prop}


\subsection{Type $5,(1,4)$}\label{case514}
In this situation, $C$ is defined by an equation of the form:
$$C:X^6+XZ^5+XY^5+\beta_{4,1}X^4YZ+\beta_{2,2}X^2Y^{2}Z^{2}+\beta_{0,3}Y^3Z^3=0,$$
where $\sigma:=\operatorname{diag}(1,\zeta_5,\zeta_5^{-1})$ is an automorphism of maximal order $5$.

Clearly, $\langle\sigma,\tau\rangle\simeq\operatorname{D}_{10}$ with $\tau:=[X:Z:Y]$ is always a subgroup of automorphisms for $C$, in particular, $\Aut(C)$ is never cyclic. Also $C$ is neither a descendant of the Fermat curve $\mathcal{F}_6)$ nor the Klein curve $\mathcal{K}_6$ as $10\nmid|\Aut(\mathcal{F}_6)|(=216)$ and $10\nmid|\Aut(\mathcal{K}_6)|(=63)$. For the same reason, $\Aut(C)$ is not conjugate to $\operatorname{Hess}_{*}$, for $*\in\{36,72,216\}$, or to $\operatorname{PSL}(2,7)$. That is, $\Aut(C)$ fixes a line and a point off this line or it is $\PGL_3(K)$-conjugate to $\operatorname{A}_5$ or $\operatorname{A}_6$.

\begin{claim}\label{claim1,514}
If $\Aut(C)$ fixes a line $\mathcal{L}$ and a point $P$ off this line, then $\Aut(C)=\langle\sigma,\tau\rangle\simeq\operatorname{D}_{10}$.
\end{claim}

\begin{proof}
Since As $\sigma,\,\tau\in\Aut(C)$, then $\mathcal{L}:X=0$ and $P=(1:0:0)$. Thus all automorphisms of $C$ are intransitive of the shape:
  $$
\left(
  \begin{array}{ccc}
    1 & 0 & 0 \\
    0 & \ast & \ast \\
    0 & \ast & \ast \\
  \end{array}
\right).
$$
Also, $\Aut(C)$ would satisfy a short exact sequence
$$1\rightarrow N=1\rightarrow \Aut(C)\rightarrow\Lambda(\Aut(C))\rightarrow 1,$$ where $\Lambda(\Aut(C))$ is a cyclic group $\Z/m\Z$ of order $m\leq d-1$, a Dihedral group $\operatorname{D}_{2m}$ of order $2m$, one of the alternating groups $\operatorname{A}_4$, $\operatorname{A}_5$, or the symmetry group $\operatorname{S}_4$.

We note that $\Lambda(\Aut(C))$ contains $\langle\Lambda(\sigma),\Lambda(\tau)\rangle=\operatorname{D}_{10}$. Hence, $\Lambda(\Aut(C))\neq\Z/m\Z,\,\operatorname{A}_4,$ or $\operatorname{S}_4$. Furthermore, if $\Lambda(\Aut(C))=\operatorname{A}_{5}$, then there must be $\Lambda(\tau')\notin\langle\Lambda(\sigma),\Lambda(\tau)\rangle$ of order two that commutes with $\Lambda(\tau)$ as any copy of $\Z/2\Z$ inside $\operatorname{A}_5$ has centralizer $(\Z/2\Z)^2.$ This leads to $\Lambda(\tau')=[Y:-Z]$, equivalently, $\tau'=\operatorname{diag}(1,\lambda,-\lambda)$ for some $\lambda\in K^*$. The presence of the monomials $XZ^5$ and $XY^5$ prevents the existence of such an automorphism, so $\Lambda(\Aut(C))\neq\operatorname{A}_5$.

Lastly, if $\Lambda(\Aut(C))=\operatorname{D}_{2m}$, then $m=5$ by the maximality of the order of $\sigma$.

This proves Claim \ref{claim1,514}.
\end{proof}

Now, suppose that $\Aut(C)$ is conjugate to $\operatorname{A}_5$ or $\operatorname{A}_6$. In either way, $\Aut(C)$ has another involution $\tau'$ that commutes with $\tau$ as any $\Z/2\Z$ inside $\operatorname{A}_5$ or $\operatorname{A}_6$ is contained in a $(\Z/2\Z)^2$. By \cite[Lemma 2.16]{DoiIdei}, we obtain that this involution has the shape:
$$\tau'_{\gamma}=\left(
                                                                                                  \begin{array}{ccc}
                                                                                                    1 & 2/\gamma & 2/\gamma \\
                                                                                                    \gamma & (-1+\sqrt{5})/2 & (-1-\sqrt{5})/2 \\
                                                                                                    \gamma & (-1-\sqrt{5})/2 & (-1+\sqrt{5})/2 \\
                                                                                                  \end{array}
                                                                                                \right),$$ where $\beta_{4,1}=\dfrac{12-\gamma^5}{\gamma^2},\,\beta_{2,2}=\dfrac{48+\gamma^5}{\gamma^4},$ and $\beta_{0,3}=\dfrac{64-2\gamma^5}{\gamma^6}$ for some $\gamma\in K^*.$
That is, the defining equation reduces to
$$
C_{\gamma}:X^6+X(Y^5+Z^5)+\dfrac{12-\gamma^5}{\gamma^2}X^4YZ+\dfrac{48+\gamma^5}{\gamma^4}X^2Y^2Z^2+\dfrac{64-2\gamma^5}{\gamma^6}Y^3Z^3=0,
$$
Next, suppose that $\Aut(C)$ is equal to $\operatorname{A}_6$. By \cite[Thoerem 2.1]{DoiIdei}, $C$ must be $K$-equivalent to the Wiman sextic curve $\mathcal{W}_6$:
$$
\mathcal{W}_6:27X^6+9XZ^5+9XY^5-135X^4YZ-45X^2Y^{2}Z^{2}+10Y^3Z^3=0.
$$
All the copies of $\operatorname{D}_{10}$ inside $\operatorname{A}_6$ are conjugated. So, there is no loss of generality to assume that $^{\phi}C_{\gamma}$ is the Wiman sextic curve via $\phi\in\PGL_3(K)$ such that $\phi^{-1}\langle\sigma,\,\tau\rangle\phi=\langle \sigma,\,\tau\rangle$ as $\langle\sigma,\,\tau\rangle\leq\operatorname{Aut}(\mathcal{W}_6).$ That is, $\phi$ has the shape $\operatorname{diag}(\lambda,\zeta_5^j,1)$ or $[\lambda X:\zeta_5^j Z:Y]$ for some $j=0,1,2,3,4$. In either
way, the transformed equation of $C$ became
 \begin{eqnarray*}
    ^{\phi}C:\lambda^6X^6&+&\lambda XZ^5+\lambda XY^5+\dfrac{12-\gamma^5}{\gamma^2}\lambda^4 \zeta_5^j X^4YZ+\dfrac{48+\gamma^5}{\gamma^4}\lambda^2 \zeta_5^{2j}X^2Y^{2}Z^{2} \\
    &+& \dfrac{64-2\gamma^5}{\gamma^6}\zeta_5^{3j}Y^3Z^3=0.
 \end{eqnarray*}
In particular, we should have
$$
\lambda^6=27\nu,\,\lambda=9\nu,\,\dfrac{12-\gamma^5}{\gamma^2}\lambda^4\zeta_5^j=-135\nu,\,\dfrac{64-2\gamma^5}{\gamma^6}\zeta_5^{3j}=10\nu,\,\dfrac{48+\gamma^5}{\gamma^4}\lambda^2\zeta_5^{2j}=-45\nu,\\
$$
for some $\nu\in K^*.$ Equivalently, 
$$\dfrac{12-\gamma^5}{\gamma^2}=\dfrac{-15}{(\sqrt[5]{3})^3}\zeta_5^{-(3i+j)},\,\dfrac{48+\gamma^5}{\gamma^4}=\dfrac{-5}{\sqrt[5]{3}}\zeta_5^{-(i+2j)},\,\dfrac{64-2\gamma^5}{\gamma^6}=\dfrac{10\sqrt[5]{3}}{9}\zeta_5^{i-3j}$$ for $i,j=0,1,2,3,4$. Eventually, the automorphism group is $R^{-1}\,\langle T_1,\,T_2,\,T_3,\,T_4\rangle\,R$ as illustrated in Appendix  \ref{appB}.

Lastly, we remark that $$\tau':=\phi^{-1}\,\tau'_{\gamma}\,\phi=\left(
                                                                                                  \begin{array}{ccc}
                                                                                                    1 & 1 & 1 \\
                                                                                                    2 & (-1+\sqrt{5})/2 & (-1-\sqrt{5})/2 \\
                                                                                                    2 & (-1-\sqrt{5})/2 & (-1+\sqrt{5})/2 \\
                                                                                                  \end{array}
                                                                                                \right),$$
with $\phi=\operatorname{diag}(2/\gamma,1,1)$. Moreover, that $\phi$ is in the normalizer of $\langle\sigma,\tau\rangle$ in $\PGL_3(K)$. In this case, $C$ is defined by an equation of the form:
$$
32X^6+\gamma^5X(Y^5+Z^5)+8(12-\gamma^5)X^4YZ+2(48+\gamma^5)X^2Y^2Z^2+(32-\gamma^5)Y^3Z^3=0,
$$
with $\langle\sigma,\tau,\tau'\rangle\leq\Aut(C)$. Because $\operatorname{A}_5$ is the only subgroup of $\operatorname{A}_6$ that contains $\operatorname{D}_{10}$ properly, then $\langle\sigma,\tau,\tau'\rangle\simeq\operatorname{A}_5$.

From the above argument, we deduce:
\begin{prop}\label{aut514}
Let $C$ be a smooth plane sextic curves $C$ of Type $5, (1,4)$ as above. Then, $\Aut(C)$ is classified as follows:
\begin{enumerate}[1.]
\item If $\beta_{4,1}=\dfrac{-15}{(\sqrt[5]{3})^3}\zeta_5^{-(3i+j)},\,\beta_{2,2}=\dfrac{-5}{\sqrt[5]{3}}\zeta_5^{-(i+2j)},\,\beta_{0,3}=\dfrac{10\sqrt[5]{3}}{9}\zeta_5^{i-3j}$ for some $i,j\in\{0,1,2,3,4\}$, then $C$ is $K$-isomorphic to the Wiman sextic curve $$27X^6+9XZ^5+9XY^5-135X^4YZ-45X^2Y^{2}Z^{2}+10Y^3Z^3=0.$$ In this situation, $\Aut(C)$ is $\PGL_3(K)$ conjugate to $\operatorname{A}_6=R^{-1}\langle T_1,\,T_2,\,T_3,\,T_4\rangle R$.
\item If $\beta_{4,1}=\dfrac{12-\gamma^5}{\gamma^2},\,\beta_{2,2}=\dfrac{48+\gamma^5}{\gamma^4},\,\beta_{0,3}=\dfrac{64-2\gamma^5}{\gamma^6}$ for some $\gamma\in K^*$ such that $\beta_{2,1}\neq\dfrac{-15}{(\sqrt[5]{3})^3}\zeta_5^{-(3i+j)},\,\beta_{4,2}\neq\dfrac{-5}{\sqrt[5]{3}}\zeta_5^{-(i+2j)}$ or $\beta_{0,3}\neq\dfrac{10\sqrt[5]{3}}{9}\zeta_5^{i-3j}$ for some $i,j\in\{0,1,2,3,4\}$, then $C$ is $K$-isomorphic to
    $$
32X^6+\gamma^5X(Y^5+Z^5)+8(12-\gamma^5)X^4YZ+2(48+\gamma^5)X^2Y^2Z^2+(32-\gamma^5)Y^3Z^3=0,
$$
where $\Aut(C)=\langle\sigma,\tau,\tau'\rangle\simeq\operatorname{A}_5$.
\item Otherwise, $\Aut(C)=\langle\sigma,\tau\rangle\simeq\operatorname{D}_{10}$.
\end{enumerate}
\end{prop}

\subsection{Type $5,(1,2)$}\label{case512}
In this situation, $C$ is defined by an equation of the form:
$$C:X^6+XZ^5+XY^5+\beta_{3,1}X^3YZ^2+\beta_{2,3}X^2Y^3Z+\beta_{0,2}Y^2Z^4=0,$$
where $\sigma:=\operatorname{diag}(1,\zeta_5,\zeta_5^2)$ is an automorphism of maximal order $5$.

Since $5\nmid|\Aut(\mathcal{F}_6)|(=216)$ and $5\nmid|\Aut(\mathcal{K}_6)|(=63)$, then $C$ is neither a descendant of the Fermat curve $\mathcal{F}_6$ nor the Klein curve $\mathcal{K}_6$.

If $\Aut(C)$ is conjugate to one of the finite primitive subgroups of $\PGL_3(K)$, then it should be $\operatorname{A}_5$ or $\operatorname{A}_6$ as these are the only ones that have elements of order $5$. Hence, the normalizer of $\langle\sigma\rangle$ in $\Aut(C)$ is a $\operatorname{D}_5$, so there must be an automorphism $\tau$ for $C$ of order $2$ such that $\tau\sigma\tau\in\langle\sigma\rangle$. It is straightforward to see that such an involution does not exist. That is, $\Aut(C)$ is not one of the primitive subgroups of $\PGL_3(K)$.

Now, assume that $\Aut(C)$ fixes a line $\mathcal{L}$ and a point $P$ not lying on $\mathcal{L}$. Then, $\mathcal{L}:X=0$ and $P=(1:0:0)$ as $\sigma$ is a non-homology in $\Aut(C)$ in its canonical form. Consequently, all automorphisms of $C$ are intransitive of the shape:
  $$
\left(
  \begin{array}{ccc}
    1 & 0 & 0 \\
    0 & \ast & \ast \\
    0 & \ast & \ast \\
  \end{array}
\right).
$$
Following the notations of Theorem \ref{teoHarui}, $N=1$ and the image $\Lambda(\Aut(C))$ of $\Aut(C)$ in $\PGL_2(K)$ contains a $\Z/5\Z$ generated by $\Lambda(\sigma)=[\zeta_5Y:\zeta_5^2Z]$. In particular, $\Lambda(\Aut(C))\neq\operatorname{A}_4$ or $\operatorname{S}_4$. In other words, $\Lambda(\Aut(C))$ is either $\Z/5\Z,\,\operatorname{D}_{10}$ or $\operatorname{A}_{5}$ (recall that, if $\Lambda(\Aut)(C)=\operatorname{D}_{2m}$, then $m=5$ by the maximality of the order of $\sigma$). Now, if $\Lambda(\Aut(C))\neq\Z/5\Z$, then there must be $\Lambda(\tau)$ of order two in the normalizer of $\langle\Lambda(\sigma)\rangle$ as any copy of $\Z/5\Z$ inside $\operatorname{D}_{10}$ (respectively $\operatorname{A}_5$) has normalizer $\operatorname{D}_{10}$. By \cite[Lemma 2.2.3-(a)]{Hugg1}, we have
$\Lambda(\tau)=[\lambda Z:Y]$ for some $\lambda\in K^*$. Equivalently, $\tau=[X:\lambda\mu Z:\mu Y]\in\Aut(C)$ for some $\lambda,\mu\in K^*$. This holds only when $\beta_{3,1}=\beta_{4,3}=\beta_{6,2}=0$, which violates the smoothness of $C$. Therefore, $\Aut(C)$ must be $\Z/5\Z$ in this scenario.

As a consequence of the previous discussion, we obtain:
\begin{prop}\label{512}
Let $C$ be a smooth plane sextic curves $C$ of Type $5, (1,2)$ as above. Then, $\Aut(C)$ is always cyclic of order $5$, generated by $\sigma=\operatorname{diag}(1,\zeta_5,\zeta_5^2)$.
\end{prop}


\subsection{Type $4,(1,3)$}\label{case413}
In this situation, $C$ is defined by a smooth plane equation of the form:
$$C:X^6+Y^5Z+YZ^5+\beta_{0,3}Y^3Z^3+\beta_{4,1}X^4YZ+X^2\left(\beta_{2,0}Z^4+\beta_{2,2}Y^2Z^2+\beta_{2,4}Y^4\right)=0,$$
where $\sigma:=\operatorname{diag}(1,\zeta_4,\zeta_4^{-1})$ is an automorphism of maximal order $4$.

- Obviously, $C$ is not a descendant of $\mathcal{K}_6$ as $4\nmid|\Aut(\mathcal{K}_6)|(=63)$. On the other hand, if $C$ is a descendant of the Fermat curve $\mathcal{F}_6$ with a bigger automorphism group than $\langle\sigma\rangle$, then $\Aut(C)=\operatorname{D}_8$ or $\operatorname{S}_4$ since these are the only subgroups of $\Aut(\mathcal{F}_6)$ that contains $\Z/4\Z$ as a subgroup and its elements are of orders $\leq4$. In both situations, $C$ admits an involution $\tau$ such that $\tau\sigma\tau=\sigma^{-1}$. One easy checks that $\tau=[X:\zeta_4^i Z:\zeta_4^{-i}Y]$ for some $i\in\{0,1,2,3\}$, and that $\beta_{2,0}=\beta_{2,4}$ is necessary. So far, $\operatorname{D}_8$ is a subgroup of automorphisms for $C$ generated by $\sigma$ and $\tau=[X:Z:Y]$ provided that $\beta_{2,0}=\beta_{2,4}$. Moreover, if $\Aut(C)$ is exactly an $\operatorname{S}_4$, then the \href{https://people.maths.bris.ac.uk/~matyd/GroupNames/1/S4.html}{Subgroups Lattice of $\operatorname{S}_4$} assures that any $\Z/4\Z$ inside $\Aut(\mathcal{F}_6)$ is $\Aut(\mathcal{F}_6)$-conjugate to $\langle[X:Z:-Y]\rangle$. Therefore, there exists a change of variables $\phi\in\operatorname{PGL}_3(K)$ that transforms $C$ with $\beta_{2,0}=\beta_{2,4}$ to $X^6+Y^6+Z^6+...$ such that $\phi^{-1}\sigma\phi=[X:Z:-Y]$. This yields two possibilities for $\phi$:
$$
\phi_1=\left(
       \begin{array}{ccc}
         1 & 0 & 0 \\
         0 & \zeta_4s & s \\
         0 & \zeta_4^{-1}t & t \\
       \end{array}
     \right)\,\,\,\text{or}\,\,\,\phi_2=\left(
       \begin{array}{ccc}
         s' & 1 & -1 \\
         0 & \zeta_4s & s \\
         0 & \zeta_4^{-1}t & t \\
       \end{array}
     \right).
$$
- For $\phi_1$, to guarantee that $\phi_1^{-1}\tau\phi_1$ is an order $2$ automorphism of $\mathcal{F}_6$, it must be the case that $s=ct$ for some $c$ such that $c^4=1$. Second, to get the core $X^6+Y^6+Z^6$, we should impose $\beta_{0,3}=\dfrac{1-2ct^6}{\pm ct^6}$. In this case, the transformed equation of $C$ becomes
\begin{eqnarray*}
X^6+Y^6+Z^6&+&t^4\left(2 \beta _{2,0}\pm\beta_{2,2}\right)X^2\left(Y^4+Z^4\right)+\left(3-16 c t^6\right) \left(Y^4Z^2+Y^2Z^4\right)\\
&+&ct^2\beta _{4,1}X^4\left(Y^2+Z^2\right)+2t^4 \left(\pm\beta _{2,2}-6 \beta_{2,0}\right)X^2Y^2Z^2=0
\end{eqnarray*}
where $\langle[X:Z:-Y],\,\operatorname{diag}(1,1,-1)\rangle=\operatorname{D}_8$ is a subgroup of automorphisms inside $\Aut(\mathcal{F}_6)$. To reach an $\operatorname{S}_4$, it suffices to guarantee an extra automorphism of order $3$ from $\Aut(\mathcal{F}_6)$ as $\operatorname{D}_8$ is maximal inside $\operatorname{S}_4$. Now, there is no loss of generality to impose such an element to be $[Y:Z:X]$ as any $\operatorname{S}_4$ inside $\Aut(\mathcal{F}_6)$ is conjugated to the one generated by $\operatorname{diag}(1,-1,1),\,\operatorname{diag}(1,1,-1),\,[X:Z:Y],$ and $[Y:Z:X]$.
This holds only if $\beta_{4,1}=\dfrac{3-16 c t^6}{ct^2}$ and $\beta_{2,0}=(\mp/2)\beta_{2,2}+\dfrac{3-16 c t^6}{2t^4}$, equivalently only if $C$ is $K$-isomorphic to
$$
X^6+Y^6+Z^6+\beta_1(X^2Y^4+Y^2Z^4+X^4Z^2+X^2Z^4+X^4Y^2+Y^4Z^2)+\beta_2X^2Y^2Z^2=0.
$$
for some $\beta_1,\beta_2\in K$.

- For $\phi_2$, one sees that
$$
\phi_2^{-1}\tau\phi_2=\left(
       \begin{array}{ccc}
         1 & \dfrac{(1+\zeta_4)(t+s)(t-\zeta_4s)}{2sts'} & -\dfrac{(1+\zeta_4)(s-t)(s+\zeta_4t)}{2sts'} \\
         0 & -\dfrac{s^2+t^2}{2st} & \zeta_4\dfrac{s^2-t^2}{2st} \\
         0 & \zeta_4\dfrac{s^2-t^2}{2st} & \dfrac{s^2+t^2}{2st} \\
       \end{array}
     \right),
$$
in particular, we must impose $s=-\zeta_4t$ to guarantee that $\phi_2^{-1}\tau\phi_2\in\Aut(\mathcal{F}_6)$. Actually, $\phi_2^{-1}\tau\phi_2$ becomes $[X:Z:Y]$. Again to get the core $X^6+Y^6+Z^6$, we should force $\beta_{0,3}=(s'/t)^6\zeta_4+2$. In this case, the transformed equation of $C$ becomes
\begin{eqnarray*}
X^6+Y^6+Z^6&+&(t/s')^4\left(2 \beta _{2,0}-\beta_{2,2}\right)X^2\left(Y^4+Z^4\right)+\left(3+16(t/s')^6\zeta_4\right)\left(Y^4Z^2+Y^2Z^4\right)\\
&-&\zeta_4(t/s')^2\beta _{4,1}X^4\left(Y^2+Z^2\right)-2(t/s')^4 \left(\beta _{2,2}+6 \beta_{2,0}\right)X^2Y^2Z^2=0
\end{eqnarray*}
where $\langle[X:Z:-Y],\,\operatorname{diag}(1,1,-1)\rangle=\operatorname{D}_8$ is a subgroup of automorphisms inside $\Aut(\mathcal{F}_6)$. A similar argument allows us to reach an $\operatorname{S}_4$ by forcing $[Y:Z:X]$ to be an extra automorphism, which holds true only if $\beta _{4,1}=
   -\frac{16 t^4}{s_1^4}+\frac{3 i s_1^2}{t^2}$ and $\beta _{2,0}=\frac{1}{2} \left(\beta _{2,2}+\frac{3
   s_1^4}{t^4}+\frac{16 i t^2}{s_1^2}\right),$ equivalently only if $C$ is $K$-isomorphic to
$$
X^6+Y^6+Z^6+\beta_1(X^2Y^4+Y^2Z^4+X^4Z^2+X^2Z^4+X^4Y^2+Y^4Z^2)+\beta_2X^2Y^2Z^2=0.
$$
for some $\beta_1,\beta_2\in K$ as before.

- Now, suppose that $\Aut(C)$ fixes a line $\mathcal{L}$ and a point $P\notin C$ off this line $\mathcal{L}$ as in Theorem \ref{teoHarui}-(2). Hence $\mathcal{L}:X=0$ and $P=(1:0:0)$ because $\sigma\in\Aut(C)$ is a non-homology in its canonical form. In particular, all automorphisms of $C$ are intransitive of the shape:
$$
\left(
  \begin{array}{ccc}
    1 & 0 & 0 \\
    0 & \ast & \ast \\
    0 & \ast & \ast \\
  \end{array}
\right)
$$
and we can think about $\Aut(C)$ in a short exact sequence $$1\rightarrow N=\langle\sigma^2\rangle=\Z/2\Z\rightarrow \Aut(C)\rightarrow\Lambda(\Aut(C))\rightarrow 1,$$ where $\Lambda(\Aut(C))$ contains a $\Z/2\Z$ generated by $\Lambda(\sigma)=[\zeta_4Y:\zeta_4^{-1}Z]$. Thus, by Theorem \ref{teoHarui}, $\Lambda(\Aut(C))$ is either $\Z/m\Z,\,\operatorname{D}_{2m}$, with $m=2$ or $4,\,$ $\operatorname{A}_4,\,\operatorname{A}_5,\,$ or $\operatorname{S}_4$.
\begin{enumerate}[(i)]
\item First, we exclude the cases $\operatorname{A}_4,\,\operatorname{A}_5\,$ and $\operatorname{S}_4$ as $\Aut(C)$ would be a group of order divisible by $3$ that has a $\Z/2\Z=N$ as a normal subgroup, in particular, it admits an element of order $6>4$, a contradiction.

\item Secondly, if $\Lambda(\Aut(C))=\operatorname{D}_{2m}$ with $m=2$ or $4$, then there is $\Lambda(\tau)\notin\langle\Lambda(\sigma)\rangle$ such that $\Lambda(\tau)$ has order two and commutes with $\Lambda(\sigma)$. By \cite[Lemma 2.2.3-(a)]{Hugg1}, $\Lambda(\tau)=[\lambda Z:Y]$ for some $\lambda\in K^*$. Equivalently, $\tau=[X:\lambda\mu Z:\mu Y]\in\Aut(C)$ for some $\lambda,\mu\in K^*$. This holds only when $\beta_{2,0}=\beta_{2,4}$, moreover, we must have $\lambda\mu^2=\pm1$ for $\tau$ to be order $\leq4$. This reduces $\tau$ to $[X:\pm1/\mu Z:\mu Y]$ with $\mu^4=\pm1$. In particular, $\operatorname{\Aut}(C)$ contains a $\operatorname{D}_8$ generated by $\sigma$ and $[X:Z:Y]$ if and only if $\beta_{2,0}=\beta_{2,4}$. Furthermore, if $\Lambda(\Aut(C))=\operatorname{D}_{8}$, then $\Aut(C)=\operatorname{GAP}(16,11)=\Z/2\Z\times\operatorname{D}_8$ or $\operatorname{GAP}(16,13)=\Z/4\Z\circ\operatorname{D}_8$ because these are the only subgroups of order $16$, which contain a (normal) $\Z/2\Z$, a $\operatorname{D}_8$, and all elements are of order $\leq4$. See \href{https://people.maths.bris.ac.uk/~matyd/GroupNames/1/C16.html}{Groups of order $16$} for more details. In these cases, every copy of $\Z/4\Z$ inside $\Aut(C)$ (for example, $\langle\sigma\rangle$) is contained in a $\Z/2\Z\times\Z/4\Z$. So there exists an involution $\tau'\notin\langle\sigma,\tau\rangle$ that commutes with $\sigma$. An elementary calculation show that $\tau'=\operatorname{diag}(1,-1,1)$ or $\operatorname{diag}(1,1,-1)$. But non of these defines an automorphism for $C$ due to the monomial $X^6$ in the defining equation for $C$.

\item Thirdly, if $\Lambda(\Aut(C))=\Z/4\Z$, then $\Aut(C)$ is one of the following groups: $
\operatorname{Q}_8,\,
\operatorname{D}_8$ or $
\Z/2\Z\times\Z/4\Z$ since these are the only groups of order $8$ that has a normal $\Z/2\Z$ and whose elements have orders $\leq4$. As before, we can see that $\Aut(C)=\operatorname{D}_8$ if and only if $\beta_{2,0}=\beta_{2,4}$, and $\Aut(C)=\Z/2\Z\times\Z/4\Z$ only if $\Aut(C)$ contains $\operatorname{diag}(1,-1,1)$ or $\operatorname{diag}(1,1,-1)$, which is not possible for any specialization of the parameters. If $\Aut(C)$ is the quaternion group $\operatorname{Q}_8$, then it should contain an element $\sigma'$ of order $4$ such that $\sigma'^2=\sigma^2$ and $\sigma'\sigma\sigma'^{-1}=\sigma^{-1}$. This gives us $\sigma'=[X:\lambda Z:-1/\lambda Y]$ for some $\lambda\in K^*$. Furthermore, such a $\sigma'$ lies in $\Aut(C)$ only if $\lambda^4=-1$ (the monomials $Y^5Z$ and $YZ^5$), $\beta_{4,1}=\beta_{0,3}=0$ (the monomials $Y^3Z^3$ and $YZ$), and $\beta_{2,4}=-\beta_{2,0}$ (the monomials $Z^4$ and $Y^4$).
\end{enumerate}
- If $\Aut(C)$ is conjugate to one of the finite primitive subgroups of $\PGL_3(K)$, then it should be $\operatorname{Hess}_{36}$ or $\operatorname{Hess}_{72}$ because these are the only ones that contain elements of order $4$ and no elements of order $>4$.

    Following the notations in Section \ref{case613}, we may assume that $\phi^{-1}\Aut(C)\phi$ contains $\langle S,T,V\rangle\simeq\operatorname{Hess}_{36}$ for some $\phi\in\PGL_3(K)$. Since all the elements of order $2$ inside $\operatorname{Hess}_{36}$ are conjugated, then, we can force $\phi^{-1}\sigma^2\phi=V^2=[X:Z:Y]$, in particular, $$
\phi=\left(
       \begin{array}{ccc}
         0 & 1 & -1 \\
         \lambda & \gamma & \gamma \\
         \mu & \nu & \nu \\
       \end{array}
     \right)
$$
We may further impose that $\phi^{-1}\sigma\phi=V$ or $V^{-1}$ since all groups of order $4$ inside $\operatorname{Hess}_{36}$ are conjugated to $\langle V\rangle$. Thus $\phi$ reduces to $\phi_{\pm}$ below:
$$
\phi_{\pm}=\left(
       \begin{array}{ccc}
         0 & 1 & -1 \\
         a & -a(1\pm\sqrt{3})/2 & -a(1\pm\sqrt{3})/2 \\
         b & -b(1\mp\sqrt{3})/2 & -b(1\mp\sqrt{3})/2 \\
       \end{array}
     \right).$$
The monomials $X^5Y,\,X^5Z,\,X^4Y^2,\,X^4Z^2,\,X^3Y^2Z,\,X^3YZ^2,\,X^2Y^4,\,X^2Z^4,$ $X^2YZ^3$, $YZ^5,$ and $Z^5Y$ should not appear in the defining equation for $^{\phi_{\pm}}C$ as $S\in\Aut(^{\phi_{\pm}}C)$. Therefore, relative to $\phi_{\pm}$, we should have

\begin{eqnarray*}
  \beta_{0,3}&=&\pm\dfrac{(2\sqrt{3}\mp3)b^4-(2\sqrt{3}\pm3)a^4}{3a^2b^2},\,\beta_{2,0}=\dfrac{- a\left(9(12\pm7\sqrt{3})a^4\pm\sqrt{3}b^4\right)}{8b^3},\\
  \beta_{2,4}&=&\dfrac{-b\left(9(12\mp7\sqrt{3})b^4\mp\sqrt{3}a^4\right)}{8a^3},\,\beta_{2,2}=\dfrac{\mp15\left((\sqrt{3}\mp2)b^4-(\sqrt{3}\pm2)a^4\right)}{4ab},\\
\beta_{4,1}&=&\pm\frac{3}{2}((2\sqrt{3}\pm3)a^4-(2\sqrt{3}\mp3)b^4)
\end{eqnarray*}
That is, $^{\phi_{\pm}}C$ is defined by
\begin{eqnarray*}
  ^{\phi_{+}}C: X^6+Y^6+Z^6&+&f_4(a,b)XYZ(X^3+Y^3+Z^3)+3f_4(a,b)X^2Y^2Z^2\\
  &-&2(f_4(a,b)+5)(X^3Y^3+X^3Z^3+Y^3Z^3)=0
   \end{eqnarray*}
with
$f_4(a,b)=\pm\dfrac{12\left((\sqrt{3}\pm2)a^4+(\sqrt{3}\mp2)b^4\right)}{b^4-a^4}$ such that
\begin{eqnarray*}
  3ab\left((\sqrt{3}\pm2)a^4-(\sqrt{3}\mp2)b^4\right)&=&8  \\
3ab\left((2\sqrt{3}\pm3)a^4+(2\sqrt{3}\mp3)b^4\right)&=&8\sqrt{3}.
\end{eqnarray*}

Lastly, $UVU^{-1}\in\Aut(^{\phi_{\pm}}C)$ if and only if $f_4(a,b)=0$ or $-6$. However, we know by Proposition \ref{3.3}-(2) that $\Aut(^{\phi_{\pm}}C)=\langle S,T,U,V\rangle\simeq\operatorname{Hess}_{216}$ when $f_4(a,b)=0$, so we discard this situation. On the other hand, when $f_4(a,b)=-6$, the curve $^{\phi_{\pm}}C$ becomes singular at the point $(1:-1:0)$. We thus conclude that $\Aut(C)$ is not conjugate to $\operatorname{Hess}_{72}$.

As a consequence of the previous discussion:
\begin{prop}\label{413}
Let $C$ be a smooth plane sextic curves $C$ of type $4, (1,3)$ as above. Then, $\Aut(C)$ is classified as follows.
\begin{enumerate}[1.]
  \item If $\beta_{2,0}=\beta_{2,4}$, then $\Aut(C)=\langle\sigma,\tau\rangle\simeq\operatorname{D}_8$, where $\sigma=\operatorname{diag}(1,\zeta_4,\zeta_4^{-1})$ and $\tau=[X:Z:Y]$, except when $\beta_{0,3}=\dfrac{1-2ct^6}{\pm ct^6},\,$ $\beta_{4,1}=\dfrac{3-16 c t^6}{ct^2},$ and $\beta_{2,0}=(\mp/2)\beta_{2,2}+\dfrac{3-16 c t^6}{2t^4}$ for some $t\in K^*$ such that $c^4=1$ or when
$\beta_{0,3}=(s'/t)^6\zeta_4+2,\,$  $\beta _{4,1}=-(2t/s')^4+3\zeta_4(s'/t)^2,$ and $\beta _{2,0}=\frac{1}{2} \left(\beta _{2,2}+3(s'/t)^4+\zeta_4(4t/s')^2\right)$ for some $s',t\in K^*$ In these cases, then $C$ is $K$-isomorphic to
$$
X^6+Y^6+Z^6+\alpha_{2,4}(X^2Y^4+Y^2Z^4+X^4Z^2+X^2Z^4+X^4Y^2+Y^4Z^2)+\alpha_{2,2}X^2Y^2Z^2=0.
$$
for some $\alpha_{2,4},\alpha_{2,2}\in K^*$ such that $\Aut(C)$ is conjugate to $\operatorname{S}_4$ generated by $\operatorname{diag}(1,-1,1),\,$ $\operatorname{diag}(1,1,-1),\,$ $[X:Z:Y]$ and $[Y:Z:X]$.

   \item If $\beta_{2,4}=-\beta_{2,0}$ and $\beta_{4,1}=\beta_{0,3}=0$, then $\Aut(C)=\langle\sigma,\sigma'\rangle\simeq\operatorname{Q}_8$, where $\sigma'=[X:\zeta_8Z:-\zeta_8^{-1}Y]$.
   \item   If $\beta_{0,3}=\pm\dfrac{(2\sqrt{3}\mp3)b^4-(2\sqrt{3}\pm3)a^4}{3a^2b^2},\,\beta_{2,0}=\dfrac{- a\left(9(12\pm7\sqrt{3})a^4\pm\sqrt{3}b^4\right)}{8b^3},$
  $\beta_{2,4}=\dfrac{-b\left(9(12\mp7\sqrt{3})b^4\mp\sqrt{3}a^4\right)}{8a^3},\,\beta_{2,2}=\dfrac{\mp15\left((\sqrt{3}\mp2)b^4-(\sqrt{3}\pm2)a^4\right)}{4ab},$
$\beta_{4,1}=\pm\frac{3}{2}((2\sqrt{3}\pm3)a^4-(2\sqrt{3}\mp3)b^4)$ such that
\begin{eqnarray*}
  3ab\left((\sqrt{3}\pm2)a^4-(\sqrt{3}\mp2)b^4\right)&=&8  \\
3ab\left((2\sqrt{3}\pm3)a^4+(2\sqrt{3}\mp3)b^4\right)&=&8\sqrt{3},
\end{eqnarray*}
then $C$ is $K$-isomorphic to
\begin{eqnarray*}
  ^{\phi_{+}}C: X^6+Y^6+Z^6&+&f_4(a,b)XYZ(X^3+Y^3+Z^3)+3f_4(a,b)X^2Y^2Z^2\\
  &-&2(f_4(a,b)+5)(X^3Y^3+X^3Z^3+Y^3Z^3)=0
   \end{eqnarray*}
with $f_4(a,b)=\pm\dfrac{12\left((\sqrt{3}\pm2)a^4+(\sqrt{3}\mp2)b^4\right)}{b^4-a^4}$. In this case, $\Aut(C)$ is conjugate to $\operatorname{Hess}_{36}=\langle S, T, V\rangle$.
  \item Otherwise, $\Aut(C)=\Z/4\Z$ generated by $\sigma$.
\end{enumerate}

\end{prop}
\subsection{Type $3,(1,2)$}\label{case312}
In this situation $C$ is defined over $K$ by a smooth plane equation of one of the following forms:
  \begin{eqnarray*}
  \mathcal{C}_1&:&X^6+Y^6+Z^6+XYZ\left(\beta_{4,1}X^3+\beta_{1,4}Y^3+\beta_{1,2}Z^3\right)+\beta_{2,2}X^2Y^2Z^2\\
  &+&\beta_{3,3}X^3Y^3+\beta_{3,0}X^3Z^3+\beta_{0,3}Y^3Z^3=0\\
  \mathcal{C}_2&:&X^5Y+Y^5Z+XZ^5+XYZ\left(\beta_{3,2}X^2Y+\beta_{1,3}Y^2Z+\beta_{2,1}XZ^2\right)\\
  &+&\beta_{2,4}X^2Y^4+\beta_{0,2}Y^2Z^4+\beta_{4,0}X^4Z^2=0,
  \end{eqnarray*}
where $\sigma=\operatorname{diag}(1,\zeta_3,\zeta_3^{-1})$ is an automorphism of maximal order $3$.

The automorphism group of $\mathcal{C}_i$ was investigated in \cite[Theorem 2.4]{BadrBarspreprint}. As a consequence, we have:
\begin{prop}\label{312}
Let $C$ be a smooth plane sextic curves $C$ of Type $3, (1,2)$ as above. Then, $\Aut(C)$ is classified as follows:
\begin{enumerate}[(1)]
  \item The automorphism group $\Aut(\mathcal{C}_1)$ is always cyclic generated by $\sigma$ except when we are in one of the situations below.
\begin{enumerate}[(i)]
\item If $\beta_{4,1}=\beta_{1,4}=\beta_{1,2}=\beta_{2,2}=0$, then $\mathcal{C}_1$ reduces to
          $$
       X^6+Y^6+Z^6+X^3\left(\beta_{3,3}Y^3+\beta_{3,0}Z^3\right)+\beta_{0,3}Y^3Z^3=0,
      $$
where $\Aut(\mathcal{C}_1)$ is $(\Z/3\Z)^2$ generated by $\operatorname{diag}(1,\zeta_3,1)$ and $\operatorname{diag}(1,1,\zeta_3)$.
\item If one of the following conditions occurs
\begin{eqnarray*}
&\textbf{(a)}& \beta_{4,1}=\pm\beta_{1,4}\,\, \text{and}\,\,\beta_{3,0}=\pm\beta_{0,3},\\
&\textbf{(b)}& \beta_{1,4}=\pm\beta_{1,2}\,\, \text{and}\,\,\beta_{3,3}=\pm\beta_{3,0},\\
&\textbf{(c)}&\beta_{4,1}=\pm\beta_{1,2}\,\, \text{and}\,\,\beta_{3,3}=\pm\beta_{0,3},
\end{eqnarray*}
then $\mathcal{C}_1$ is $K$-isomorphic to
    \begin{eqnarray*}
  \mathcal{C}'_1&:&X^6+Y^6+Z^6+\beta'_{4,1}X^4YZ+\beta'_{3,3}X^3(Y^3+Z^3)+\beta'_{2,2}X^2Y^2Z^2\\
  &+& \beta'_{1,2}XYZ(Y^3+Z^3)+\beta'_{0,3}Y^3Z^3=0,
\end{eqnarray*}
where $\Aut(\mathcal{C}'_1)$ is $\operatorname{S}_3$  generated by $\sigma$ and $[X:Z:Y]$ if $\beta'_{4,1}\neq\beta'_{1,2}$ or $\beta'_{3,3}\neq\beta'_{0,3}$, and $\Aut(C')$ is $\Z/3\Z\rtimes\operatorname{S}_3$ generated by $\sigma,\,$ $[X:Z:Y]$ and $[Y:Z:X]$ otherwise.
\begin{rem}
We note that $\left(\beta'_{3,3},\beta'_{1,2}\right)\neq(0,0)$ or $\operatorname{diag}(1,\zeta_6,\zeta_6^{-1})$ will be an automorphism of order $6>3$.
\end{rem}
\item If \textbf{(a)} $(\beta_{4,1},\beta_{1,2},\beta_{1,4}),\,$ $(\beta_{1,4},\beta_{4,1},\beta_{1,2})$ or $(\beta_{1,2},\beta_{1,4},\beta_{4,1})$ equals
    $$
    \left(\dfrac{2\left(29-54\lambda^6-54\mu^6\right)}{27\lambda\mu},\,\dfrac{2\left(27\mu^6-54\lambda^6-52\right)}{27\lambda\mu ^4},\,\dfrac{2\left(27\lambda^6-54\mu ^6-52\right)}{27\lambda^4\mu}\right),
    $$
\textbf{(b)} $(\beta_{3,0},\beta_{3,3},\beta_{0,3}),\,$ $(\beta_{3,3},\beta_{0,3},\beta_{3,0})$ or $(\beta_{0,3},\beta_{3,0},\beta_{3,3})$ equals
$$
\left(\dfrac{2\left(81\lambda ^6-27\mu^6-26\right)}{27\mu ^3},\,\dfrac{2\left(81\mu^6-27\lambda^6-26\right)}{27\lambda^3},\,\dfrac{2\left(82-27\lambda^6-27\mu ^6\right)}{27\lambda^3\mu^3}\right),
$$
and \textbf{(c)} $\beta_{2,2}=\dfrac{9\lambda^6+9\mu^6+10}{3\lambda^2\mu^2}$ for some $\lambda,\mu\in K^*$, then $\mathcal{C}_1$ is $K$-isomorphic to
\begin{eqnarray*}
  \mathcal{C}_{1,\lambda,\mu}: X^6+Y^6+Z^6&+&f_1(\lambda,\mu)X^2Y^2Z^2+f_2(\lambda,\mu)(X^4Y^2+X^2Z^4+Y^4Z^2)\\
&+&f_2(\mu,\lambda)(X^4Z^2+X^2Y^4+Y^2Z^4)=0,
  \end{eqnarray*}
where
\begin{eqnarray*}
f_1(\lambda,\mu)&:=&3(80+81\lambda^6+81\mu^6),\\
f_2(\lambda,\mu)&:=&81\left(1+\zeta_3\lambda^6+\zeta_3^{-1}\mu^6\right).
\end{eqnarray*}
In this case, $\Aut(\mathcal{C}_{1,\lambda,\mu})$ is $\operatorname{A}_{4}$ generated by $\operatorname{diag}(1,-1,1)$ and $\operatorname{diag}(1,1,-1)$ and $[Y:Z:X]$.
\end{enumerate}

\item The automorphism group $\Aut(\mathcal{C}_2)$ is always cyclic generated by $\sigma$ except when we are in one of the situations below.
      \begin{enumerate}[(i)]
\item If $\beta_{0,2}=\zeta_{21}^{-12r}\beta_{4,0},\,$ $\beta_{2,4}=\zeta_{21}^{3r}\beta_{4,0},\,$ $\beta_{1,3}=\zeta_{21}^{-6r}\beta_{3,2},\,$ $\beta_{2,1}=\zeta_{21}^{3r}\beta_{3,2},$ then $\mathcal{C}_2$ is $K$-isomorphic to
    \begin{eqnarray*}
  \mathcal{C}'_2&:&X^5Y+Y^5Z+XZ^5+\beta_{4,0}\zeta_{21}^{4r}\left(X^4Z^2+X^2Y^4+Y^2Z^4\right)\\
  &+& \beta_{3,2}\zeta_{21}^{-r}XYZ\left(X^2Y+XZ^2+Y^2Z\right)=0,
\end{eqnarray*}
where $\Aut(\mathcal{C}'_2)$ is $(\Z/3\Z)^2$ generated by $\sigma$ and $[Y:Z:X]$.
\begin{rem}
We note that $\left(\beta_{2,4},\beta_{1,3}\right)\neq(0,0)$ or $\operatorname{diag}(1,\zeta_{21},\zeta_{21}^{-4})$ will be an automorphism of order $21>3$.
\end{rem}
\item If \textbf{(a)} $(\beta_{2,4},\beta_{4,0},\beta_{0,2}),\,$ $(\beta_{0,2},\beta_{2,4},\beta_{4,0})$ or $(\beta_{4,0},\beta_{0,2},\beta_{2,4})$ equals
    $$
    \left(\dfrac{\lambda^5\mu+4\mu^5}{2\lambda^4},\,\dfrac{\lambda+4\lambda^5\mu}{2\mu^2},\,\dfrac{4\lambda+\mu^5}{2\lambda^2\mu^4}\right)
    $$
and \textbf{(b)} $(\beta_{1,3},\beta_{3,2},\beta_{2,1}),\,$ $(\beta_{2,1},\beta_{1,3},\beta_{3,2})$ or $(\beta_{3,2},\beta_{2,1},\beta_{1,3})$ equals
$$
\left(\dfrac{2\left(2\lambda^5\mu+2\lambda+\mu^5\right)}{\lambda^3\mu^2},\,\dfrac{2\lambda^5\mu+4\lambda+4\mu^5}{\lambda^2\mu},\,\frac{2\left(2\lambda^5\mu+\lambda+2\mu^5\right)}{\lambda\mu^3}\right),
$$
then $\mathcal{C}_2$ is $K$-isomorphic to
\begin{eqnarray*}
  \mathcal{C}_{2,\lambda,\mu}: X^6+Y^6+Z^6&+&g_1(\lambda,\mu)(\zeta_3^{-1}X^4Y^2+X^2Z^4+Y^4Z^2)\\
  &+&g_2(\lambda,\mu)(X^4Z^2+\zeta_3X^2Y^4+Y^2Z^4)=0,
  \end{eqnarray*}
where
\begin{eqnarray*}
g_1(\lambda,\mu)&:=&\dfrac{\sqrt{3}\zeta_9\left(\zeta_4\lambda^5\mu+\zeta_{12}\lambda+\zeta_{12}^5\mu^5\right)}{\lambda^5\mu+\lambda+\mu^5},\\
g_2(\lambda,\mu)&:=&\dfrac{\sqrt{3}\zeta_{18}\left(\zeta_{12}^5\lambda^5\mu+\zeta_{12}\lambda+\zeta_4\mu^5\right)}{\lambda^5\mu+\lambda+\mu^5}.
\end{eqnarray*}
In this case, $\Aut(\mathcal{C}_{2,\lambda,\mu})$ is $\operatorname{A}_{4}$ generated by $\operatorname{diag}(1,-1,1)$ and $\operatorname{diag}(1,1,-1)$ and $[\zeta_6^{-1}Y:Z:X]$.
\end{enumerate}

\end{enumerate}
\end{prop}
%

\begin{rem}
  The irreducibility of the stratum $\mathcal{M}^{\operatorname{Pl}}_{10}(\Z/3\Z)$, in particular, the existence of the two components $\widetilde{\mathcal{M}^{\operatorname{Pl}}_{10}}(\rho_1(\Z/3\Z))$ and $\widetilde{\mathcal{M}^{\operatorname{Pl}}_{10}}(\rho_2(\Z/3\Z))$ were first studied by Badr-Bars in \cite{MR3554328}.
\end{rem}


\subsection{Type $2,(0,1)$}\label{case301}
In this $C$ is defined by an equation of the form:
$$C:Z^6+Z^4L_{2,Z}+Z^2L_{4,Z}+L_{6,Z}=0$$
where $\sigma:=\operatorname{diag}(1,1,-1)$ is an automorphism of maximal order $2$. We remark that $L_{2,Z}\neq0$ or $L_{4,Z}\neq0$ because $\operatorname{diag}(1,1,\zeta_4)$ will be an automorphism of order $4>2$ otherwise. Also, by smoothness, $L_{6,Z}$ is of degree $\geq5$ in both $X$ and $Y$.

The full automorphism group of such a $C$ was investigated in \cite{BadrBarspreprint}. As a consequence of \cite[Theorem 2.1]{BadrBarspreprint}, we have:
\begin{prop}\label{201}
Let $C$ be a smooth plane sextic curves $C$ of Type $2, (0,1)$ as above. Then, $\Aut(C)$ is always cyclic of order $2$ unless $L_{2,Z}, L_{4,Z}$ and $L_{6,Z}$ are elements of the ring $K[X^2,Y^2]$. In this situation, $\Aut(C)$ equals $(\Z/2\Z)^2$ generated by $\sigma$ and $\operatorname{diag}(1,-1,1)$.
\end{prop}
%
\section{Appendix}

\subsection{$\operatorname{S}_3$ and $\operatorname{A}_4$ inside $\Aut(\mathcal{F}_6)$}\label{appA}
We have seen in Proposition \ref{1217} that $\Aut(\mathcal{F}_6)$ is isomorphic to $(\Z/6\Z)^2\times\operatorname{S}_3$.

- Consider the following subgroups of $\Aut(\mathcal{F}_6)$.
\begin{eqnarray*}
  \operatorname{S}_{3,1} &:=& \langle\operatorname{diag}(1,\zeta_3,\zeta_3^{-1}),\,[X:Z:Y]\rangle \\
  \operatorname{S}_{3,2} &:=& \langle[Y:Z:X],\,[X:Z:Y]\rangle
\end{eqnarray*}
Clearly each of them is isomorphic to
$$
\operatorname{S}_3=\langle a,b:a^2=b^3=(ab)^2=1\rangle
$$
with $a=[X:Z:Y]$ and $a=\operatorname{diag}(1,\zeta_3,\zeta_3^{-1}),\,[Y:Z:X]$ respectively. It is straightforward to see that  $\langle\operatorname{diag}(1,\zeta_3,\zeta_3^{-1})\rangle$ and $\langle[Y:Z:X]\rangle$ are not conjugated inside $\Aut(\mathcal{F}_6)$. Thus we conclude that any copy of $\operatorname{S}_3$ inside $\Aut(\mathcal{F}_6)$ is $\Aut(\mathcal{F}_6)$-conjugate to $\operatorname{S}_{3,1}$ or $\operatorname{S}_{3,2}$ since $\Aut(\mathcal{F}_6)$ contains exactly two conjugacy classes of $\operatorname{S}_3$s as illustrated by  \href{https://people.maths.bris.ac.uk/~matyd/GroupNames/193/C6%5E2sS3.html}{subgroups lattice of $\Aut(\mathcal{F}_6)$}.

- Similarly, any copy of $\operatorname{A}_4$ inside $\Aut(\mathcal{F}_6)$ is $\Aut(\mathcal{F}_6)$-conjugate to $\operatorname{A}_{4,1}$ or $\operatorname{A}_{4,2}$, where
\begin{eqnarray*}
  \operatorname{A}_{4,1} &:=& \langle [Y:Z:X],\,\operatorname{diag}(1,1,-1),\,\operatorname{diag}(1,-1,1)\rangle, \\
\operatorname{A}_{4,2} &:=& \langle [Y:\zeta_6Z:\zeta_6X],\,\operatorname{diag}(1,1,-1),\,\,\operatorname{diag}(1,-1,1)\rangle.
\end{eqnarray*}
Indeed, each of them is isomorphic to
$$\operatorname{A}_4=\langle a,b,c:a^2=b^2=c^3=1,\,cac^{-1}=ab=ba,cbc^{-1}=a\rangle$$
with $a=\operatorname{diag}(1,-1,1),\,b=\operatorname{diag}(1,1,-1)$ and $c=[\nu Y:Z:X]$ such that $\nu=1,\zeta_6^{-1}$ respectively. Moreover, if $\phi^{-1}\operatorname{A}_{4,2}\phi=\operatorname{A}_{4,1}$ for some $\phi\in\Aut(\mathcal{F}_6)$, then $\phi$ must be in the normalizer of $\langle\operatorname{diag}(1,1,-1),\operatorname{diag}(1,-1,1)\rangle\simeq\Z/2\Z\times\Z/2\Z$. Hence $\phi$ reduces to one of the forms:
$$
\phi_1=[aX:bY:Z],\,\phi_2=[aY:bX:Z],\,\phi_3=[aX:Z:bY],
$$
$$
\phi_4=[aY:Z:bX],\,\phi_5=[Z:aX:bY],\,\phi_6=[aZ:Y:bX],
$$
for some $a,b$ such that $a^6=b^6=1$. This yields the $24$ $\Aut(\mathcal{F}_{6})$-conjugates of $[Y:Z:X]$ namely, $\{\phi_i[Y:Z:X]\phi_i^{-1}:i=1,2,3,4,5,6\}$ More precisely, the $\Aut(\mathcal{F}_{6})$-conjugates of $[Y:Z:X]$ are given by the following sets:
$$\{[Y:\zeta_6^{\ell}Z:\zeta_6^{\ell'}X]\}_{(\ell,\ell')\in\{(0,0),(0,3),(3,0),(1,2),(2,1),(2,4),(4,2),(3,3),(4,5),(5,4),(5,1)\}}$$ $$\{[Z:\zeta_6^{\ell}X:\zeta_6^{\ell'}Y]\}_{(\ell,\ell')\in\{(0,0),(0,3),(3,0),(1,1),(1,4),(4,1),(2,2),(2,5),(5,2),(3,3),(4,4),(5,5)\}}$$
Because non of these conjugates lies in $\operatorname{A}_{4,2}$, we deduce that $\operatorname{A}_{4,1}$ and $\operatorname{A}_{4,2}$ are non-conjugated inside $\Aut(\mathcal{F}_6)$. However, it is worthy to mention that both groups are $\PGL_3(K)$ conjugated via a re-scaling of the variables in the normalizer of $\Aut(\mathcal{F}_6)$, that is, via $X\to \lambda' X,\,Y\to \mu' Y,\,Z\to Z$ with $\dfrac{\mu'^2}{\lambda'}=\dfrac{\mu'}{\lambda'^2}=\zeta_6$. 
\subsection{The Wiman's sextic curve}\label{appB}
The most symmetric smooth plane sextic curve is known to be the Wiman's sextic $\mathcal{W}_6$ defined by
$$
\mathcal{W}_6:27X^6+9XZ^5+9XY^5-135X^4YZ-45X^2Y^{2}Z^{2}+10Y^3Z^3=0.
$$
We refer the reader to \cite{DoiIdei} for more details. In particular, by \cite[Appendix A]{Yoshida1}, we know that the automorphism group of the Wiman sextic curve is $R^{-1}\,\langle T_1,\,T_2,\,T_3,\,T_4\rangle\,R$, where
\begin{eqnarray*}
  R &:=& \left(
           \begin{array}{ccc}
             a(1-b) & a(1-b) & \dfrac{1}{2}(1+2c+bc) \\
             a\sqrt{b-3} & -a\sqrt{b-3} & 0 \\
             a & a & \dfrac{-1}{2}(1+c+b)+bc \\
           \end{array}
         \right)
   \\
  T_1 &=& \operatorname{diag}(1,-1,1),
   \\
  T_2 &=& [Z:X:Y]\\
  T_3 &=& [X:\zeta_3^{-1} Z:-\zeta_3 Y]\\
  T_4 &=&\left(
            \begin{array}{ccc}
              1 & 1/b & -b \\
              1/b & b & 1 \\
              b & -1 & 1/b \\
            \end{array}
          \right) \\
\end{eqnarray*}
with $a:=\dfrac{1}{6}(1+2c+2b+bc),\,b:=\dfrac{1+\sqrt{5}}{2},\,c:=\dfrac{-1+\sqrt{-3}}{2}$.


\begin{thebibliography}{10}
\bibitem{MR1809907} T. Arakawa, \emph{Automorphism groups of compact Riemann surfaces with invariant
subsets}. Osaka J. Math. \textbf{37} (2000), no. 4, 823-846. MR 1809907.


\bibitem{MR3508302} E. Badr and F. Bars, \emph{Automorphism groups of nonsingular plane curves of degree
5}. Comm. Algebra \textbf{44} (2016), no. 10, 4327-4340. MR 3508302.

\bibitem{MR3475065} E. Badr and F. Bars, \emph{Non-singular plane curves with an element of ``large" order in its automorphism group.} Internat. J. Algebra Comput. \textbf{26} (2016), no. 2, 399-433. MR 3475065.

\bibitem{MR3554328} E. Badr and F. Bars, \emph{On the locus of smooth plane curves with a fixed automorphism group.}
Mediterr. J. Math. \textbf{13} (2016), no. 5, 3605-3627. MR 3554328.

\bibitem{twists} E. Badr, F. Bars and E. Lorenzo, \emph{On twists of smooth plane curves.} Math. Comp. \textbf{88} (2019), 421--438, DOI: https://doi.org/10.1090/mcom/3317.

\bibitem{counter} E. Badr and F. Bars, \emph{Plane model-fields of definition, fields of definition, the field of moduli of smooth plane curves.} J. Number Theory \textbf{194} (2019) 278-283, DOI:https://doi.org/10.1016/j.jnt.2018.07.010.

\bibitem{BadrBarspreprint} E. Badr and F. Bars, \emph{On fake ES-irreducibile components of certain strata of smooth plane sextics}. Preprint 2022,
https://doi.org/10.48550/arXiv.2208.08904.

\bibitem{finalstratum} E. Badr and E. Lorenzo. \emph{A note on the stratification of smooth plane curves of genus 6.} Colloq.
Math. \textbf{192}, (2020), 207-222.

\bibitem{BlSt} J. Blanc and I. Stampfli, \emph{Automorphisms of the
plane preserving a curve}. Algebr. Geom. \textbf{2} (2015), no. 2, 193-213.

\bibitem{BlaPanVus} J. Blanc, I. Pan and T. Vust, \emph{On
birational transformations of pairs in the complex plane.} Geom.
Dedicata \textbf{139} (2009), 57-73.

\bibitem{MR1796706} T. Breuer, \emph{Characters and automorphism groups of compact Riemann surfaces}.
London Mathematical Society Lecture Note Series, vol. 280, Cambridge University
Press, Cambridge, 2000. MR 1796706.

\bibitem{MR897252} E. Bujalance, J. J. Etayo, and E. Mart\'{\i}nez, \emph{Automorphism groups of hyperelliptic
Riemann surfaces}. Kodai Math. J. \textbf{10} (1987), no. 2, 174-181. MR 897252.

\bibitem{MR1223022} E. Bujalance, J. M. Gamboa, and G. Gromadzki, \emph{The full automorphism groups of
hyperelliptic Riemann surfaces}. Manuscripta Math. \textbf{79} (1993), no. 3-4, 267-282. MR 1223022.



\bibitem{MR529972} H. C. Chang, \emph{On plane algebraic curves}. Chinese J. Math. \textbf{6} (1978), no. 2, 185-
189. MR 529972.

\bibitem{MR1724156} S. Crass, \emph{Solving the sextic by iteration: a study in complex geometry and dynamics}.
Experiment. Math. \textbf{8} (1999), no. 3, 209-240. MR 1724156.

\bibitem{Deg1} A. Degtyarev, \emph{Smooth models of singular $K3$-surfaces,} Rev. Mat. Iberoam. \textbf{35} (2019), no. 1, 125-
172. MR 3914542.

\bibitem{Deg} A. Degtyarev, \emph{Tritangents to smooth sextic curves.} To
appear in Annales de l'Institut Fourier. Published as preprint in
arXiv:1909.05657 (2019).

\bibitem{DoiIdei} H. Doi, K. Idei and H. Kaneta, \emph{Uniqueness of the most symmetric non-singular plane
sextics,} Osaka J. Math. \textbf{37} (2000), 667-687.


\bibitem{Elkies} N. D. Elkies, \emph{The Klein quartic in number theory.} The eightfold way, Math. Sci.
Res. Inst. Publ., vol. 35, Cambridge Univ. Press, Cambridge, 1999, pp. 51-101. MR 1722413

\bibitem{Fuk06} S. Fukasawa, \emph{S. Fukasawa, Galois points on quartic curves in characteristic 3}. Nihonkai Math.
J. \textbf{17} (2006), no. 2, 103-110. MR 2290435

\bibitem{Fuk08} S. Fukasawa, \emph{On the number of Galois points for a plane curve in positive characteristic}.
Comm. Algebra \textbf{36} (2008), no. 1, 29-36. MR 2378364


\bibitem{Fuk09} S. Fukasawa, \emph{Galois points for a plane curve in arbitrary characteristic}. Geom. Dedicata
\textbf{139} (2009), 211-218. MR 2481846

\bibitem{Fuk13} S. Fukasawa, \emph{Complete determination of the number of Galois points for a smooth plane
curve,} Rend. Semin. Mat. Univ. Padova \textbf{129} (2013), 93-113. MR 3090633

\bibitem{Fuk14} S. Fukasawa, \emph{Automorphism groups of smooth plane curves with many Galois points}.
Nihonkai Math. J. \textbf{25} (2014), no. 1, 69-75. MR 3270973


\bibitem{Gap} GAP, The GAP Group: Groups, Algorithms, and Programming, a system for computational discrete algebra
(2008), available at http://www.gap-system.org. Version 4.4.11.


\bibitem{Harui} T. Harui, \emph{Automorphism groups of plane
curves.} Kodai Math. J. \textbf{42} (2), (2019), 308-331.

\bibitem{henn1976automorphismengruppen} P. Henn, \emph{Die Automorphismengruppen dar algebraischen Functionenkorper vom Geschlecht 3}.
Inagural-dissertation, Heidelberg, 1976.

\bibitem{Hom06} M. Homma, \emph{Galois points for a Hermitian curve}. Comm. Algebra \textbf{34} (2006),
no. 12, 4503-4511. MR 2273720

\bibitem{Hugg1} B. Huggins, \emph{Fields of moduli and fields of definition of curves}. ProQuest LLC, Ann
Arbor, MI, 2005, PhD Thesis University of California, Berkeley. MR2708514

\bibitem{Book} J. W. P. Hirschfeld, G. Korchm\'aros, F. Torres,
\emph{Algebraic Curves over Finite Fields}, Princeton Series in
Applied Mathematics, Princeton University Press, Princeton,
NJ, 2008. MR 2386879.


\bibitem{MR1510753} A. Hurwitz, \emph{\"Uber algebraische Gebilde mit
eindeutigen Transformationen in sich}. \textbf{41} (1892), no. 3, 403-442. MR 1510753.

\bibitem{klein1879b} F. Klein, \emph{Ueber die Transformationen siebenter Ordnung der elliptischen
Funktionen}. Math. Annalen \textbf{14} (1879), 428-471. Reprinted as [Klein 1923, LXXXIV, pp. 90-136]. Translated in this collection.

\bibitem{MR555703} A. Kuribayashi and K. Komiya, \emph{On Weierstrass points and automorphisms of
curves of genus three}. Algebraic geometry (Proc. Summer Meeting, Univ. Copenhagen, Copenhagen, 1978), Lecture Notes in Math., vol. 732, Springer, Berlin,
1979, pp. 253-299. MR 555703.

\bibitem{MR839811} I. Kuribayashi and A. Kuribayashi, \emph{On automorphism groups of compact Riemann
surfaces of genus 4}. Proc. Japan Acad. Ser. A Math. Sci. \textbf{62} (1986), no. 2, 65-68.
MR 839811.

\bibitem{MR1072285} I. Kuribayashi and A. Kuribayashi, \emph{Automorphism groups of compact Riemann
surfaces of genera three and four}. J. Pure Appl. Algebra \textbf{65} (1990), no. 3, 277-
292. MR 1072285.

\bibitem{MR1068416} A. Kuribayashi and H. Kimura, \emph{Automorphism groups of compact Riemann surfaces
of genus five}. J. Algebra \textbf{134} (1990), no. 1, 80-103. MR 1068416.


\bibitem{Miller} G. A. Miller, H. F. Blichfeldt, and L. E. Dickson, \emph{Theory and applications of finite
groups}, Dover Publications, Inc., New York, 1961. MR 0123600

\bibitem{MY00} K. Miura and H. Yoshihara, \emph{Field theory for function fields of plane quartic curves}.
J. Algebra \textbf{226} (2000), no. 1, 283-294. MR 1749889

\bibitem{Mit} H. Mitchell, \emph{Determination of the ordinary and modular ternary linear groups}, Trans.
Amer. Math. Soc. \textbf{12}, no. 2 (1911), 207-242.

\bibitem{MR0080730} K. Oikawa, \emph{Notes on conformal mappings of a Riemann surface onto itself}. Kodai
Math. Sem. Rep. \textbf{8} (1956), 23-30. MR 0080730.

\bibitem{MR2280308} D. Sevilla and T. Shaska, \emph{Hyperelliptic curves with reduced automorphism group
$\operatorname{A}_5$}. Appl. Algebra Engrg. Comm. Comput. \textbf{18} (2007), no. 1-2, 3-20. MR 2280308.

\bibitem{MR2035219} T. Shaska, \emph{Determining the automorphism group of a hyperelliptic curve}. Proceedings
of the 2003 International Symposium on Symbolic and Algebraic Computation,
ACM, New York, 2003, pp. 248-254. MR 2035219.

%

\bibitem{Yoshida1} Y. Yoshida, \emph{Projective plane curves whose automorphism groups are simple and primitive}. Kodai Math. J. \textbf{44}(2): 334-368 (June 2021). DOI: 10.2996/kmj44208.

\bibitem{Yoshihara} H. Yoshihara, \emph{Function field theory of plane curves by dual curves}, J. Algebra \textbf{239}, no. 1
(2001), 340-355, MR 1827887.

\end{thebibliography}
\end{document}